\RequirePackage{fix-cm}
\documentclass[twocolumn]{svjour3}

\smartqed  

\usepackage{amsmath}
\usepackage{amssymb}
\usepackage{amsthm}
\usepackage{graphicx}
\usepackage{mathrsfs}
\usepackage{hyperref}
\usepackage{color}
\usepackage{multirow}
\usepackage{array}
\usepackage{booktabs}
\usepackage{enumitem}   
\usepackage{mathabx}
\usepackage[misc,geometry]{ifsym}

\DeclareMathOperator\Span{span}
\DeclareMathOperator\spect{Spect}

\theoremstyle{definition}

\begin{document}

\title{Model reduction for constrained mechanical systems via spectral submanifolds}

\titlerunning{Model reduction for constrained mechanical systems via SSM}

\author{Mingwu Li         \and
        Shobhit Jain      \and
        George Haller
}


\institute{M. Li (\Letter)\and S. Jain\and G. Haller \at
              Institute for Mechanical Systems, ETH Z\"{u}rich \\
              Leonhardstrasse 21, 8092 Zurich, Switzerland\\
              \email{mingwli@ethz.ch}           
}

\date{Received: date / Accepted: date}

\maketitle

\begin{abstract}
	Dynamical systems are often subject to algebraic constraints in conjunction to their governing ordinary differential equations. In particular, multibody systems are commonly subject to configuration constraints that define kinematic compatibility between the motion of different bodies. A full-scale numerical simulation of such constrained problems is challenging, making reduced-order models (ROMs) of paramount importance.
    In this work, we show how to use spectral submanifolds (SSMs) to construct rigorous ROMs for mechanical systems with configuration constraints. These SSM-based ROMs  enable the direct extraction of backbone curves and forced response curves, and facilitate efficient bifurcation analysis. We demonstrate the effectiveness of this SSM-based reduction procedure on several examples of varying complexity, including nonlinear finite-element models of multi-body systems. We also provide an open-source implementation of the proposed method that also contains all details of our numerical examples.
\end{abstract}


\section{Introduction}
Constrained mechanical systems arise in a number of engineering applications. In multibody dynamics, for instance, different components of a multibody system are connected by joints that impose kinematic configuration constraints on the system~\cite{shabana2020dynamics}. In NEMS applications, the piezoelectric equations again feature algebraic contraints that couple electric potentials to mechanical displacements~\cite{lazarus2012finite}. In fluid mechanics, the preservation of mass in the incompressible, spatially-discretized Navier–Stokes equations is another example of an algebraic constraint~\cite{ascher1998computer}.  The sheer number of variables that arise from the spatial discretization of partial differential equations in such applications, e.g., the generalized displacements in flexible multibody systems, renders the full-scale simulation of such constrained dynamical systems infeasible. Model reduction enables efficient analysis of these high-dimensional systems. The goal of this paper is to present a mathematically rigorous and computationally efficient nonlinear model reduction method for constrained nonlinear systems based on the theory of spectral submanifolds.

%

\begin{sloppypar}
For model reduction of flexible multibody systems, the floating frame of reference formulation is commonly adopted, wherein the motion of a deformable body is decomposed into rigid-body motion of a floating frame and an elastic deformation relative to the floating frame~\cite{shabana2020dynamics}. Model reduction is then performed on the elastic deformation component by projecting it onto a linear subspace. Depending on the choice of that subspace, various projection-based reduction techniques have been developed such as linear normal mode (LNM) approaches~\cite{agrawal1985dynamic} and component mode synthesis (CMS) methods~\cite{cammarata2020use}. In both the LNM and CMS-based projection methods, each flexbile body has its own local reduction subspace. In systems comprising many bodies, this may result in an excessive number of variables in the reduced-order model (ROM). To overcome this issue, projection-based reduction methods based on global flexible modes have been proposed recently~\cite{cammarata2020global,cammarata2021global}.
\end{sloppypar}

In control systems applications, other non-modal reduction techniques~\cite{lehner2006use,fehr2011simulation,saak2018model} have also been employed to generate the projection subspace. For linear control systems, transfer functions characterize the relations between the control inputs and outputs in frequency domain. A common goal of reduction is then to find an appropriate projection such that the transfer function of the ROM approximates the full transfer function (see~\cite{benner2017model} for a survey).

For flexible bodies undergoing large elastic deformations, geometric nonlinearities need to be considered~\cite{shabana2020dynamics,grazioso2019geometrically}. 
Absolute nodal coordinate formulation (ANCF) is useful when the system performs both overall motions and large deformations~\cite{shabana2020dynamics}. In the ANCF, all motions are expressed in the same inertial frame and hence there is no separation between reference frame motion and elastic deformation. CMS-based reduction methods have been developed in the ANCF as well~\cite{tang2019model,tian2020model}. Here, the equations of motion in ANCF are locally linearized with respect to a set of quasi-static loading configurations and the corresponding local modes are used for model reduction.

Projection subspaces have also been obtained in data-driven settings. Specifically, the proper orthogonal decomposition (POD) has been used to perform model reduction for flexible multibody systems in ANCF~\cite{luo2017model}. As an extension of POD, a proper symplectic decomposition~\cite{peng2022data} (PSD) has been proposed recently to construct the modal bases while preserving the symplectic structure of the equations. The PSD has shown better numerical accuracy and higher computational efficiency than the POD~\cite{peng2022data}.

All these reduction methods are fundamentally linear because they project governing equations from the full phase space or configuration space onto a linear subspace. Such linear techniques generally fail to reproduce geometrically nonlinear response involving larger displacements. Indeed, no linear subspace is invariant in a generic nonlinear system and hence linear projection can only be guaranteed to work in a small neighbourhood of the equilibrium~\cite{haller2017exact}. More accurate approaches should therefore be based on ROMs on attracting, low-dimensional invariant manifolds in the full phase space of the system. 

\begin{sloppypar}
In structural dynamics, invariant-manifold-based model reduction has been explored for a few decades now~(\cite{shaw1993normal,jiang2005nonlinear,touze2006nonlinear,haller2016nonlinear,SHOBHIT,touze2021model}). Nonlinear normal modes (NNMs) \cite{shaw1993normal} were first sought as invariant manifolds tangent to linear modal subspaces at the origin, serving as nonlinear continuations of linear normal modes. When they exist, such NNMs offer a way to mathematically rigorous model reduction for nonlinear systems. The recent theory of spectral submanifolds (SSMs)~\cite{haller2016nonlinear} has indeed clarified the spectral conditions under which NNMs exist uniquely as SSMs, which are the unique smoothest invariant manifolds among all possible NNMs. Importantly, slow SSMs attract nearby full system trajectories and their internal dynamics serve as a mathematically exact ROM for the full nonlinear system. Furthermore, SSM-based ROMs enable a direct extraction of backbone curves~\cite{breunung2018explicit}, forced response curves~\cite{ponsioen2020model,part-i} and bifurcation analysis~\cite{part-ii}. A recent development is the computation of SSMs and their reduced dynamics in physical coordinates using only a minimal number of eigenvectors, which makes SSM computations scalable to realistic finite element models in structural dynamics~\cite{SHOBHIT}. An open-source implementation of the computational procedure has been available in the MATLAB-based \textsc{ssmtool} package~\cite{ssmtool21}.

All these SSM computations have targeted mechanical systems defined by systems of ordinary differential equations (ODEs). In mechanical systems with configuration constraints, however, the constraints satisfy additional algebraic equations, resulting in governing equations that are systems of differential-algebraic equations (DAEs). While these DAEs can be formulated as an equivalent system of ODEs via an appropriate choice of generalized coordinates or the elimination of the associated Lagrange multipliers~\cite{laulusa2008review,bauchau2008review}, we focus here on the DAE formulation due to its compactness and simplicity. In this work, we leverage SSM theory to reduce mechanical systems with configuration constraints. Specifically, we generalize  SSM computations and SSM-based model reduction to constrained nonlinear systems described by DAEs, thereby enabling rigorous and efficient nonlinear analysis of high-dimensional constrained mechanical systems. We also provide an open-source numerical implementation of this approach which constitutes an extension of \textsc{ssmtool}~\cite{ssmtool21}, a MATLAB-based package for the calculation of SSMs for differential equations.
\end{sloppypar}

The rest of this paper is organized as follows. Section~\ref{sec:system-setup} details the setup of
mechanical systems with configuration constraints. In section~\ref{sec:ssm-theory}, we review SSM theory for ODEs and extend it to DAEs. In Section~\ref{sec:ssmtool}, we briefly review \textsc{ssmtool}, which is used for the computation of SSMs for constrained mechanical systems in this work. Furthermore, we discuss the computational treatment for non-polynomial nonlinearities, which frequently arise in multibody systems. Finally, in Section~\ref{sec:examples}, we present several examples to illustrate the power of the SSM-based ROMs before drawing conclusions in Section~\ref{sec:conclusion}.

\section{System setup}
\label{sec:system-setup}
We consider a periodically forced nonlinear mechanical system with configuration constraints
\begin{align}
&\boldsymbol{M}\ddot{\boldsymbol{x}} +\boldsymbol{C}\dot{\boldsymbol{x}}+\boldsymbol{K}\boldsymbol{x}+\boldsymbol{f}(\boldsymbol{x},\dot{\boldsymbol{x}})+[\boldsymbol{G}(\boldsymbol{x})]^\mathrm{T}\boldsymbol{\mu}\nonumber\\
& =\epsilon \boldsymbol{f}^{\mathrm{ext}}(\boldsymbol{x},\dot{\boldsymbol{x}},\Omega t),\quad \boldsymbol{g}(\boldsymbol{x})=\boldsymbol{0},\quad 0\leq\epsilon\ll1,\label{eq:eom-second-dae}
\end{align}
where $\boldsymbol{x}\in\mathbb{R}^n$ is a displacement vector; $\boldsymbol{M}, \boldsymbol{C},\boldsymbol{K}\in\mathbb{R}^{n\times n}$ are mass, damping and stiffness matrices; $\boldsymbol{f}(\boldsymbol{x},\dot{\boldsymbol{x}})$ is a $C^r$ smooth nonlinear function for some integer $r\ge 2$ such that
$\boldsymbol{f}(\boldsymbol{x},\dot{\boldsymbol{x}})\sim \mathcal{O}(|\boldsymbol{x}|^2,|\boldsymbol{x}||\dot{\boldsymbol{x}}|,|\dot{\boldsymbol{x}}|^2)$; $ \boldsymbol{f}^{\mathrm{ext}}(\boldsymbol{x},\dot{\boldsymbol{x}},\Omega t)$ denotes an external (parametric) harmonic excitation with forcing frequency $\Omega$ and scalar amplitude $\epsilon$ such that $\epsilon=0$ corresponds to the unforced autonomous limit of the system; $\boldsymbol{g}:\mathbb{R}^n\to\mathbb{R}^{n_\mathrm{c}}$ ($n_\mathrm{c}<n$) represents a set of $C^r$ smooth configuration constraints with $\boldsymbol{G}=\partial\boldsymbol{g}/\partial\boldsymbol{x}:\mathbb{R}^n\to\mathbb{R}^{n_\mathrm{c}\times n}$ being the Jacobian of the constraints; and $\boldsymbol{\mu}$ denotes a vector of Lagrange multipliers corresponding to the configuration constraints~\cite{shabana2020dynamics}.

Without loss of generality, we assume $\boldsymbol{g}(\boldsymbol{0})=\boldsymbol{0}$ such that the origin $(\boldsymbol{x},\dot{\boldsymbol{x}},\boldsymbol{\mu})=\boldsymbol{0}$ of the phase space is a fixed point of system~\eqref{eq:eom-second-dae} when $\epsilon=0$. We further assume that the matrix $\boldsymbol{G}$ is of full rank, i.e., the constraints $\boldsymbol{g}$ are not redundant. Consequently, the vector $\boldsymbol{\mu}$ of Lagrange multipliers is well-defined. Letting $\boldsymbol{G}_0:=\boldsymbol{G}(\boldsymbol{0})$, the configuration constraints can be rewritten as
\begin{equation}
    \boldsymbol{g}(\boldsymbol{x})=\boldsymbol{G}_0\boldsymbol{x}+\boldsymbol{g}_\mathrm{nl}(\boldsymbol{x})=\mathbf{0},
\end{equation}
where $\boldsymbol{g}_\mathrm{nl}$ is a $C^r$ smooth nonlinear function such that $\boldsymbol{g}_\mathrm{nl}(\boldsymbol{x})\sim\mathcal{O}(|\boldsymbol{x}|^2)$. Accordingly, we have
\begin{equation}
    \boldsymbol{G}=\boldsymbol{G}_0+\boldsymbol{G}_\mathrm{nl}(\boldsymbol{x}),\quad \boldsymbol{G}_\mathrm{nl}=\partial\boldsymbol{g}_\mathrm{nl}/{\partial\boldsymbol{x}}.
\end{equation}

We transform the second-order DAE system~\eqref{eq:eom-second-dae} into a first-order form as
\begin{equation}
\label{eq:full-first}
\boldsymbol{B}\dot{\boldsymbol{z}}	=\boldsymbol{A}\boldsymbol{z}+\boldsymbol{F}(\boldsymbol{z})+\epsilon\boldsymbol{F}^{\mathrm{ext}}(\boldsymbol{z},{\Omega t}),
\end{equation}
where
\begin{gather}
\boldsymbol{z}=\begin{pmatrix}\boldsymbol{x}\\\dot{\boldsymbol{x}}\\\boldsymbol{\mu}\end{pmatrix}\in\mathbb{R}^{2n+n_\mathrm{c}},\nonumber\\
\boldsymbol{A}=\begin{pmatrix}-\boldsymbol{K} 
& \boldsymbol{0} & -\boldsymbol{G}_0^\mathrm{T}\\\boldsymbol{0} & \boldsymbol{M} & \boldsymbol{0} \\ \boldsymbol{G}_0& \boldsymbol{0} &\boldsymbol{0} \end{pmatrix},\,\,
\boldsymbol{B}=\begin{pmatrix}\boldsymbol{C} 
& \boldsymbol{M} & \boldsymbol{0}\\\boldsymbol{M} & \boldsymbol{0} & \boldsymbol{0}\\ \boldsymbol{0}& \boldsymbol{0}& \boldsymbol{0}\end{pmatrix},\nonumber\\
\boldsymbol{F}(\boldsymbol{z})=\begin{pmatrix}\boldsymbol{-\boldsymbol{f}(\boldsymbol{x},\dot{\boldsymbol{x}})}-\left(\boldsymbol{G}_{\mathrm{nl}}(\boldsymbol{x})\right)^\mathrm{T}\boldsymbol{\mu}\\\boldsymbol{0}\\\boldsymbol{g}_\mathrm{nl}(\boldsymbol{x})\end{pmatrix},\nonumber\\
\boldsymbol{F}^{\mathrm{ext}}(\boldsymbol{z},\Omega t) = \begin{pmatrix}\boldsymbol{f}^{\mathrm{ext}}(\boldsymbol{x},\dot{\boldsymbol{x}},\Omega t)\\\boldsymbol{0}\\\boldsymbol{0}\end{pmatrix}.\label{eq:zABF}
\end{gather}
The linearization of system~\eqref{eq:full-first} at the origin leads to the eigenvalue problem
\begin{equation}
	\label{eq:eig-AB}
	\boldsymbol{A}\boldsymbol{v}_j=\lambda_j\boldsymbol{B}\boldsymbol{v}_j,\quad \boldsymbol{u}_j^\ast \boldsymbol{A}=\lambda_j \boldsymbol{u}_j^\ast \boldsymbol{B},
\end{equation}
for $j=1,\cdots,2n+n_c$, where $\lambda_j$ is a generalized eigenvalue and $\boldsymbol{v}_j$ and $\boldsymbol{u}_j$ are the corresponding right and left eigenvectors, respectively.

Since the $n_\mathrm{c}$ constraints are well-defined, the mechanical system~\eqref{eq:eom-second-dae} effectively has $n-n_\mathrm{c}$ degrees of freedom. This results in the matrix pair $(\boldsymbol{A},\boldsymbol{B})$ having $2(n-n_\mathrm{c})$ eigenvalues with finite magnitude. We assume that the linear matrix pencil $(\boldsymbol{A},\boldsymbol{B})$ is regular, namely, there exists $\lambda\in\mathbb{C}$ such that $\mathrm{det}(\lambda\boldsymbol{B}-\boldsymbol{A})\neq0$. Then, the remaining $3n_\mathrm{c}$ eigenvalues of the system have infinite magnitude due to the singularity of $~\boldsymbol{B}$~\cite{benner2015numerical}. We note that this regularity assumption is already satisfied when $\boldsymbol{M}$ is non-singular and the constraints $\boldsymbol{g}$ are not redundant. The eigenvectors corresponding to the infinite eigenvalues are called \emph{constraint modes}~\cite{cammarata2021global,cardona1989time}. Further details about the spectrum of the linear part of the DAE system~\eqref{eq:eom-second-dae} are given in Appendix~\ref{sec:spectrum-dae}. 

We now assume that the real parts of all finite-magnitude eigenvalues are strictly less than zero, which is the case for dissipative mechanical systems.  Hence, the fixed point of the linearized system $\boldsymbol{B}\dot{\boldsymbol{z}}=\boldsymbol{A}\boldsymbol{z}$ is asymptotically stable~\cite{du2013robust}. We sort the finite-magnitude eigenvalues in decreasing order of their real parts as
\begin{equation}
	\mathrm{Re}(\lambda_{2n-2n_\mathrm{c}})\leq\cdots\leq\mathrm{Re}(\lambda_{1})<0.
	\label{eq:eig_sort}
\end{equation}

We have listed all eigenvalues of finite magnitude here for completeness. However, as we will see, it is not necessary to calculate all eigenvalues in our SSM computations because we employ the computation procedure proposed in~\cite{SHOBHIT}. Following that procedure, the SSM and its reduced dynamics are computed in the physical coordinates using only the master modes associated with the SSM.

\begin{remark}
	\label{rem:orth_cons_modes}
The $3n_\mathrm{c}$ constraint modes of system~\eqref{eq:full-first} are orthogonal to the configuration space $(x_1,\cdots,x_n)$. Indeed, for each constraint mode, we have $\boldsymbol{x}=\boldsymbol{0}$ and $\boldsymbol{\mu}\neq0$ (see~Eq. \eqref{eq:group-d} in Appendix~\ref{sec:spectrum-dae}). As the constraint modes exhibit infinite-magnitude eigenvalues, they are irrelevant for model reduction via slow SSMs. Hence, the constraint modes are not included in the master subspace of the relevant SSM and need not be computed.
\end{remark}

\section{Spectral submanifolds for constrained mechanical systems}
\label{sec:ssm-theory}
In this section, we first review the definition of spectral submanifolds for unconstrained systems whose governing equations are in the form of ODEs. Then, we show how this definition can be extended to constrained systems whose governing equations are DAEs.

\subsection{SSM theory for unconstrained systems}
\label{sec:ssm-ode}
\begin{sloppypar}
For systems without configuration constraints, we have $n_\mathrm{c}=0$ and
\begin{gather}
\label{eq:zABF-unconstrained}
\boldsymbol{z}=\begin{pmatrix}\boldsymbol{x}\\\dot{\boldsymbol{x}}\end{pmatrix},\quad
\boldsymbol{A}=\begin{pmatrix}-\boldsymbol{K} 
& \boldsymbol{0}\\\boldsymbol{0} & \boldsymbol{M}\end{pmatrix},\nonumber\\
\boldsymbol{B}=\begin{pmatrix}\boldsymbol{C} 
& \boldsymbol{M}\\\boldsymbol{M} & \boldsymbol{0}\end{pmatrix},\quad
\boldsymbol{F}(\boldsymbol{z})=\begin{pmatrix}\boldsymbol{-\boldsymbol{f}(\boldsymbol{x},\dot{\boldsymbol{x}})}\\\boldsymbol{0}\end{pmatrix},\nonumber\\
\boldsymbol{F}^{\mathrm{ext}}(\boldsymbol{z},\Omega t) = \begin{pmatrix}\boldsymbol{f}^{\mathrm{ext}}(\boldsymbol{x},\dot{\boldsymbol{x}},\Omega t)\\\boldsymbol{0}\end{pmatrix}.
\end{gather}
We note that for a positive definite mass matrix $\boldsymbol{M}$, the matrix $\boldsymbol{B}$ is invertible. In this work, we consider the SSM constructed around a $2m$-dimensional \emph{master} spectral subspace
\begin{equation}
\label{eq:master-E}
\mathcal{E}=\Span\{\boldsymbol{v}^\mathcal{E}_1,\bar{\boldsymbol{v}}^\mathcal{E}_1,\cdots,\boldsymbol{v}^\mathcal{E}_m,\bar{\boldsymbol{v}}^\mathcal{E}_m\},
\end{equation}
which is spanned by $m$ pairs of underdamped modes corresponding to the eigenvalues $\lambda_1^{\mathcal{E}},\bar{\lambda}_1^{\mathcal{E}},\dots,\lambda_m^{\mathcal{E}},\bar{\lambda}_m^{\mathcal{E}}$. Hence, we have 
\begin{equation}
    \spect(\mathcal{E}) = \{\lambda^\mathcal{E}_1,\bar{\lambda}^\mathcal{E}_1,\cdots,\lambda^\mathcal{E}_m,\bar{\lambda}^\mathcal{E}_m\}.
\end{equation}
We further define the eigenvalues of the linearization of system~\eqref{eq:zABF-unconstrained} as  \begin{equation}
    \spect(\boldsymbol{\Lambda})=\{\lambda_1,\cdots,\lambda_{2n}\}.
\end{equation}
\end{sloppypar}

\subsubsection{Autonomous systems}
In the  $\epsilon=0$ limit of system~\eqref{eq:zABF-unconstrained}, we have the following statement for the existence and uniqueness of the autonomous SSM tangent to $\mathcal{E}$ at the origin~\cite{haller2016nonlinear}.

\begin{sloppypar}
\begin{theorem}
\label{th:autoSSM-existence-uniqueness} Under the non-resonance condition
\begin{gather}
\boldsymbol{a}\cdot\boldsymbol{\lambda}^\mathcal{E}+\boldsymbol{b}\cdot\bar{\boldsymbol{\lambda}}^\mathcal{E}\neq \lambda_k,\nonumber\\
\forall\,\,\lambda_k\in\spect(\boldsymbol{\Lambda})\setminus\spect(\mathcal{E}),\nonumber\\
\forall\,\,\boldsymbol{a},\boldsymbol{b}\in\mathbb{N}_0^m,\,\,2\leq |\boldsymbol{a}+\boldsymbol{b}|\leq\sigma(\mathcal{E}),\label{eq:nonres}
\end{gather}
where the \emph{relative spectral quotient} $\sigma(\mathcal{E})$ associated to the spectral subspace $\mathcal{E}$ is defined as
\begin{equation}
\sigma(\mathcal{E}) = \mathrm{Int}\left(\frac{\min_{\lambda\in\spect(\boldsymbol{\Lambda})\setminus\spect(\mathcal{E})}\mathrm{Re}\lambda}{\max_{\lambda\in\spect(\mathcal{E})}\mathrm{Re}\lambda}\right),
\end{equation}
the following hold for system~\eqref{eq:full-first}:
\begin{enumerate}[label=(\roman*)]
\item There exists a $C^{r}$-smooth, $2m$-dimensional SSM, $\mathcal{W}(\mathcal{E})\subset\mathbb{R}^{2n}$, which is tangent to the spectral subspace $\mathcal{E}$ at the origin
\item $\mathcal{W}(\mathcal{E})$ is unique among all class $C^{\sigma(\mathcal{E})+1}$ invariant manifolds of system~\eqref{eq:full-first} satisfying (i),
\item $\mathcal{W}(\mathcal{E})$ can be viewed as an embedding of an open set in the reduced coordinates $\boldsymbol{p}\in \mathbb{C}^{2m}$ into the phase space of system~\eqref{eq:full-first} via a map $\boldsymbol{W}(\boldsymbol{p}):\mathbb{C}^{2m}\to\mathbb{R}^{2n}$.
\item There exists a polynomial series $\boldsymbol{R}(\boldsymbol{p}):{\mathbb{C}^{2m}}\to{\mathbb{C}^{2m}}$ satisfying the invariance equation
\begin{equation}
\label{eq:invariance-auto}
\boldsymbol{B}{D}_{\boldsymbol{p}}\boldsymbol{W}(\boldsymbol{p}) \boldsymbol{R}(\boldsymbol{p})=\boldsymbol{A}\boldsymbol{W}(\boldsymbol{p})+\boldsymbol{F}( \boldsymbol{W}(\boldsymbol{p})),
\end{equation}
such that the reduced dynamics on $\mathcal{W}(\mathcal{E})$ can be expressed as
\begin{equation}
\label{eq:red-dyn-auto}
\dot{\boldsymbol{p}} = \boldsymbol{R}(\boldsymbol{p}).
\end{equation}
\end{enumerate}
\end{theorem}
\end{sloppypar}
\begin{proof}
This theorem is simply a restatement of Theorem 3 by Haller and Ponsioen~\cite{haller2016nonlinear} (given the matrix $\boldsymbol{B}$ is invertible), which is based on more abstract results by Cabr\'e et al.~\cite{cabre2003parameterization-i,cabre2003parameterization-ii,cabre2005parameterization-iii}.
\end{proof}

\begin{sloppypar}
Note that the non-resonance condition~\eqref{eq:nonres} allows for resonant eigenvalues within $\spect(\mathcal{E})$. Internally resonant mechanical systems, for instance, satisfy a (near) \emph{inner} resonance relationship of the form
\begin{equation}
	\label{eq:res-inner}
	\lambda_i^\mathcal{E}\approx\boldsymbol{l}\cdot\boldsymbol{\lambda}^\mathcal{E}+\boldsymbol{j}\cdot\bar{\boldsymbol{\lambda}}^\mathcal{E},\quad \bar{\lambda}_i^\mathcal{E}\approx\boldsymbol{j}\cdot\boldsymbol{\lambda}^\mathcal{E}+\boldsymbol{l}\cdot\bar{\boldsymbol{\lambda}}^\mathcal{E}
\end{equation}
for some $i\in\{1,\cdots,m\}$, where $\boldsymbol{l},\boldsymbol{j}\in\mathbb{N}_0^m$, $|\boldsymbol{l}+\boldsymbol{j}|:=\sum_{k=1}^m (l_k+j_k)\geq2$, and $\boldsymbol{\lambda}_\mathcal{E}=(\lambda^\mathcal{E}_1,\cdots,\lambda^\mathcal{E}_m)$. Thus, SSM theory can be used to treat internally resonant mechanical systems by choosing the spectral subspace $\mathcal{E}$ which contains all such resonant eigenvalues~\cite{part-i}. 

As an example of the aforementioned inner resonances, we consider an internally resonant system such that the master subspace $\mathcal{E}$ has two pairs of modes that exhibit near 1:2 inner resonances, i.e., $\lambda_2^\mathcal{E}\approx2\lambda_1^\mathcal{E}$ and $\bar{\lambda}_2^\mathcal{E}\approx2\bar{\lambda}_1^\mathcal{E}$. Then we have
\begin{gather}
	\lambda_1^\mathcal{E}\approx(j+1)\lambda_1^\mathcal{E}+j\bar{\lambda}_1^\mathcal{E}+l\lambda_2^\mathcal{E}+l\bar{\lambda}_2^\mathcal{E},\nonumber\\
	\lambda_2^\mathcal{E}\approx(j+1)\lambda_2^\mathcal{E}+j\bar{\lambda}_2^\mathcal{E}+l\lambda_1^\mathcal{E}+l\bar{\lambda}_1^\mathcal{E},\nonumber\\
	\lambda_2^\mathcal{E}\approx(j+2)\lambda_1^\mathcal{E}+j\bar{\lambda}_1^\mathcal{E}+l\lambda_2^\mathcal{E}+l\bar{\lambda}_2^\mathcal{E}
\end{gather}
for all $j,l\in\mathbb{N}_0$.
\end{sloppypar}

\subsubsection{Nonautonomous systems}
Under small-amplitude periodic forcing, i.e., $0<\epsilon\ll1$, the SSMs persist as periodic whiskers as per the following statement ~\cite{haller2016nonlinear}.
\begin{theorem}
\label{th:SSM-existence-uniqueness}
Assume the non-resonance condition
\begin{gather}
\boldsymbol{a}\cdot\mathrm{Re}(\boldsymbol{\lambda}^\mathcal{E})+\boldsymbol{b}\cdot\mathrm{Re}(\bar{\boldsymbol{\lambda}}^\mathcal{E})\neq \mathrm{Re}(\lambda_k),\nonumber\\
\forall\,\,\lambda_k\in\spect(\boldsymbol{\Lambda})\setminus\spect(\mathcal{E})\nonumber\\
\forall\,\,\boldsymbol{a},\boldsymbol{b}\in\mathbb{N}_0^m,\,\,2\leq |\boldsymbol{a}+\boldsymbol{b}|\leq\Sigma(\mathcal{E}),\label{eq:nonresnonaut}
\end{gather}
where the \emph{absolute spectral quotient} $\Sigma(\mathcal{E})$ associated to the spectral subspace $\mathcal{E}$ is defined as
\begin{equation}
\Sigma(\mathcal{E}) = \mathrm{Int}\left(\frac{\min_{\lambda\in\spect(\boldsymbol{\Lambda})}\mathrm{Re}\lambda}{\max_{\lambda\in\spect(\mathcal{E})}\mathrm{Re}\lambda}\right).
\end{equation}
Then the following hold for system~\eqref{eq:full-first} for $\epsilon>0$, small enough:
\begin{enumerate}[label=(\roman*)]
\item There exists a $2m$-dimensional, time-periodic, class $C^r$-SSM, $\mathcal{W}(\mathcal{E},\Omega t)$, that depends smoothly on $\epsilon$,
\item {The SSM $\mathcal{W}(\mathcal{E},\Omega t)$ is unique among all $C^{\Sigma(\mathcal{E})+1}$ invariant manifolds satisfying (i)}
\item $\mathcal{W}(\mathcal{E},\Omega t)$ can be viewed as an embedding of an open set in the reduced coordinates $(\boldsymbol{p},\phi)$ into the phase space of system~\eqref{eq:full-first} via the map
\begin{equation}
\label{eq:map-nonauto}
\boldsymbol{W}_{\epsilon}(\boldsymbol{p},\phi):{\mathbb{C}^{2m}}\times{S}^1\to\mathbb{R}^{2n}\quad .
\end{equation}
\item There exists a polynomial function $\boldsymbol{R}_\epsilon(\boldsymbol{p},\phi):{\mathbb{C}^{2m}}\times{S}^1\to\mathbb{C}^{2m}$ satisfying the invariance equation
\begin{align}
 & \boldsymbol{B}\left[{D}_{\boldsymbol{p}}\boldsymbol{W}_{\epsilon}(\boldsymbol{p},\phi) \boldsymbol{R}_{\epsilon}(\boldsymbol{p},\phi)+{D}_{\phi}\boldsymbol{W}_{\epsilon}(\boldsymbol{p},\phi) \Omega\right]\nonumber\\
& =\boldsymbol{A}\boldsymbol{W}_{\epsilon}(\boldsymbol{p},\phi)+\boldsymbol{F}( \boldsymbol{W}_{\epsilon}(\boldsymbol{p},\phi))\nonumber\\
& \quad+\epsilon\boldsymbol{F}^{\mathrm{ext}}(\boldsymbol{W}_{\epsilon}(\boldsymbol{p},\phi),{\phi}),\label{eq:invariance}
\end{align}
such that the reduced dynamics on the SSM, $\mathcal{W}(\mathcal{E},\Omega t)$, can be expressed as
\begin{equation}
\label{eq:red-dyn}
\dot{\boldsymbol{p}} = \boldsymbol{R}_\epsilon(\boldsymbol{p},\phi),\quad \dot{\phi}=\Omega.
\end{equation}
\end{enumerate}
\end{theorem}
\begin{proof}
This theorem is simply a restatement of Theorem 4 by Haller and Ponsioen~\cite{haller2016nonlinear} (given the matrix $\boldsymbol{B}$ is invertible), which is based on more abstract results by Haro and de la Llave~\cite{haro2006parameterization,haro2006parameterization-num}.
\end{proof}

Note that the non-resonance condition~\eqref{eq:nonresnonaut} is independent of the frequency of periodic forcing $\Omega$. We also allow for the external forcing frequency $\Omega$ to be nearly resonant with the master eigenvalues as (see~\cite{part-i}):
\begin{equation}
	\label{eq:res-forcing}
	\boldsymbol{\lambda}^{\mathcal{E}}-\mathrm{i}\boldsymbol{r}\Omega\approx0, \,\,\boldsymbol{r}\in\mathbb{Q}^m.
\end{equation}
To illustrate the external resonance above, we again consider the system whose master subspace $\mathcal{E}$ has two pairs of modes that exhibit near 1:2 inner resonances, i.e., $\lambda_2^\mathcal{E}\approx2\lambda_1^\mathcal{E}$ and $\bar{\lambda}_2^\mathcal{E}\approx2\bar{\lambda}_1^\mathcal{E}$. Now, if the external forcing frequency $\Omega$ is nearly resonant with the first pair of modes, i.e., $\lambda_1^\mathcal{E}\approx\mathrm{i}\Omega,~\lambda_2^\mathcal{E}\approx \mathrm{i}2\Omega$, we have $\boldsymbol{r}=(1,2)$ in \eqref{eq:res-forcing} In contrast, if the external forcing resonates with the second pair of modes, i.e., $\lambda_1^\mathcal{E} \approx \frac{1}{2}\mathrm{i}\Omega, ~\lambda_2^\mathcal{E}\approx \mathrm{i}\Omega$, we have $\boldsymbol{r}=(1/2,1)$.

\subsection{Extension to constrained systems}
\label{sec:ext_cons}
For constrained systems, we have $0<n_\mathrm{c}<n$ and the matrix $\boldsymbol{B}$ in~\eqref{eq:zABF} is not invertible. However, the DAEs~\eqref{eq:eom-second-dae} can be converted into an equivalent system of ODEs via an index reduction technique, as we detail below. In principle, the existence and uniqueness results for SSMs given in Theorems~\ref{th:autoSSM-existence-uniqueness} and \ref{th:SSM-existence-uniqueness} are then applicable to constrained systems as well. In practice, however, this conversion is not required during computations as we show later.

\begin{sloppypar}
\begin{theorem}
\label{th:dae2ode}
The DAE system~\eqref{eq:eom-second-dae} is equivalent to the following system of ODEs
\begin{equation}
\label{eq:eom-ode-maggi}
    \mathbf{B}(\mathbf{z})\dot{\mathbf{z}}=\mathbf{A}(\mathbf{z})\mathbf{z}+\mathbf{F}(\mathbf{z})+\epsilon\mathbf{F}^{\mathrm{ext}}(\mathbf{z},\Omega t)
\end{equation}
with
\begin{gather}
\mathbf{z}=\begin{pmatrix}{\boldsymbol{x}}\\{\boldsymbol{e}}\end{pmatrix},\quad
\mathbf{A}(\mathbf{z})=\begin{pmatrix}
    -\alpha\widecheck{\boldsymbol{\Gamma}}(\boldsymbol{x})\boldsymbol{G}_0 & \boldsymbol{\Gamma}(\boldsymbol{x})\\
    -[\boldsymbol{\Gamma}(\boldsymbol{x})]^\mathrm{T}\boldsymbol{K} & \boldsymbol{0}
    \end{pmatrix},\nonumber\\
    \scalebox{0.8}{%
    $\mathbf{B}(\mathbf{z})=\begin{pmatrix}\boldsymbol{I} &\boldsymbol{0}\\ [\boldsymbol{\Gamma}(\boldsymbol{x})]^\mathrm{T}\boldsymbol{C}-\alpha[\boldsymbol{\Gamma}(\boldsymbol{x})]^\mathrm{T}\boldsymbol{M}\widecheck{\boldsymbol{\Gamma}}(\boldsymbol{x})\boldsymbol{G}(\boldsymbol{x}) & \boldsymbol{\Gamma}^\mathrm{T}\boldsymbol{M}\boldsymbol{\Gamma}
    \end{pmatrix}$},\nonumber\\
    \scalebox{0.8}{%
    $\mathbf{F}(\boldsymbol{z})=\begin{pmatrix}
    -\alpha\widecheck{\boldsymbol{\Gamma}}(\boldsymbol{x})\boldsymbol{g}_{\mathrm{nl}}(\boldsymbol{x})\\ -[\boldsymbol{\Gamma}(\boldsymbol{x})]^\mathrm{T}\mathbf{f}(\mathbf{z})-\boldsymbol{\Gamma}^\mathrm{T}\boldsymbol{M}\dot{\boldsymbol{\Gamma}}\boldsymbol{e}+\alpha[\boldsymbol{\Gamma}(\boldsymbol{x})]^\mathrm{T}\boldsymbol{M}\dot{\widecheck{\boldsymbol{\Gamma}}}(\boldsymbol{x})\boldsymbol{g}(\boldsymbol{x})
    \end{pmatrix}$},\nonumber\\
    \mathbf{F}^{\mathrm{ext}}(\mathbf{z},\Omega t) =\begin{pmatrix}
    \boldsymbol{0}\\\boldsymbol{\Gamma}^\mathrm{T}(\boldsymbol{x})\boldsymbol{f}^{\mathrm{ext}}(\boldsymbol{x},\dot{\boldsymbol{x}},\Omega t)
\end{pmatrix},
\end{gather}
where $\boldsymbol{e}\in\mathbb{R}^{n-n_\mathrm{c}}$ is a set of generalized speeds; $\alpha\in\mathbb{R}^+$ is a user-defined stabilization parameter in differentiating the configuration constraints; $\boldsymbol{\Gamma}\in\mathbb{R}^{n\times (n-n_\mathrm{c})}$ and $\widecheck{\boldsymbol{\Gamma}}\in\mathbb{R}^{n_\mathrm{c}\times n}$ are appropriately defined full rank matrices (see Appendix~\ref{sec:appendix-dae2ode}); and $\mathbf{f}(\mathbf{z})=\boldsymbol{f}(\boldsymbol{x},\boldsymbol{\Gamma}(\boldsymbol{x})\boldsymbol{e}-\alpha\widecheck{\boldsymbol{\Gamma}}(\boldsymbol{x})\boldsymbol{g}(\boldsymbol{x}))$.
\end{theorem}
\begin{proof}
We present the proof of this theorem in Appendix~\ref{sec:appendix-dae2ode}.
\end{proof}
\end{sloppypar}

\begin{remark}
The theorem above provides an ODE reformulation~\eqref{eq:eom-ode-maggi} of the original DAE system~\eqref{eq:eom-second-dae} or equivalently~\eqref{eq:full-first}. The phase space of the original DAE system is $2n+n_\mathrm{c}$ dimensional because $\boldsymbol{z}=(\boldsymbol{x},\dot{\boldsymbol{x}},\boldsymbol{\mu})$, while the phase space dimension of the reformulated system is $2n-n_\mathrm{c}$ because $\mathbf{z}=(\boldsymbol{x},\boldsymbol{e}$). Here the generalized speeds $\boldsymbol{e}$ are also called kinematic characteristics or independent quasi-velocities~\cite{laulusa2008review}.
\end{remark}

\begin{sloppypar}
The above reformulation guarantees the invertibility of $\mathbf{B}(\boldsymbol{z})$ as long as $\boldsymbol{M}$ is invertible~\cite{Bernstein+2009}. Furthermore, as will follow from Theorem~\ref{thm:spectrum_equivalence}, the origin is a stable fixed point of system~\eqref{eq:eom-ode-maggi}. This allows us to apply SSM theory to the reformulated system~\eqref{eq:eom-ode-maggi}. Indeed, Theorems~\ref{th:autoSSM-existence-uniqueness} and~\ref{th:SSM-existence-uniqueness} stated in section~\ref{sec:ssm-ode} can be generalized to the system of ODEs~\eqref{eq:eom-ode-maggi} by extending the constant matrices $\boldsymbol{A}$ and $\boldsymbol{B}$ to $\mathbf{A}(\boldsymbol{z})$ and $\mathbf{B}(\boldsymbol{z})$. 
While useful for establishing theoretical equivalence, this conversion is not desirable in practice due to the following reasons:
\begin{enumerate}
	\item The ODEs in system~\eqref{eq:eom-ode-maggi} are implicit due to the state-dependence in the coefficient matrix $\mathbf{B(z)}$, which adds complexity to the original invariance equations~\eqref{eq:invariance-auto}, \eqref{eq:invariance} that need to be solved for SSM computation.
	 
	\item The construction of the matrices $\boldsymbol{\Gamma}(\boldsymbol{x})$ and $\widecheck{\boldsymbol{\Gamma}}(\boldsymbol{x})$ is not unique and can be computationally challenging (see Appendix~\ref{sec:appendix-dae2ode}).
	
	\item The external forcing terms in the ODE system~\eqref{eq:eom-ode-maggi} are generally state-dependent even if the external forcing in the original DAE system~\eqref{eq:full-first} is independent of the state. Again, this adds computational complexity in the solution of invariance equations.
\end{enumerate}
\end{sloppypar}

To overcome these issues, we propose to compute SSMs by solving the invariance equations~\eqref{eq:invariance-auto}, \eqref{eq:invariance} directly for the DAE system~\eqref{eq:full-first}. To apply SSM theory directly on the DAE system~\eqref{eq:full-first}, however, it is necessary to relate the spectrum of the equivalent ODE system~\eqref{eq:eom-ode-maggi} to that of the original DAE system~\eqref{eq:full-first}. The following statement provides this relationship.
\begin{theorem}
	\label{thm:spectrum_equivalence}
The spectrum of the linear part of the ODE system~\eqref{eq:eom-ode-maggi} can be divided into two groups. The first group has $2(n-n_\mathrm{c})$ eigenvalues that are exactly equal to the $2(n-n_\mathrm{c})$ eigenvalues with finite mangnitude of the linear part of the DAE system~\eqref{eq:eom-second-dae} or equivalently~\eqref{eq:full-first}. Furthermore, there exists a one-to-one correspondence between the eigenvectors in the first group of the ODE system~\eqref{eq:eom-ode-maggi} and that of the DAE system~\eqref{eq:eom-second-dae}. The second group has $n_\mathrm{c}$ eigenvalues that are all equal to $-\alpha$.
\end{theorem}

\begin{proof}
We provide the spectrum of the linear part of the ODE system~\eqref{eq:eom-ode-maggi} in Appendix~\ref{sec:spectrum-ode}. The two groups of eigenvalues and their eigenvectors are listed in~\eqref{eq:group-a} and~\eqref{eq:group-b}. The spectrum of the linear part of the DAE system~\eqref{eq:eom-second-dae} or equivalently~\eqref{eq:full-first} are listed in Appendix~\ref{sec:spectrum-dae}. By comparing~\eqref{eq:group-a} and~\eqref{eq:group-c}, we conclude the equivalence of the eigenvalues of the first group and the finite eigenvalues of the DAEs~\eqref{eq:eom-second-dae}. The one-to-one correspondence follows from the orthogonality of the constraint modes to the configuration space (see Remark~\ref{rem:orth_cons_modes}). The explicit expressions for relationship between the two sets of eigenvectors are given in Eqs.~\eqref{eq:group-a} and~\eqref{eq:group-c}.
\end{proof}

Theorem~\ref{thm:spectrum_equivalence} establishes a one-to-one correspondence of the $2(n-n_c)$ finite-magnitude eigenvalues and eigenvectors of the DAE system~\eqref{eq:full-first} to those of the equivalent ODE system~\eqref{eq:eom-ode-maggi}. Since the real parts of all finite-magnitude eigenvalues of system~\eqref{eq:full-first} are assumed to be negative, it follows from Theorem~\ref{thm:spectrum_equivalence} that the origin is a stable fixed point of system~\eqref{eq:eom-ode-maggi}. 
Furthermore, as the stabilization parameter $\alpha\in\mathbb{R}^+$ can be chosen arbitrarily, the remaining $n_c$ eigenvalues of the ODE system~\eqref{eq:eom-ode-maggi}, which are all equal to $-\alpha$, can be made to satisfy the non-resonance conditions in Theorems~\ref{th:autoSSM-existence-uniqueness} and \ref{th:SSM-existence-uniqueness}. 

As such, the eigenvectors corresponding to these $n_c$ spurious eigenvalues are an artifact of the index reduction of DAEs~\eqref{eq:eom-second-dae} and are not relevant for model reduction via SSMs. Thus, in analogy with Remark~\ref{rem:orth_cons_modes}, these modes are not included in the master subspace of the relevant SSM. We note that for the purpose of verification of non-resonance conditions given in Theorems~\ref{th:autoSSM-existence-uniqueness} and \ref{th:SSM-existence-uniqueness} for the DAE~system \eqref{eq:full-first}, only the $2(n-n_c)$ eigenvalues outside the constraint modes are relevant (cf. Remark~\ref{rem:orth_cons_modes}). Hence, without any loss of generality, for the DAE system, we define the spectrum as $\spect(\boldsymbol{\Lambda})=\{\lambda_1,\cdots,\lambda_{2(n-n_\mathrm{c})}\}$. 

To summarize, Theorems~\ref{th:dae2ode} and~\eqref{thm:spectrum_equivalence} together provide us a constructive tool to apply SSM theory and computation to systems with configuration constraints.

\section{Computation of SSMs and their reduced dynamics}
\label{sec:ssmtool}

\subsection{SSMTool}
The DAE system~\eqref{eq:full-first} is of the form of the equation of motion for which SSMTool~\cite{SHOBHIT,part-i,part-ii,ssmtool21} computes SSMs and their reduced dynamics. Next we briefly review the main features of SSMTool. 

\begin{sloppypar}
SSMTool performs SSM computations in physical coordinates using only the master modes associated with the SSMs~\cite{SHOBHIT}. This makes SSM computations applicable even to high-dimensional finite-element models. The algorithm supports the computation of parameterized SSMs along with their reduced dynamics up to an arbitrary order of accuracy in an automated fashion. Specifically, for autonomous systems ($\epsilon=0)$, it returns the Taylor expansions of the SSM parametrization $\boldsymbol{W}(\boldsymbol{p})$ and its reduced dynamics $\boldsymbol{R}(\boldsymbol{p})$, defined in Theorem~\ref{th:autoSSM-existence-uniqueness}, by solving the invariance equation~\eqref{eq:invariance-auto}.

For non-autonomous systems with periodic forcing ($\epsilon>0)$, SSMTool returns time-periodic SSM parametrization~\eqref{eq:map-nonauto} and its reduced dynamics~\eqref{eq:red-dyn} by solving the invariance equation~\eqref{eq:invariance}. Specifically, it computes the terms in the expansions
\begin{align}
\boldsymbol{z}&=\boldsymbol{W}_\epsilon(\boldsymbol{p},\phi)=\boldsymbol{W}(\boldsymbol{p})+\epsilon\boldsymbol{X}_{{0}}(\phi)+\mathcal{O}(\epsilon|\boldsymbol{p}|),\label{eq:ssm-nonauto}\\
\dot{\boldsymbol{p}}&=\boldsymbol{R}_{\epsilon}(\boldsymbol{p},\phi)=\boldsymbol{R}(\boldsymbol{p})+\epsilon \boldsymbol{S}_{{0}}(\phi)+\mathcal{O}(\epsilon|\boldsymbol{p}|), \label{eq:red-nonauto}
\end{align}
where $\boldsymbol{X}_{{0}}$ is the leading-order, non-autonomous part of the SSM and $\boldsymbol{S}_{{0}}$ is the leading-order periodic forcing term in the reduced dynamics. We use a normal-form-style parametrization to compute the SSM and its reduced dynamics. This allows us to factor out the $\phi$-dependent terms in~\eqref{eq:red-nonauto} after appropriate coordinate transformations, and simplifies the forced response curve (FRC) computation (see~\cite{part-i} for more details). Specifically, the periodic response of the full nonlinear system is simply obtained by computing the fixed points of the corresponding SSM-based ROM~\cite{SHOBHIT,part-i}. Similarly, quasiperiodic tori in the full system can be directly obtained by computing periodic orbits in the SSM-based ROM~\cite{part-ii}.
\end{sloppypar}

Furthermore, these SSM-based ROMs enable efficient bifurcation analysis of high-dimensional systems. For instance, one can predict the existence of limit cycles and quasi-periodic orbits of full systems via the Hopf bifurcation of fixed points and limit cycles of the reduced dynamics on the SSMs. To this end, SSMTool has been integrated with the numerical continuation package \textsc{coco}~\cite{dankowicz2013recipes,COCO,ahsan2022methods} to facilitate nonlinear analysis of SSM-based ROMs (see~\cite{part-i,part-ii} for the details).

\subsection{Treatment of non-polynomial nonlinearities}
\label{sec:nonpoly-to-poly}
In constrained mechanical systems, nonlinearities with trigonometric terms commonly arise. The current implementation of SSMTool, however, supports only polynomial nonlinearities. Nonetheless, by introducing auxiliary variables and constraints, systems with trigonometric terms can be recast into DAEs with polynomial terms~\cite{guillot2019generic}. A table of recasts of the most common transcendental functions is given in the Appendix of~\cite{guillot2019generic}. This conversion technique can be applied to systems with common non-polynomial nonlinearities to obtain DAEs with polynomial nonlinearities.

As an illustration, we consider here the example of a forced pendulum to show how its equation of motion, which contains a sinusoidal nonlinearity, can be recast into a DAE system with polynomial nonlinearities. The equation of motion is given by
\begin{equation}
\label{eq:pend-ode}
    \ddot{\varphi}+c\dot{\varphi}+\sin\varphi=\epsilon\cos\Omega t,
\end{equation}
where $\varphi$ is the rotation angle of the pendulum and $c>0$ is a damping coefficient. We define $z_1=\varphi$, $z_2=\dot{\varphi}$ and introduce auxiliary variables $z_3=\sin\varphi$ and $z_4=\cos\varphi-1$. Then \eqref{eq:pend-ode} can be expressed as
\begin{gather}
    \dot{z}_1=z_2,\nonumber\\
    \dot{z}_2=-cz_2-z_3+\epsilon\cos\Omega t,\nonumber\\
    \dot{z}_3=(1+z_4)z_2,\nonumber\\
    z_3^2+(1+z_4)^2=1\label{eq:pend-dae},
\end{gather}
where the last two sub-equations above relate the auxiliary variables, $z_3$ and $z_4$, to the pendulum's dynamics governed by the state variable $\varphi$. Indeed, the reformulated DAE system~\eqref{eq:pend-dae} contains only polynomial terms in the variables $(z_1, z_2, z_3, z_4)$. Such a reformulation is not unique: the auxiliary variable $z_4=\cos\varphi$, suggested in~\cite{guillot2019generic}, would also work. Here, we choose $z_4=\cos\varphi-1$ instead to ensure that the origin remains a fixed point of the reformulated system.

The $z_3$-equation in system~\eqref{eq:pend-dae} results in a zero eigenvalue for the linear part of the DAE system~\eqref{eq:pend-dae}. Furthermore, the algebraic equation in the DAE system~\eqref{eq:pend-dae} introduces an eigenvalue with infinite magnitude due to the singularity of the coefficient matrix $\boldsymbol{B}$ in the reformulated setting, which is typical for mechanical systems with configuration constraints (see Section~\ref{sec:system-setup}). The two modes corresponding to the spurious zero and infinite eigenvalues are an artifact of the reformulation and should not be included in the master subspace for SSM-based reduction, as we discussed in Remark~\ref{rem:orth_cons_modes}.

\begin{sloppypar}
In summary, by introducing auxiliary variables and additional constraint equations, systems with non-polynomial nonlinearities can be converted into systems with purely polynomial nonlinearities, to which SSMTool can be directly applied to compute SSMs and their reduced dynamics. Without this conversion, one would have to approximate general nonlinerities locally via their Taylor expansions. This would render SSMs and their reduced dynamics over substantially smaller domains of the phase space.
\end{sloppypar}


\section{Examples}
\label{sec:examples}
\subsection{A spatial oscillator with a path constraint}

Consider a spatial oscillator shown in Fig.~\ref{fig:3Doscillator}. The equations of motion are given by~\cite{buza2021using}
\begin{align}
&\ddot{x}_1+2\zeta_1\omega_1\dot{x}_1+\omega_1^2x_1+\frac{\omega_1^2}{2}(3x_1^2+x_2^2 + x_3^2)\nonumber\\
&\quad +\omega_2^2x_1x_2+\omega_3^2x_1x_3\nonumber\\
&\quad +\frac{\omega_1^2+\omega_2^2 + \omega_3^2}{2}x_1(x_1^2+x_2^2+x_3^2)=\epsilon f_1\cos\Omega t,\nonumber\\
& \ddot{x}_2+2\zeta_2\omega_2\dot{x}_2+\omega_2^2x_2+\frac{\omega_2^2}{2}(3x_2^2+x_1^2 + x_3^2)\nonumber\\
&\quad +\omega_1^2x_1x_2+ \omega_3^2x_2 x_3 \nonumber\\
&\quad+ \frac{\omega_1^2+\omega_2^2 + \omega_3^2}{2}x_2(x_1^2+x_2^2+ x_3^2)=0,\nonumber\\
& \ddot{x}_3+2\zeta_3\omega_3\dot{x}_3+\omega_3^2x_3+\frac{\omega_3^2}{2}(3x_3^2+x_1^2+x_2^2)\nonumber\\
&\quad+\omega_1^2x_1x_3+\omega_2^2x_2x_3\nonumber\\
&\quad+\frac{\omega_1^2+\omega_2^2+\omega_3^2}{2}x_3(x_1^2+x_2^2+ x_3^2)=0.\label{eq:eom-3Doscillator}
\end{align}
We impose a configuration constraint $g(x_1,x_2,x_3)=0$ so that the oscillator can move only on the surface defined by the constraint. We consider the following two types of constraints:
\begin{align}
    &\text{cubic}: g(x_1,x_2,x_3)=x_3-x_1^3-x_2^3,\label{eq:cubic}\\
    &\text{spherical}: g(x_1,x_2,x_3)=x_1^2+x_2^2+(x_3-1)^2-1\label{eq:sphere}.
\end{align}
Each of these constraint functions satisfies our initial assumptions, namely, $g(0,0,0)=0$ and $G_0$ is of full rank. When either of these configuration constraints is added, the left-hand side of~\eqref{eq:eom-3Doscillator} is modified by the addition of $\mu\nabla g$ (see Eq.~\eqref{eq:eom-second-dae}).

\begin{figure}[!ht]
\centering
\includegraphics[width=.2\textwidth]{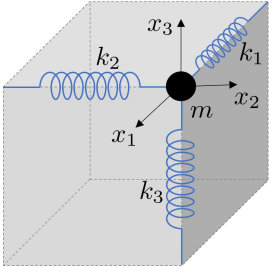}
\caption{\small An oscillator in three-dimensional space (see Eq.~\eqref{eq:eom-3Doscillator}) with possible configuration constraints imposed in Eqs.~\eqref{eq:cubic} and~\eqref{eq:sphere}.}
\label{fig:3Doscillator}
\end{figure}

In the following computations, we set $\zeta_1=0.01$, $\zeta_2=\zeta_3=0.05$, $\omega_1=2$, $\omega_2=3$, $\omega_3=5$ and $f_1=1$. The three pairs of complex conjugate eigenvalues for the linear part of the system~\eqref{eq:eom-3Doscillator} without configuration constraints are
\begin{gather}
    \lambda_{1,2}=-0.02\pm1.9999\mathrm{i},\nonumber\\ \lambda_{3,4}=-0.15\pm2.9962\mathrm{i},\nonumber\\ \lambda_{5,6}=-0.25\pm4.9937\mathrm{i}.\label{eq:3Dos-spectrum}
\end{gather}
When a configuration constraint is added, the effective number of degrees of freedom is reduced by one. Thus, based on our discussion in section~\ref{sec:ext_cons}, we obtain only two finite-magnitude natural frequencies and three eigenvalues with infinite magnitude corresponding to the constraint modes of the governing DAE system. For both configuration constraints in~\eqref{eq:cubic}-\eqref{eq:sphere}, the eigenvalues of the DAE system are given as
\begin{gather}
    \lambda_{1,2}=-0.02\pm1.9999\mathrm{i},\nonumber\\ \lambda_{3,4}=-0.15\pm2.9962\mathrm{i},\quad |\lambda_{5,6,7}|=\infty.\label{eq:3Dos-spectrum-dae}
\end{gather}

Comparing the eigenvalues~\eqref{eq:3Dos-spectrum} of the unconstrained system with the eigenvalues~\eqref{eq:3Dos-spectrum-dae} of the constrained DAE system, we note that the first two pairs of eigenvalues are common (see Theorem~\ref{thm:spectrum_equivalence}). Thus, the mode corresponding to the third pair of eigenvalues in~\eqref{eq:3Dos-spectrum} is constrained by the configuration constraint (cubic or spherical). Indeed, the linear parts of constraint equations~\eqref{eq:cubic} and \eqref{eq:sphere} depend only on the modal coordinate~$x_3$. 

Next, we calculate the SSM tangent to the spectral subspace corresponding to the first pair of eigenvalues in~\eqref{eq:3Dos-spectrum} and~\eqref{eq:3Dos-spectrum-dae}, along with the reduced dynamics on the SSM. Thus, our SSM-based ROM is two-dimensional, independently of the type of configuration constraint (cubic or spherical) and of the dimension of the full phase space.

\subsubsection{Autonomous dynamics}
\begin{sloppypar}
First, we analyze the unforced ($\epsilon=0$) limit of system~\eqref{eq:eom-3Doscillator} in the absence of any configuration constraints and then in the presence of cubic and spherical constraints shown in~\eqref{eq:cubic} and \eqref{eq:sphere}. We compute the two-dimensional, slow SSM associated to $ \lambda_{1,2} $ in \eqref{eq:3Dos-spectrum} and \eqref{eq:3Dos-spectrum-dae} using SSMTool. 



Transforming the parametrization coordinates $\boldsymbol{p}\in\mathbb{C}^2$ into polar coordinates $(\rho,\vartheta)$ such that $\boldsymbol{p}=(\rho e^{\mathrm{i}\vartheta},\rho e^{-\mathrm{i}\vartheta})$, we obtain two-dimensional ROMs up to $\mathcal{O}(13)$ for the unconstrained and constrained variants of system~\eqref{eq:eom-3Doscillator} in the following form:
\end{sloppypar}
\begin{itemize}
    \item unconstrained
    \begin{align}
        \dot{\rho} & =- 0.02\rho- 0.2387\rho^3+ 1.08\rho^5- 4.408\rho^7\nonumber\\
        &\quad+ 27.75\rho^9- 71.08\rho^{11}+50.58\rho^{13},\nonumber\\
    \dot{\vartheta} &= \omega(\rho) = 2.0- 1.206\rho^2- 0.3417\rho^4- 4.035\rho^6\nonumber\\
    &\quad - 23.49\rho^8+ 121.5\rho^{10}- 1370\rho^{12}.\label{eq:red-free}
    \end{align}
    \item constrained (cubic)
    \begin{align}
        \dot{\rho}&= -0.02\rho-0.02188{\rho}^3+0.02972{\rho_1 }^5-1.029\,{\rho}^7\nonumber\\
        &\quad+5.913{\rho}^9-27.97\,{\rho}^{11}+214.2{\rho}^{13},\nonumber\\
        \dot{\vartheta} & = \omega(\rho) =2.0+0.8168{\rho}^2-8.958{\rho}^4+3.485{\rho}^6\nonumber\\
        &\quad-66.98{\rho}^8-7.963{\rho}^{10}-882.8{\rho}^{12}.\label{eq:red-cubic}
    \end{align}
    \item constrained (spherical)
    \begin{align}
        \dot{\rho}& = -0.02\rho-0.05085{\rho}^3+0.2779{\rho}^5-1.945{\rho}^7\nonumber\\
        &\quad+5.725{\rho}^9+26.99{\rho}^{11}+1068.0{\rho}^{13} ,\nonumber\\
        \dot{\vartheta} &=\omega(\rho)= 2.0+4.421{\rho}^2-3.666{\rho}^4-88.02{\rho}^6\nonumber\\
        &\quad +1341.0{\rho}^8-12060.0{\rho}^{10}+55620.0{\rho}^{12}.\label{eq:red-sphere}
    \end{align}
\end{itemize}
The second sub-equation in the ROMs~\eqref{eq:red-free},\eqref{eq:red-cubic} and \eqref{eq:red-sphere} determines an instantaneous frequency of oscillation $\omega$ as a function of the polar amplitude $\rho$, defining the damped backbone curve in polar coordinates. Figure~\ref{fig:3Dos-rho-omega} shows these backbone curves for the three ROMs~\eqref{eq:red-free},\eqref{eq:red-cubic} and \eqref{eq:red-sphere} for increasing order of SSM approximation upto $\mathcal{O}(13)$. 

We observe that higher-order expansions are useful to obtain convergence in backbones at higher amplitudes. For instance, in the left panel of Fig.~\ref{fig:3Dos-rho-omega}, the backbone curve is well converged at $\mathcal{O}(3)$ expansion for amplitude $\rho\leq0.2$, at $\mathcal{O}(7)$ for amplitude $\rho\leq0.3$ and at $\mathcal{O}(13)$ for amplitude $\rho\leq0.35$.

\begin{figure*}[!ht]
\centering
\includegraphics[width=.32\textwidth]{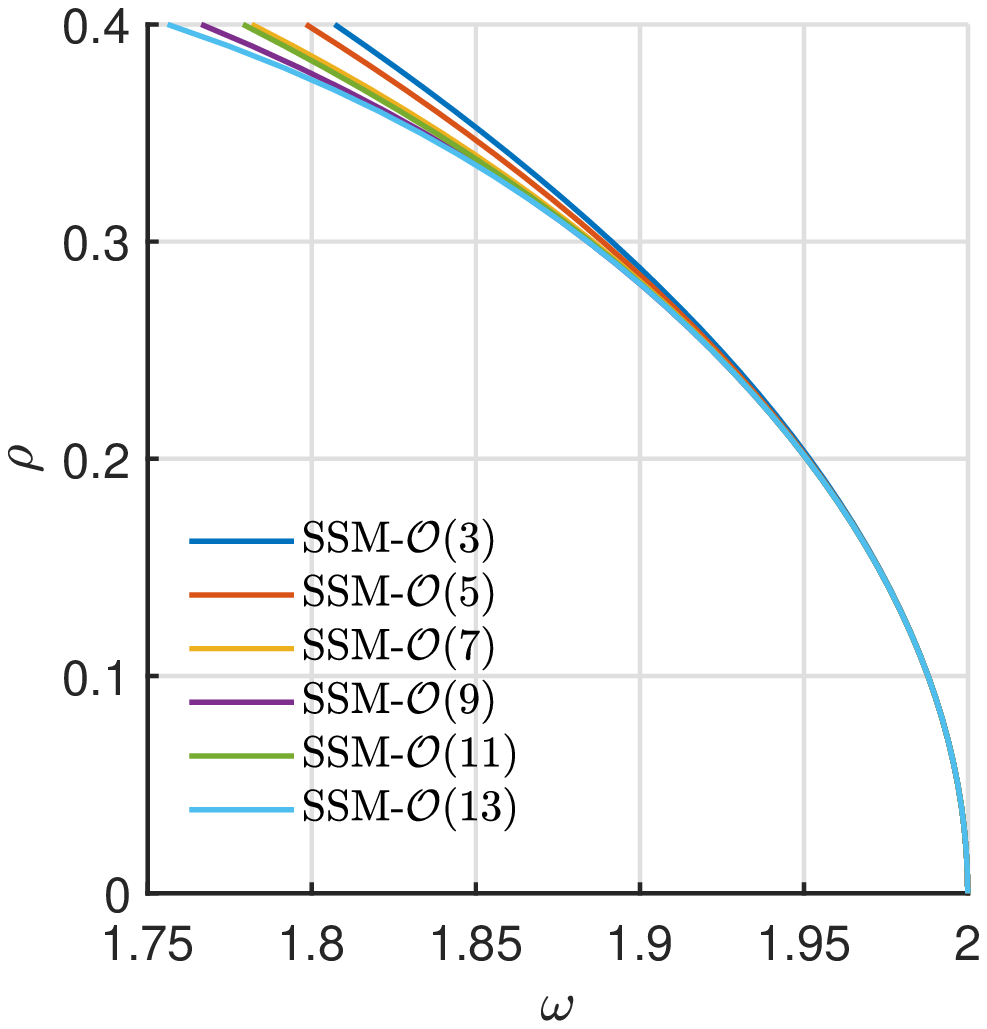}
\includegraphics[width=.32\textwidth]{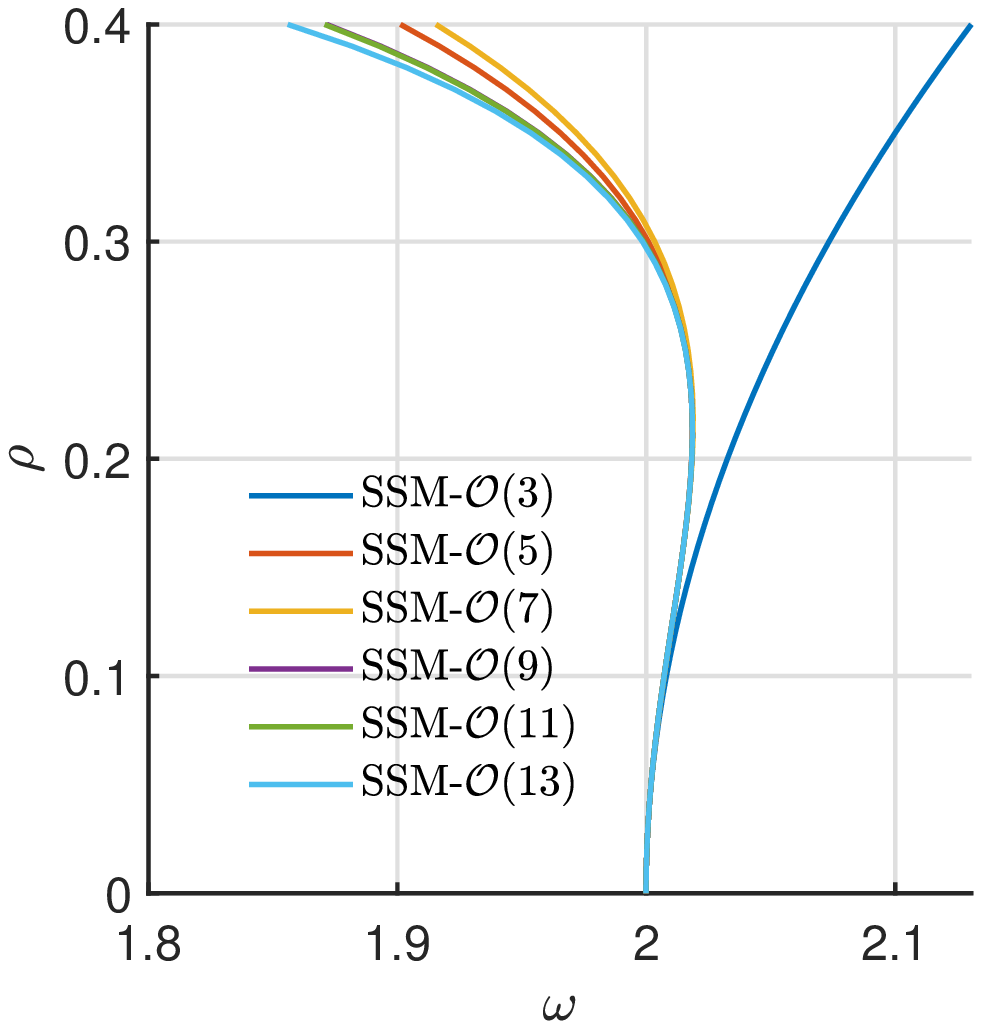}
\includegraphics[width=.32\textwidth]{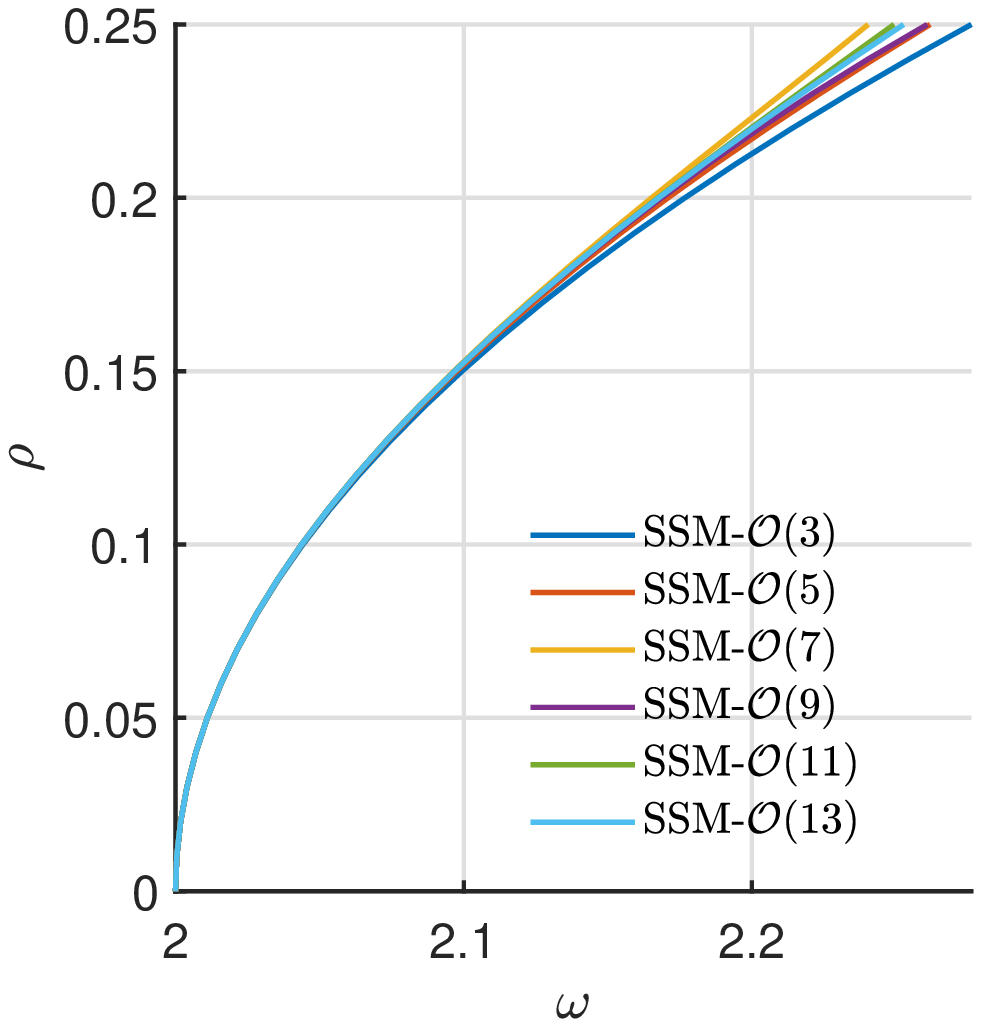}
\caption{\small Backbone curves in polar (reduced) coordinates under increasing orders of approximation for the slowest two-dimensional SSM of the oscillator system~\eqref{eq:eom-3Doscillator} in the absence of any configuration constraints (left), in the presence of the cubic configuration constraint~\eqref{eq:cubic} (middle), and in the presence of the sphere configuration constraint~\eqref{eq:sphere} (right).}
\label{fig:3Dos-rho-omega}
\end{figure*}

The presence of constraints have a remarkable effect on the nature of backbones. Indeed, as the amplitude $\rho$ increases, the backbone exhibits softening behavior for the unconstrained system (see Fig.~\ref{fig:3Dos-rho-omega}(left)), hardening followed by softening for the system with cubic configuration constraint (see Fig.~\ref{fig:3Dos-rho-omega}(middle)), and hardening behavior for the system with spherical configuration constraint (see Fig.~\ref{fig:3Dos-rho-omega}(right)).

These nonlinear hardening and softening effects on the shape of the backbone curves can also be observed in the physical coordinates of the full system. In Fig.~\ref{fig:3Dos-x1-omega}, we plot the damped backbone curves in the $(x_1,\omega)$ coordinates obtained via SSM reduction in the absence of any configuration constraints, in the presence of the cubic configuration constraint~\eqref{eq:cubic}, and in the presence of the spherical configuration constraint~\eqref{eq:sphere}. To validate the backbone curves obtained by the SSM reduction, we compute the conservative backbone curves of the original systems, to which the damped backbone curves converge in the small damping limit~\cite{breunung2018explicit}. 
\begin{figure*}[!ht]
	\centering
	\includegraphics[width=.32\textwidth]{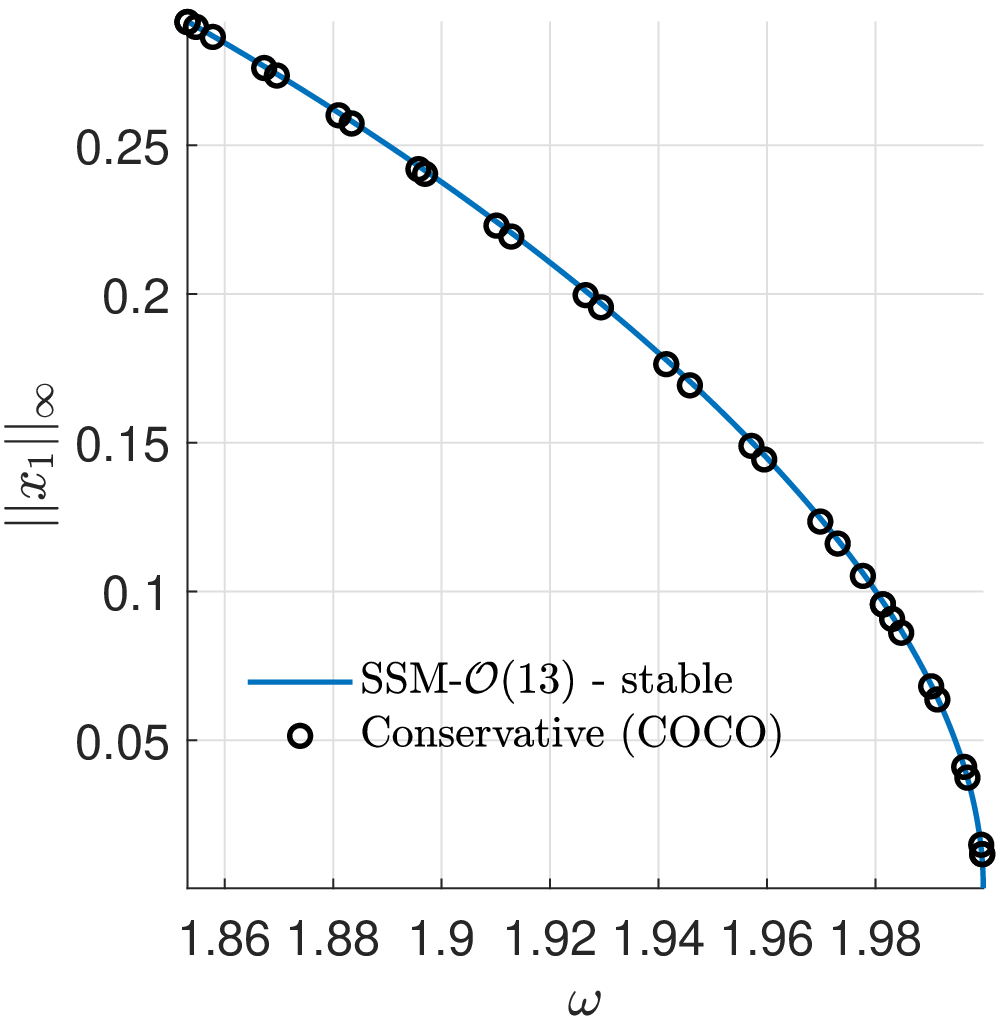}
	\includegraphics[width=.32\textwidth]{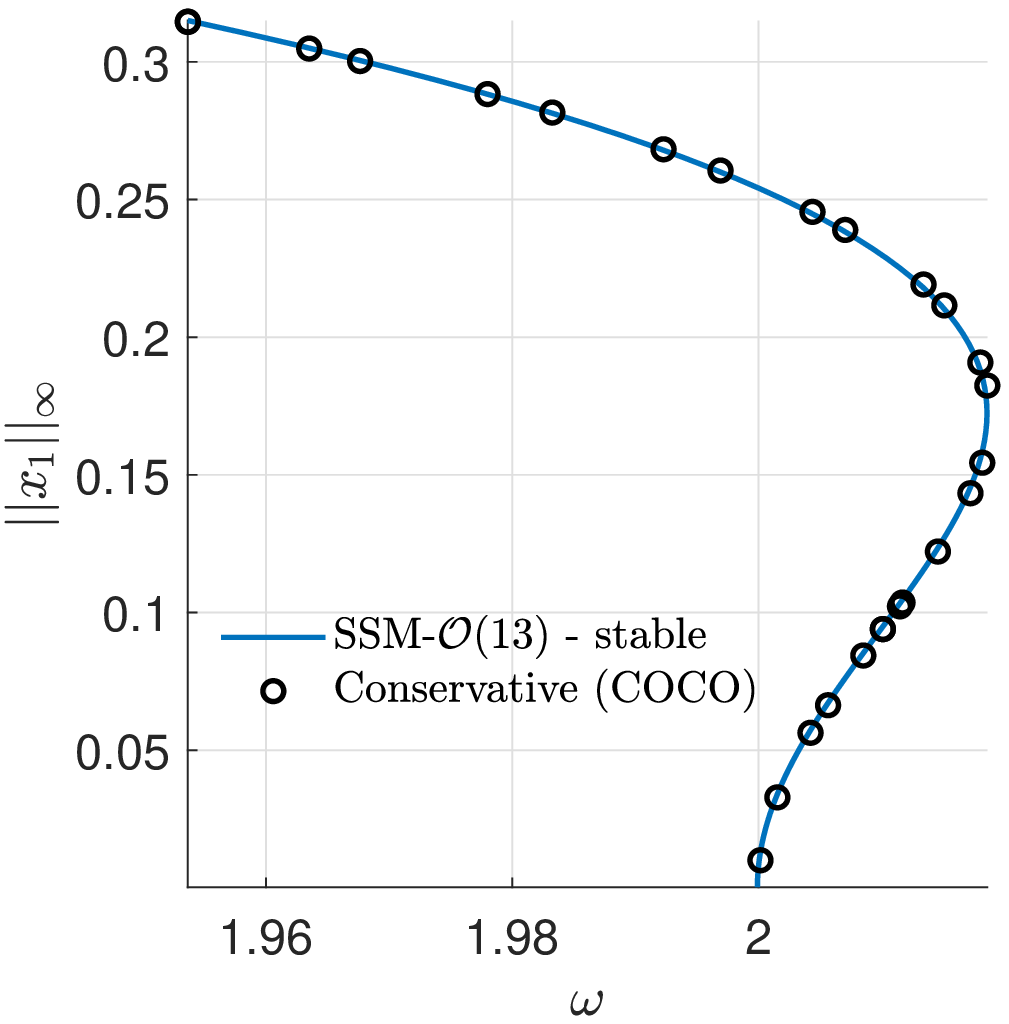}
	\includegraphics[width=.32\textwidth]{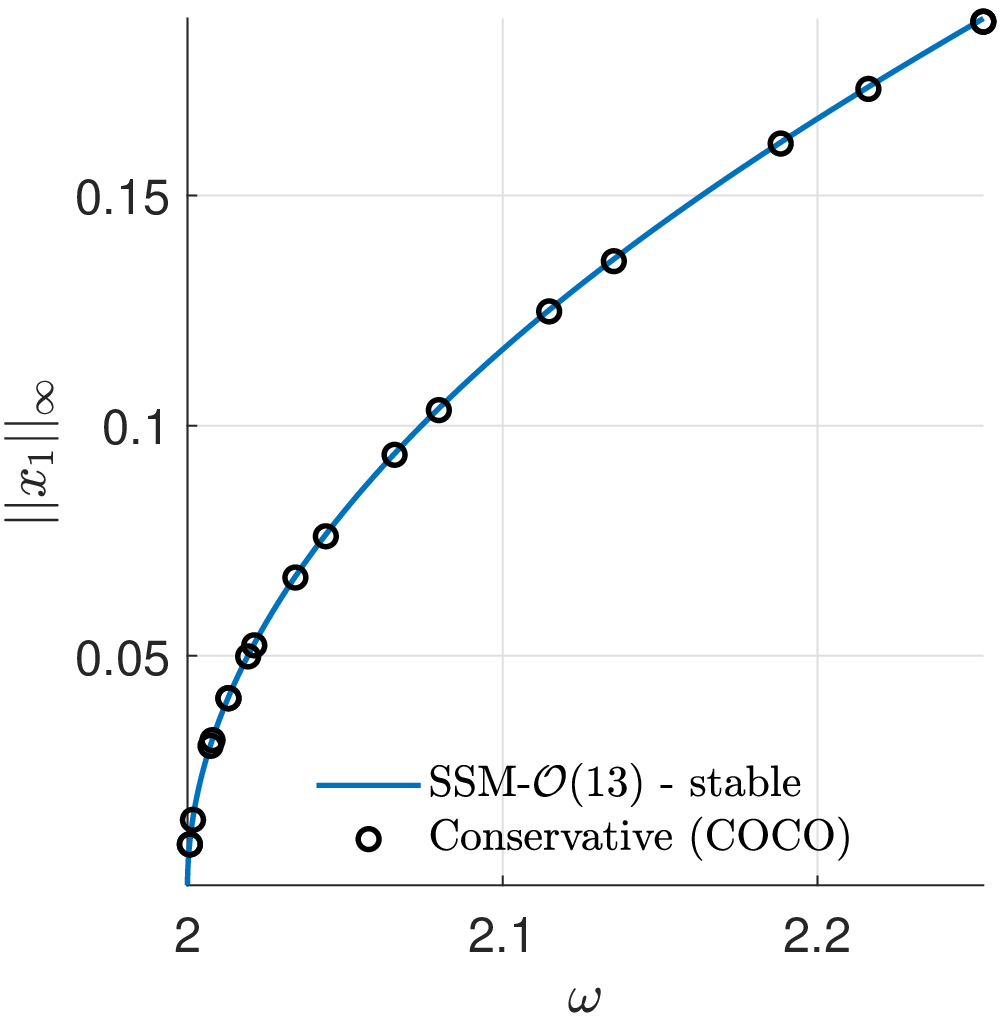}
	\caption{\small Damped and conservative backbone curves depicting the instantaneous and periodic vibration amplitude of the $x_1$ degree of freedom for the oscillator system~\eqref{eq:eom-3Doscillator} in the absence of any configuration constraints (left), in the presence of the cubic configuration constraint~\eqref{eq:cubic} (middle), and in the presence of the spherical configuration constraint~\eqref{eq:sphere} (right). The backbone curves at the conservative limits of the full systems are obtained with the \texttt{po}-toolbox of \textsc{coco}, while the damped backbones are obtained via SSM-based ROMs~\eqref{eq:red-free}-\eqref{eq:red-sphere}.}
	\label{fig:3Dos-x1-omega}
\end{figure*}

Most available tools for the nonlinear analysis of dynamical systems are restricted to ODEs. Indeed, the direct computation of periodic orbits of constrained mechanical systems in the form of DAEs is still an emerging field~\cite{han2020simulation,ju2021efficient}. To obtain a reference periodic solution for validation purposes, however, we simply convert the constrained DAE system into an equivalent ODE system (see Theorem~\ref{th:dae2ode}) and perform numerical continuation on the equivalent ODE system. In this work, instead of using the ODE system~\eqref{eq:eom-ode-maggi}, we employ an index-1 formulation with stabilization~\cite{laulusa2008review,bauchau2008review} to obtain the equivalent ODE system~\eqref{eq:eom-2nd-nolambda} as this alternative is simpler to implement. We refer to Appendix~\ref{sec:index-1} for further details about this index-1 formulation.
\begin{sloppypar}
Since the chosen damping ratios are small, we expect the damped backbone curves to match closely the conservative backbone curves. We compute the conservative backbone curve of the full system in its undamped limit by parameter continuation of periodic orbits using the \texttt{po}-toolbox of \textsc{coco}~\cite{COCO}. Indeed, the damped backbone curves obtained from SSM-based ROMs~\eqref{eq:red-free}-\eqref{eq:red-sphere} agree with the conservative backbone curves of the full systems, as shown in Fig.~\ref{fig:3Dos-x1-omega}.
\end{sloppypar}

\begin{figure*}[!ht]
\centering
\includegraphics[width=.32\textwidth]{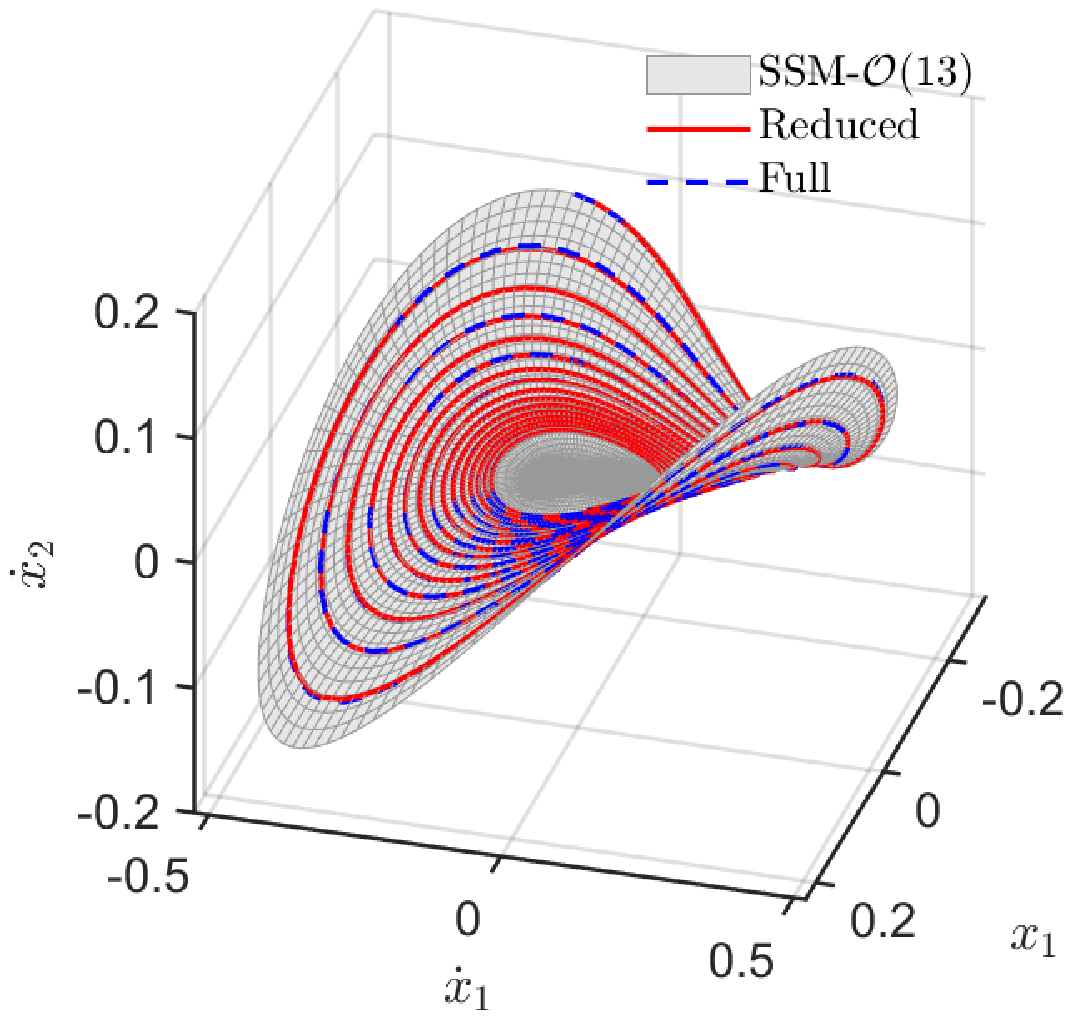}
\includegraphics[width=.32\textwidth]{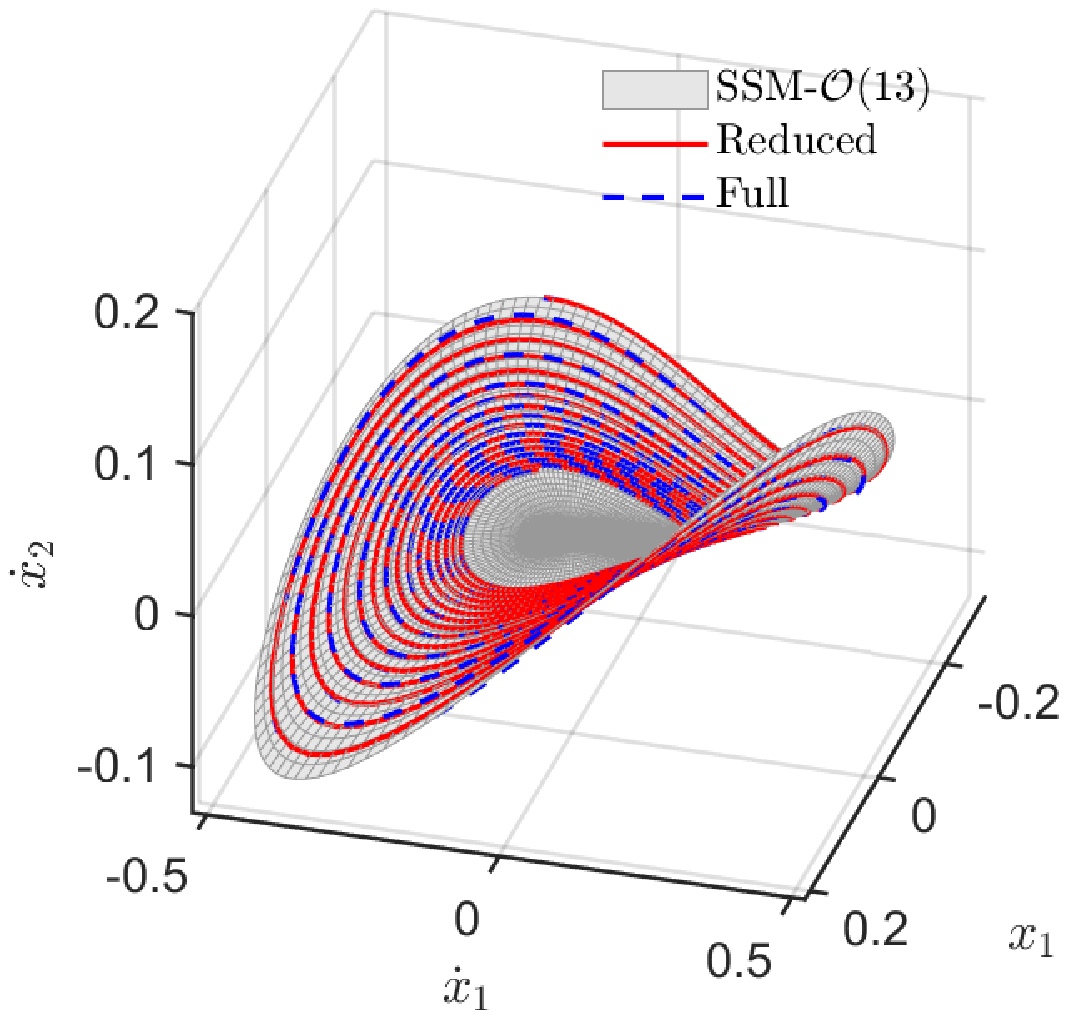}
\includegraphics[width=.32\textwidth]{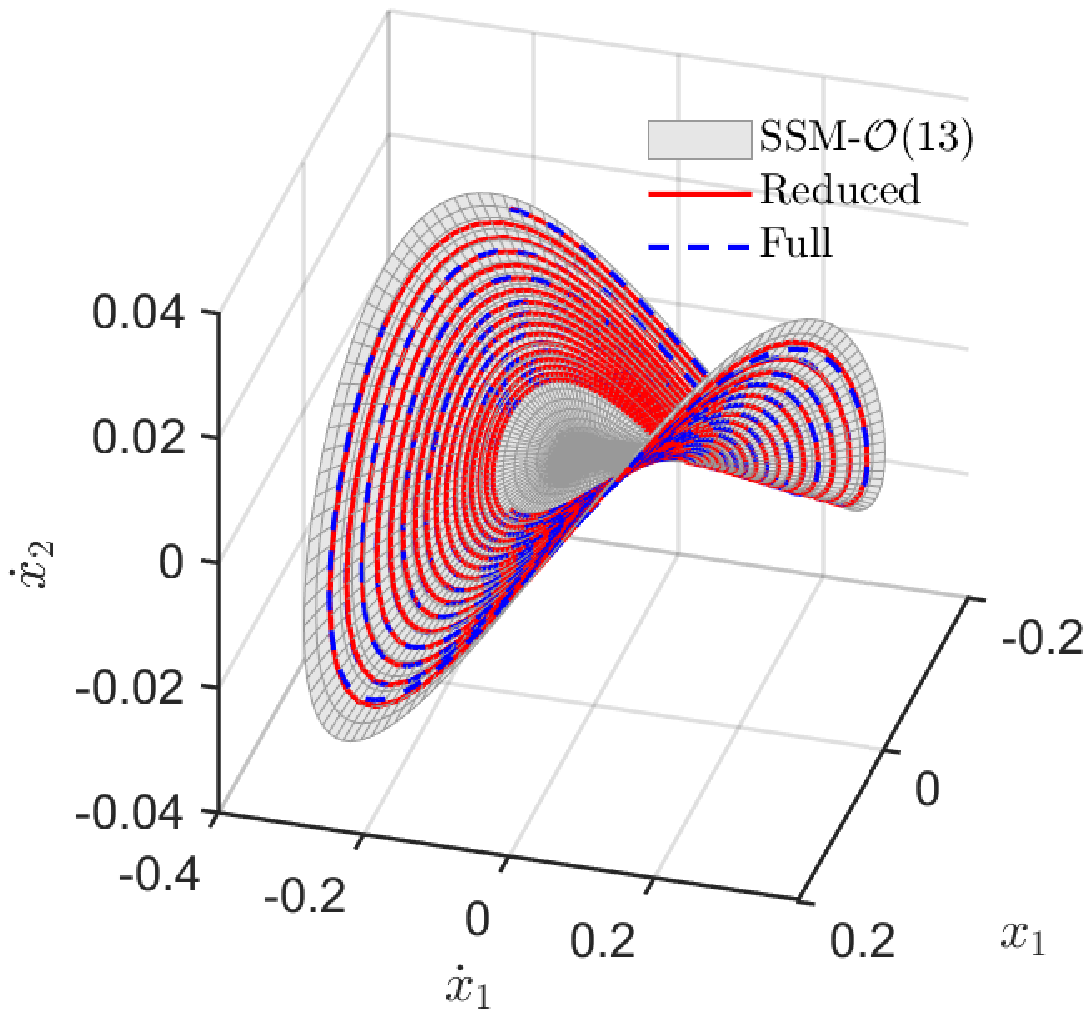}
\caption{\small Projection of SSM approximated upto $\mathcal{O}(13)$ for the oscillator system~\eqref{eq:eom-3Doscillator} in the absence of any configuration constraints (left), in the presence of the cubic configuration constraint~\eqref{eq:cubic} (middle), and in the presence of the spherical configuration constraint~\eqref{eq:sphere} (right). The red solid trajectory in each panel is obtain via simulation of SSM-based ROMs~\eqref{eq:red-free}-\eqref{eq:red-sphere}, while the dashed blue trajectory is obtained via full system simulation with the same initial condition as that for the ROM counterpart. For trajectory simulations, we choose initial conditions on the SSM as $(\rho_0,\vartheta_0)=(0.35,0.5)$ for the left and middle panels and $(\rho_0,\vartheta_0)=(0.24,0.5)$ for the right panel.}
\label{fig:3Dos-ssm}
\end{figure*}

We also provide an alternative validation of our SSM-based reduction by checking the invariance of the SSM. For the SSM computed up to a given order, we take an initial condition $\boldsymbol{p}_0=(\rho_0 e^{\mathrm{i}\vartheta_0},\rho_0e^{-\mathrm{i}\vartheta_0})$ on the SSM, and perform forward time integration both of the ROM and of the full system using the same initial condition. As the SSM is an invariant manifold, the trajectory obtained by simulating the full system must coincide with the reduced trajectory obtained by simulating the SSM-based ROM. We plot the SSM projection on to the coordinates $(x_1,\dot{x}_1,\dot{x}_2)$ for the system without configuration constraints, with the cubic configuration constraint, and with the spherical configuration constraint in Fig.~\ref{fig:3Dos-ssm}. Indeed, we observe from Fig.~\ref{fig:3Dos-ssm} that for an $\mathcal{O}(13)$ SSM computation, the simulated trajectory of the full system stays invariant on the computed SSM and matches well with the simulated trajectory of SSM-based ROM.

\subsubsection{Forced response curve}
We now introduce a small-amplitude external harmonic excitation by setting $0<\epsilon\ll1$ in the governing equations~\eqref{eq:eom-3Doscillator}. We are interested in computing the forced response curve (FRC) of the system near the first natural frequency, i.e., for $\Omega\approx\omega_1$.


\begin{figure*}[!ht]
\centering
\includegraphics[width=.45\textwidth]{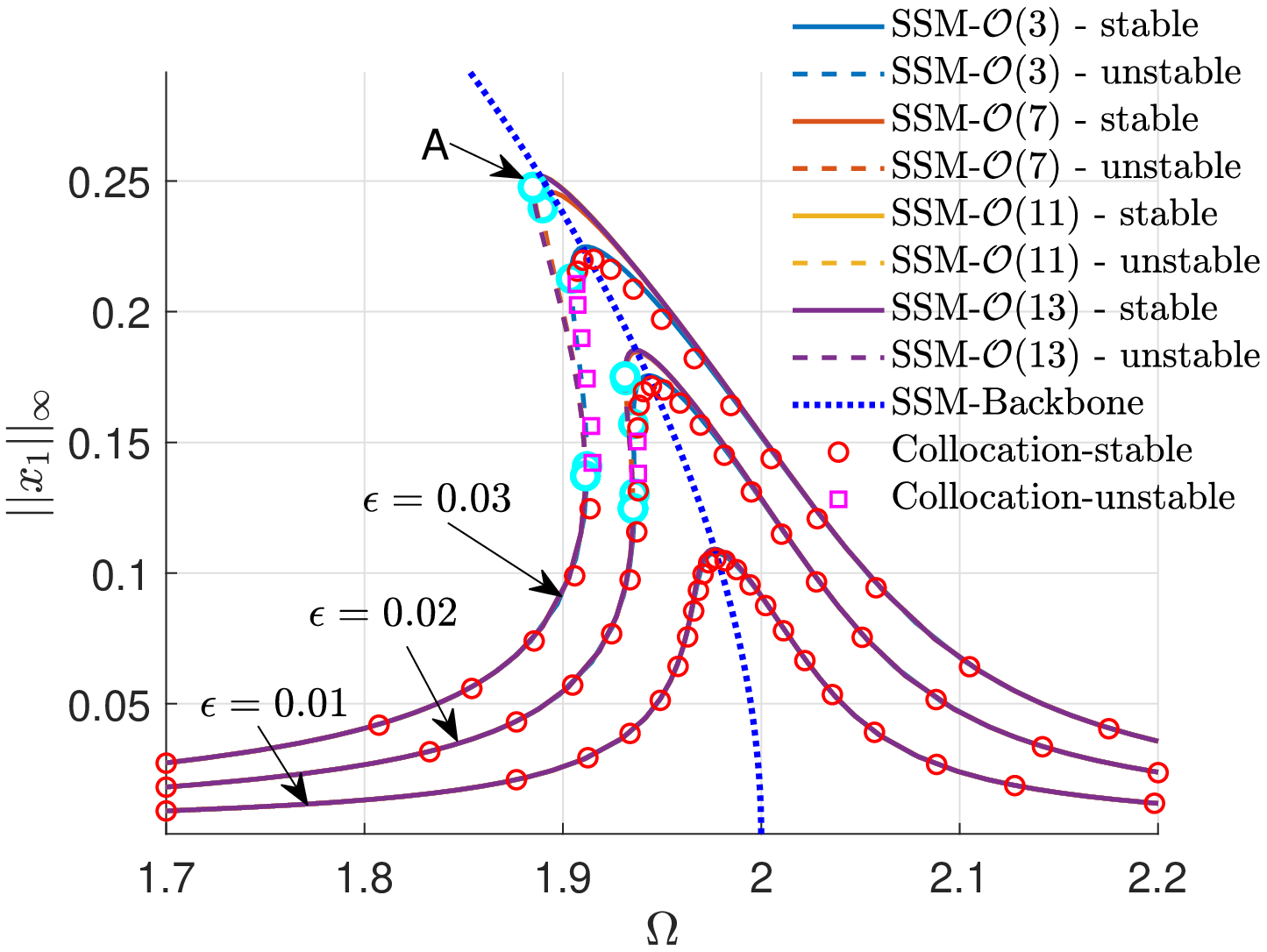}
\includegraphics[width=.45\textwidth]{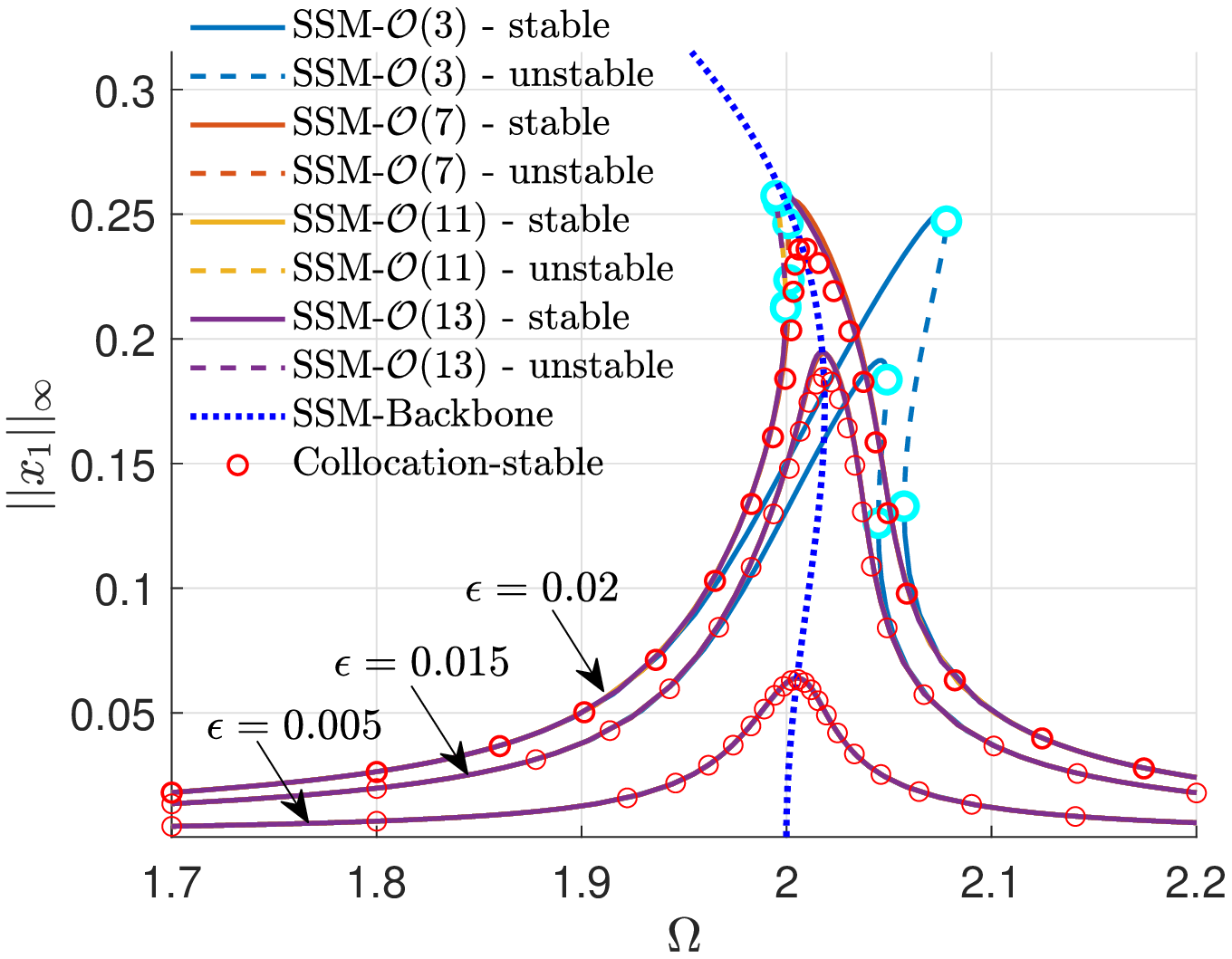}\\
\includegraphics[width=.45\textwidth]{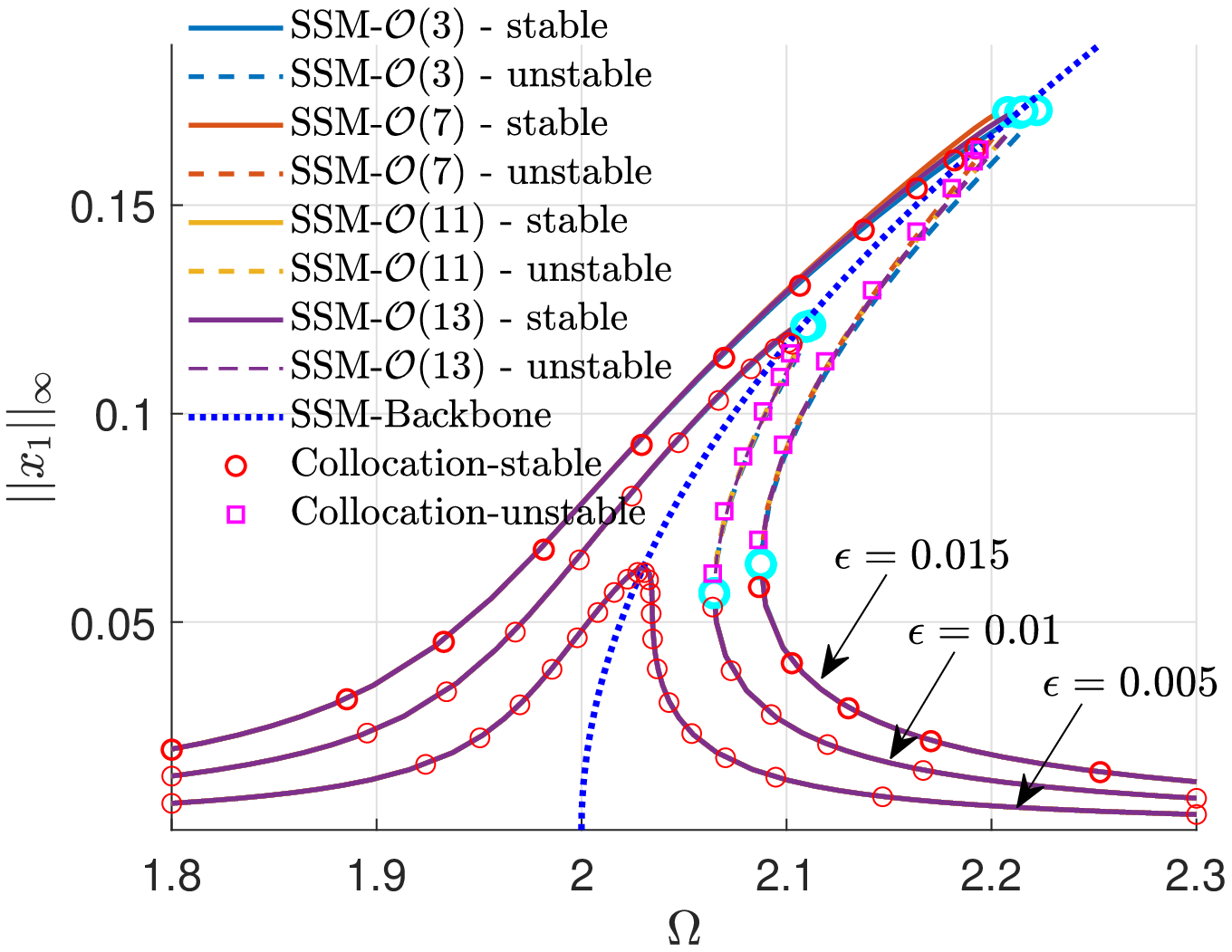}
\caption{\small Forced response curves in the vibration amplitude of $x_1$ for the oscillator in three-dimensional space without configuration constraints (upper-left panel), with the cubic configuration constraint (upper-right panel), and with the spherical configuration constraint (lower panel). The reference solutions of the full systems are obtained with collocation methods implemented in the \texttt{po}-toolbox of \textsc{coco}. The backbone curves in Fig.~\ref{fig:3Dos-x1-omega} are also included. Here, and throughout the paper, the solid lines of FRCs
indicate stable solution branches, while dashed lines of FRCs mark unstable solution branches. The cyan circles on FRCs denote saddle-node
bifurcation points.}
\label{fig:3Dos-x1-frc}
\end{figure*}

The near-resonant FRCs at various excitation amplitudes are shown in Fig.~\ref{fig:3Dos-x1-frc}. We observe that as the excitation amplitude $\epsilon$ increases, higher-order SSM expansions are required for convergence of the reduced response to the full system's response. For instance, for $\epsilon=0.01$ in the system without configuration constraints (upper-left panel of Fig.~\ref{fig:3Dos-x1-frc}), the FRC from SSM prediction already converges to the full reference solution at $\mathcal{O}(3)$ onwards. However, for $\epsilon=0.02$, the FRC from SSM prediction shows convergence to the full solution only beyond $\mathcal{O}(7)$.
In Fig.~\ref{fig:3Dos-x1-frc}, the full system's response (labelled with 'collocation') is obtained by parameter continuation of periodic orbits of~\eqref{eq:eom-3Doscillator} (and its variants in the presence of constraints) using the \texttt{po}-toolbox of \textsc{coco}~\cite{dankowicz2013recipes}.

We further observe from Fig.~\ref{fig:3Dos-x1-frc} that at the highest forcing amplitudes, the SSM-based reduced response converges to a response different from the full system's response. This may be attributed to the fact that we have neglected higher-order non-autonomous terms at $\mathcal{O}(\epsilon|\boldsymbol{p}|)$ in the expansions~\eqref{eq:ssm-nonauto}-\eqref{eq:red-nonauto}. To probe this further, we use an alternative pointwise validation technique that involves the computation of the residual of the invariance equation ~\eqref{eq:invariance} (see Fig.~\ref{fig:residual-none-003-sn}). 

\begin{figure}[!ht]
	\centering
	\includegraphics[width=.45\textwidth]{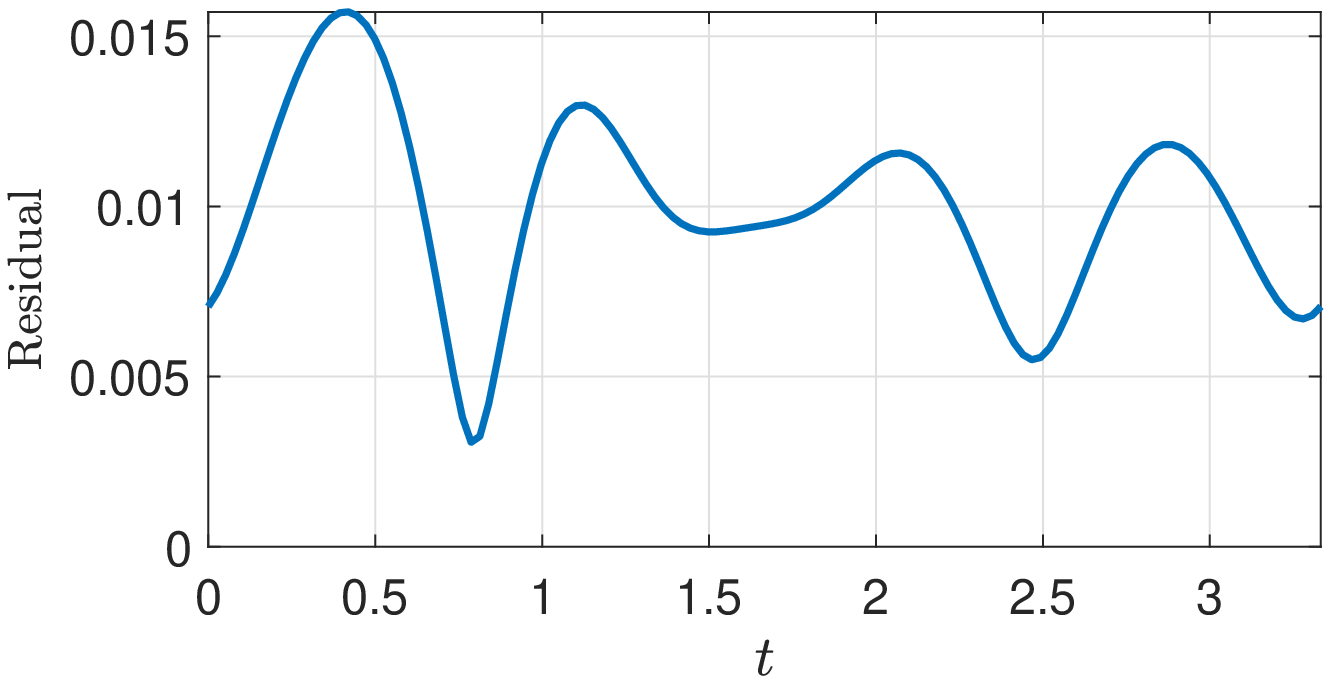}
	\caption{\small Time history of the residual of the invariance equation~\eqref{eq:invariance} for the saddle-node periodic orbit A shown in the upper-left panel of Fig.~\ref{fig:3Dos-x1-frc}.}
	\label{fig:residual-none-003-sn}
\end{figure}

As an illustration of this validation technique, we consider the periodic orbit at saddle-node bifurcation point, labelled `A' in Fig.~\ref{fig:3Dos-x1-frc}. This point obtained via an $\mathcal{O}(11)$-expansion of the SSM at $\epsilon=0.03$, overestimates the peak of FRC computed from the reference full solution. We determine the residual of the invariance equation~\eqref{eq:invariance} for this reduced solution and plot the residual time-history in Fig.~\ref{fig:residual-none-003-sn}. We observe that the maximum residual in Fig.~\ref{fig:residual-none-003-sn} is more than 0.015, which is large for an excitation amplitude $\epsilon=0.03$. Hence, our SSM approximations are not accurate for such values of $\epsilon$ and higher-order non-autonomous terms need to be considered for reducing the residual.  We remark that this validation technique does not require the full system's response in order to estimate the accuracy of the prediction.

\subsection{Pendulum models with and without internal resonances}
In this section, we demonstrate the computation of SSMs for systems with non-polynomial nonlinearities via SSMTool by recasting them into polynomial DAE systems using the approach of Section~\ref{sec:nonpoly-to-poly}.
\subsubsection{Simple pendulum}
 We first consider the pendulum equation~\eqref{eq:pend-ode} in its unforced limit, i.e., with $\epsilon=0$. Similarly to the previous example, we will compute the damped backbone curve via the reduced dynamics on the two-dimensional SSM of the DAE system~\eqref{eq:pend-dae}. Then, we will add external periodic forcing in the next two pendulum examples to compute the FRC via the reduced dynamics on the corresponding non-autonomous SSMs created by the forcing. 

For low damping ($0<c<2$), the eigenvalues for the linear part of the DAE system~\eqref{eq:pend-dae} are given by
\begin{equation}
    \lambda_1=0,\quad \lambda_{2,3}= -\frac{c}{2} \pm \mathrm{i}\sqrt{1 - \frac{c^2}{4}},\quad |\lambda_{4}|=\infty,
\end{equation}
where the zero and infinite magnitude eigenvalues are artifacts of reformulation of the ODE system~\eqref{eq:pend-ode} to the DAE system~\eqref{eq:pend-dae}, as discussed in Section~\ref{sec:nonpoly-to-poly}. Hence, we take the spectral subspace corresponding to $\lambda_{2,3}$ as the master subspace of the SSM. We consider two values of damping $c = 0.001, 0.1 $  and obtain SSM-based ROMs up to $\mathcal{O}(35)$ via SSMTool. 
\begin{figure}[!ht]
	\centering
	\includegraphics[width=.45\textwidth]{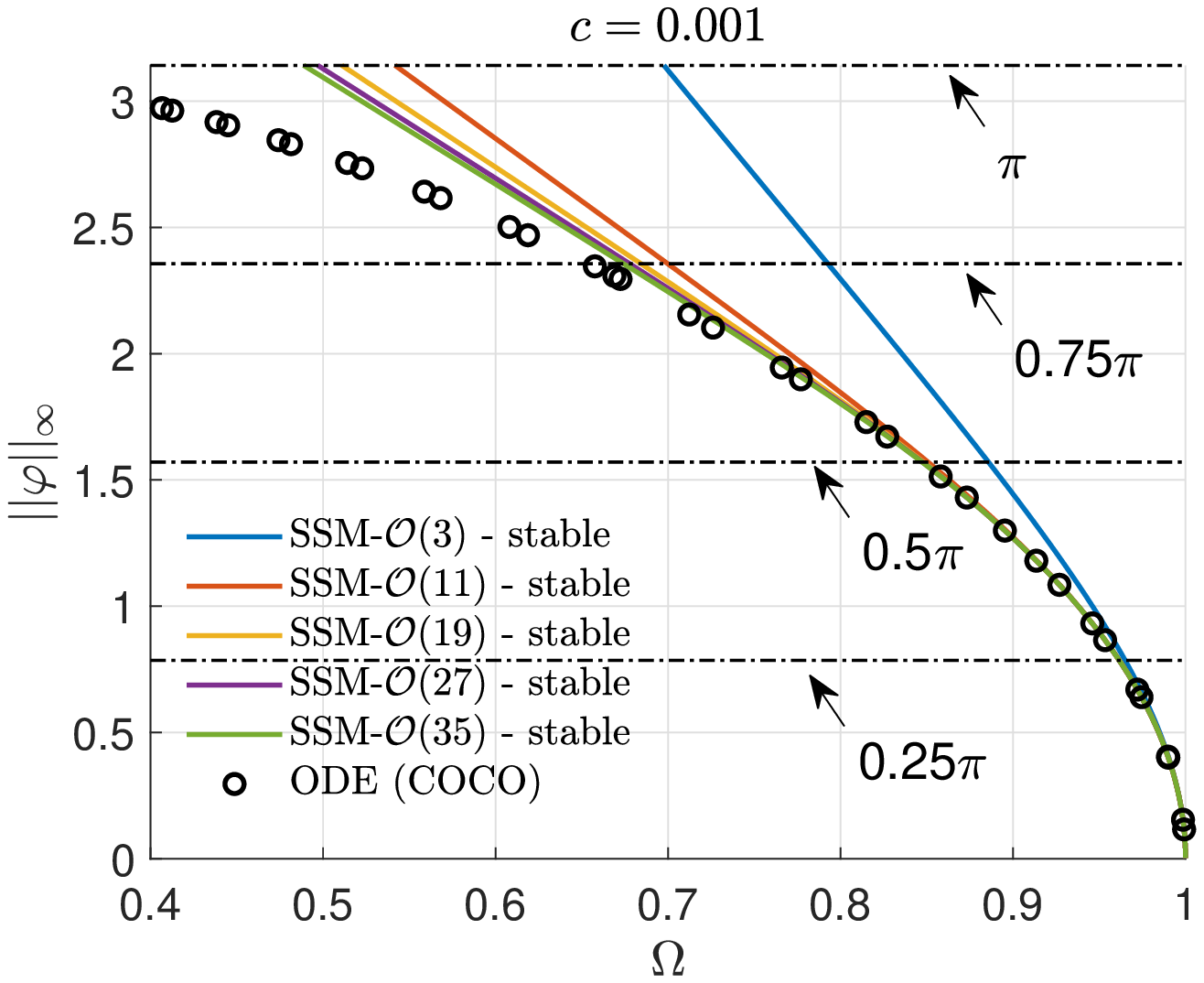}
	\includegraphics[width=.45\textwidth]{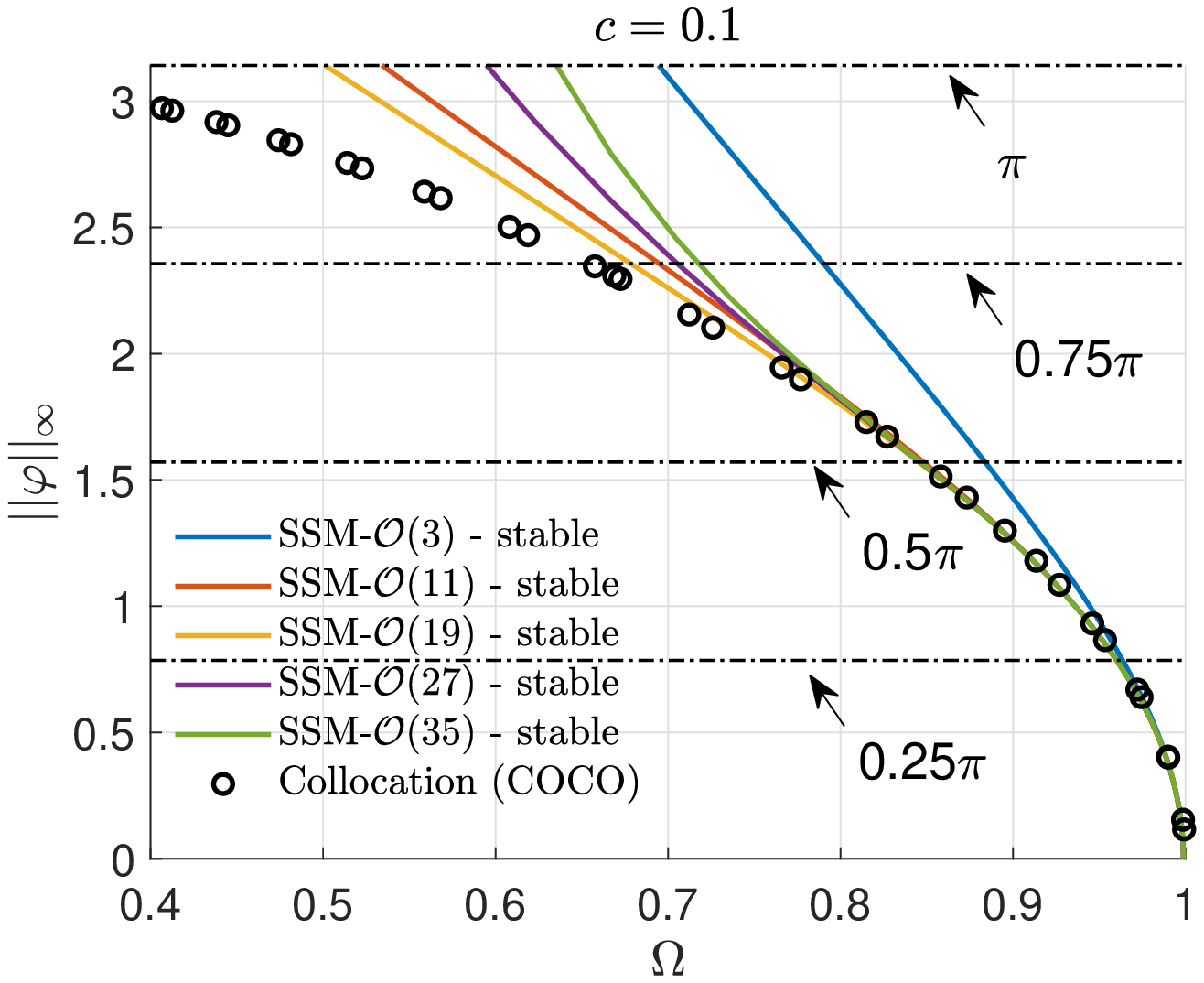}
	\caption{\small Backbone curves in the vibration amplitude of the angle of pendulum, obtained by 2-dimensional SSM reduction at various orders, and the parameter continuation of the original system in undamped limit.}
	\label{fig:single-pend-backbone}
\end{figure}

In Fig.~\ref{fig:single-pend-backbone}, we compare the damped-backbone curves obtained via SSM-based ROMs to the conservative backbone of the pendulum obtained via periodic orbit continuation using the \texttt{po}-toolbox of \textsc{coco}~\cite{dankowicz2013recipes}. We observe that for the lower damping value, $c=0.001$, the SSM-based damped backbone agrees with the conservative backbone till a higher response amplitude of around $\frac{3\pi}{4}$ relative to that for the higher damping value of $c=0.1$. Once again, this confirms the expectation that for lightly damped systems, the conservative backbone serves as a first-order approximation for the damped backbone curve.

\subsubsection{A pendulum-slider with 1:3 internal resonance}
\label{sec:pend-slider}
Next, we consider a pendulum attached to a slider under periodic forcing, as shown in Fig.~\ref{fig:slide_beam}. The FRC of periodic orbits for this system has been studied in~\cite{ju2021efficient} using the harmonic balance method. Here, we adjust the system parameters to introduce a 1:3 internal resonance between the first two modes of the system and study free and forced vibrations of the system using SSM reduction.

\begin{figure}[!ht]
\centering
\includegraphics[width=.3\textwidth]{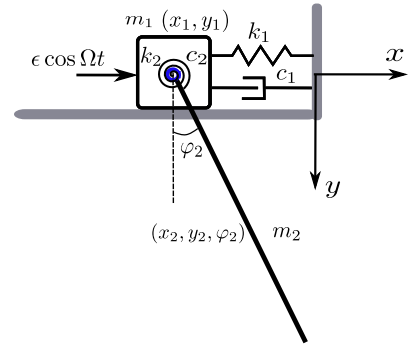}
\caption{\small A pendulum slider under periodic forcing.}
\label{fig:slide_beam}
\end{figure}

\begin{sloppypar}
Assuming the suspended beam in Fig.~\ref{fig:slide_beam} to be rigid, the equations of motion for the pendulum slider system are given as
\begin{gather}
m_1\ddot{x}_1+c_1\dot{x}_1+k_1x_1-\mu_2=\epsilon \cos\omega t,\nonumber\\
m_1\ddot{y}_1+\mu_1-\mu_3=m_1g,\nonumber\\
m_2\ddot{x}_2+\mu_2=0,\nonumber\\
m_2\ddot{y}_2+\mu_3=m_2g,\nonumber\\
J_2\ddot{\varphi}_2+c_2\dot{\varphi}_2+k_2\varphi_2\nonumber\\
\quad\quad-0.5l\mu_2\cos\varphi_2+0.5l\mu_3\sin\varphi_2=0,\label{eq:eom-slider}
\end{gather}
where $g$ is the acceleration due to gravity and the masses $m_1$ and $m_2$ satisfy the following configuration constraints 
\begin{align}
    g_1&=y_1 = 0,\nonumber\\ 
    g_2 &= x_2-(x_1+0.5l\sin\varphi_2) = 0,\nonumber\\ 
    g_3&=y_2-(y_1+0.5l\cos\varphi_2) = 0.\label{eq:constraints-slider}
\end{align}
We note that the origin is not an equilibrium of the unforced system and perform the transformation
\begin{gather}
    y_2=\hat{y}_2+0.5l,\nonumber\\ \mu_1=\hat{\mu}_1+m_1g+m_2g,\nonumber\\ \mu_3=\hat{\mu}_3+m_2g,
\end{gather}
such that the trivial equilibrium is shifted to the origin. Similarly to the previous simple pendulum example, we introduce the auxiliary variables $u_1=\sin\varphi_2$ and $u_2=1-\cos\varphi_2$ to recast the trigonometric functions in the equations of motion into polynomials as
\begin{gather}
m_1\ddot{x}_1+c_1\dot{x}_1+k_1x_1-\mu_2=\epsilon f_1\cos\omega t,\nonumber\\
 m_1\ddot{y}_1+\hat{\mu}_1-\hat{\mu}_3=0,\nonumber\\
m_2\ddot{x}_2+\mu_2=0,\nonumber\\
m_2\ddot{\hat{y}}_2+\hat{\mu}_3=0,\nonumber\\
J_2\ddot{\varphi}_2+c_2\dot{\varphi}_2+k_2\varphi_2-0.5l\mu_2(1-u_2)\nonumber\\
\qquad\,\,+0.5l(\hat{\mu}_3+m_2g)u_1=0,\nonumber\\
g_1=y_1=0,\nonumber\\
 g_2=x_2-(x_1+0.5lu_1)=0,\nonumber\\
g_3=\hat{y}_2-y_1+0.5lu_2=0,\nonumber\\
\dot{u}_1=(1-u_2)\dot{\varphi}_2,\nonumber\\
 u_1^2+(1-u_2)^2=1\label{eq:dae-pend-slider}.
\end{gather}
With the state vector
\begin{align}
  \boldsymbol{z}=
   (& x_1,y_1,x_2,\hat{y}_2,\varphi_2,\dot{x}_1,\dot{y}_1,\dot{x}_2,\nonumber\\
   &\dot{\hat{y}}_2,\dot{\varphi}_2,\mu_1,\mu_2,\mu_3,u_1,u_2),  
\end{align} 
we can rewrite the DAE system~\eqref{eq:dae-pend-slider} in the first-order form~\eqref{eq:full-first} with a 15-dimensional phase space.
\end{sloppypar}

For the system parameters $m_1=m_2=1$, $c_1=c_2=0.02$, $k_1=7.48$, $k_2=1$ and $g=9.8$, the eigenvalues of the linear part of the DAE system~\eqref{eq:dae-pend-slider} are given as
\begin{gather}
  \lambda_{1,2}=-0.0047\pm1.8522\mathrm{i},\nonumber\\ \lambda_{3,4}=-0.0513\pm5.5561\mathrm{i},\nonumber\\
  \lambda_5=0,\quad |\lambda_{6,\cdots,15}|=\infty.
\end{gather}
As $\lambda_{3,4}\approx3\lambda_{1,2}$, the system exhibits a near 1:3 internal resonance for the chosen parameters. Thus, we take the spectral subspace spanned by the first two modes as the master subspace for SSM computation. We compute the corresponding 4-dimensional SSM and its reduced dynamics via SSMTool. In the unforced setting ($\epsilon=0$), the autonomous ROM on the SSM in polar coordinates $\boldsymbol{p}=(\rho_1 e^{\mathrm{i}\vartheta_1},\rho_1 e^{-\mathrm{i}\vartheta_1},\rho_2 e^{\mathrm{i}\vartheta_2},\rho_2 e^{-\mathrm{i}\vartheta_2})$ up to cubic terms is given as
\begin{align}
    \dot{\rho}_1=&-{1.096\cdot 10^{-6}}\rho_1^3+ \nonumber\\
    & {\left({2.278\cdot10^{-5}}\cos\sigma+4.657\cdot10^{-4}\sin\sigma\right)}\rho_1^2\rho_2\nonumber\\
    & -{5.61\cdot 10^{-6}}\rho_1\rho_2^2 -4.701\cdot 10^{-3}\rho_1,\nonumber\\
    \dot{\vartheta}_1= & 8.074\cdot10^{-4}\rho_1^2 +\nonumber\\
    & {\left(4.657\cdot10^{-4}\cos\sigma -{2.278\cdot10^{-5}}\sin\sigma\right)}\rho_1\rho_2\nonumber\\
    & -1.278\cdot10^{-3}\rho_2^2 +1.852,\nonumber\\
    \dot{\rho}_2=&{\left(-7.438\cdot10^{-5}\cos\sigma -1.692\cdot10^{-3}\sin\sigma\right)}\rho_1^3+\nonumber\\& 6.25\cdot10^{-4}\rho_1^2\rho_2 +2.125\cdot10^{-3}\rho_2^3-0.0513\rho_2,\nonumber\\
    \dot{\vartheta}_2=& {\left(1.692\cdot10^{-3}\cos\sigma -{7.438\cdot10^{-5}}\sin\sigma\right)}\rho_1^3/\rho_2\nonumber\\
    & -0.1007\rho_2^2 -0.01393\rho_1^2 +5.556,
\end{align}
where $\sigma=3\vartheta_1-\vartheta_2$. We see coupling terms between the dynamics of the two pair of modes due to the internal resonance. The damped backbone curve defined by instantaneous frequency (cf.~\eqref{eq:red-free}-\eqref{eq:red-sphere}) is
inapplicable because of the coupling terms. Then we are not able to select the expansion order of the SSM based on the convergence of the backbone curve.

\subsubsection*{A posteriori invariance error measure}
Next, we provide an a posteriori error estimation method which is also useful in determining an appropriate order for the SSM expansion. To this end, we compute the residual of the invariance equation~\eqref{eq:invariance-auto} over a set of sampled points on the computed SSM and use the averaged residual over the samples as a measure for the invariance error. The sample set of points on the SSM is constructed as
\begin{equation}
\label{eq:samples-on-ssm}
    \boldsymbol{p}_{ijk}=\begin{pmatrix}\varrho\cos\alpha_ie^{\mathrm{i}\vartheta_{1,j}}\\\varrho\cos\alpha_ie^{-\mathrm{i}\vartheta_{1,j}}\\\varrho\sin\alpha_ie^{\mathrm{i}\vartheta_{1,k}}\\\varrho\sin\alpha_ie^{-\mathrm{i}\vartheta_{1,k}}\end{pmatrix},\,\,
\end{equation}
for $ 1\leq i\leq n_\alpha$ and $1\leq i,j\leq n_\vartheta$
Here, each $ \boldsymbol{p}_{ijk} $ is a point on a 4-dimensional sphere with radius $\varrho=\sqrt{\rho_1^2+\rho_2^2}$ in the SSM-parametrization space, and
\begin{gather}
    \alpha_i=\frac{(i-1)\pi}{2(n_\alpha-1)},\quad \vartheta_{1,j}=\frac{2(j-1)\pi}{n_\vartheta},\nonumber\\
    \vartheta_{1,k}=\frac{2(k-1)\pi}{n_\vartheta}.
\end{gather}
We choose $\alpha_i\in[0,\pi/2]$ such that $\rho_1\geq0$ and $\rho_2\geq0$. As the numbers $n_\alpha, n_\vartheta$ increase, we obtain a more refined sampling over the 4-sphere. Now, $\boldsymbol{W}(\boldsymbol{p}_{ijk})$ provides the coordinates of the sampled points on the SSM in the full phase space, which enables us to calculate the averaged invariance error as
\begin{equation}
\label{eq:auto-error-measure}
    \mathrm{Error} =\frac{1}{Nn_{\alpha}n_{\vartheta}^2}
    \sum_{i=1}^{n_\alpha}\sum_{j=1}^{n_\vartheta}\sum_{k=1}^{n_\vartheta} ||\mathbf{Res}(\boldsymbol{p}_{ijk})||_2.
\end{equation}
where
\begin{align}
    \mathbf{Res}(\boldsymbol{p}_{ijk})=&\boldsymbol{B}{D}_{\boldsymbol{p}}\boldsymbol{W}(\boldsymbol{p}_{ijk}) \boldsymbol{R}(\boldsymbol{p}_{ijk})-\nonumber\\
    & \boldsymbol{A}\boldsymbol{W}(\boldsymbol{p}_{ijk})-\boldsymbol{F}( \boldsymbol{W}(\boldsymbol{p}_{ijk})).
\end{align}
Here we normalize the error norm by dividing it by the phase space dimension $N$ as the Euclidean norm of the residual vector increases linearly with the dimension of the phase space. We use such a normalized error measure to characterize the accuracy of SSM expansions.
Generally, for a fixed small value of $\varrho$, i.e., near the origin, the invariance error~\eqref{eq:auto-error-measure} will decrease as the order of approximation of the SSM increases. We remark that the invariance error estimate~\eqref{eq:auto-error-measure} does not require any simulation of the full system or of the reduced system and is applicable for general invariant manifolds. Furthermore, while we have only discussed the case of 4-dimensional SSM for simplicity, the sampling formula~\eqref{eq:samples-on-ssm} can be generalized to $2m$-dimensional SSMs for $ m\in\mathbb{N}$.

\begin{figure}[!ht]
	\centering
	\includegraphics[width=.5\textwidth]{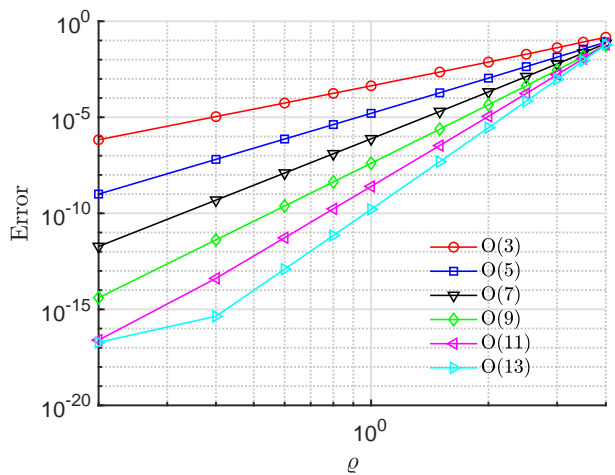}
	\caption{Invariance error measure~\eqref{eq:auto-error-measure} of the 4-dimensional SSM approximations of the pendulum slider system~\eqref{eq:dae-pend-slider} at various orders as the distance from the origin ($\varrho$) varies.}
	\label{fig:error-auto-slide-pend}
\end{figure}

In this example, we choose $n_\alpha=10$ and $n_\vartheta=30$ and calculate the invariance errors for different values of $\varrho$ and expansion order. The logarithmic plot in Fig.\ref{fig:error-auto-slide-pend} shows that the invariance error decays following a power law  until it approaches machine precision as $\varrho$ decreases. The corresponding decay rate increases with increasing order of approximation of the SSM, as expected. Furthermore, we observe that for a given $\varrho$, the error decreases with increase in the approximation order up to a critial value of $\varrho\approx 4$, beyond which the error seems unaffected with increase in the expansion order. Hence, we conclude that $\varrho=4$ is close to the boundary of the domain of convergence of our SSM expansions. 

Indeed, Fig.~\ref{fig:error-auto-slide-pend} provides critical insights for assessing the accuracy of SSM-based ROM predictions and for choosing a suitable order of approximation for the SSM. For instance, with an error tolerance of 0.01, an $\mathcal{O}(3)$-expansion is sufficient in the domain $\varrho\leq1$, whereas $\mathcal{O}(13)$-expansion is needed to guarantee similar accuracy in the domain  $\varrho\leq3$. 

\subsubsection*{Time histories and forced response curves}

Next, we choose an error tolerance of 0.01, which is near but within the convergence domain boundary deduced from Fig.~\ref{fig:error-auto-slide-pend}. We compare the trajectories from the simulations of the  SSM-based ROM upto $\mathcal{O}(13)$ with those of the full system. In particular, we consider two sets of initial conditions as
\begin{gather}
    \mathrm{IC1}: \boldsymbol{p}_0 = (3.5e^{\mathrm{i}},3.5e^{-\mathrm{i}},0,0),\nonumber\\ \mathrm{IC2}: \boldsymbol{p}_0 = (0,0,3.5e^{\mathrm{i}},3.5e^{-\mathrm{i}}),\label{eq:ic-slide-pend-ssm}
\end{gather}
where IC1 and IC2 are chosen along the first and the second mode on a hypersphere of radius $\varrho=3.5$ in the parametrization space of the SSM. Physically, the first mode is dominated by the vibration of the slider along the $x_1$ direction whereas the second mode is dominated by the oscillation of the pendulum ($\varphi_2$). The trajectories of the horizontal displacement $x_1$, the rotation angle $\varphi_2$ of the pendulum initialized at IC1 and IC2 are shown in the left and right panels of Fig.~\ref{fig:slide-pend-ic1}. The results obtained by SSM prediction match well with the reference results of the full system. An excellent match is also obtained for Lagrange multipliers, as detailed in Appendix~\ref{sec:appendix-slider}. This indicates a high accuracy of SSM prediction for reaction forces as well.

\begin{figure*}[!ht]
	\centering
	\includegraphics[width=.45\textwidth]{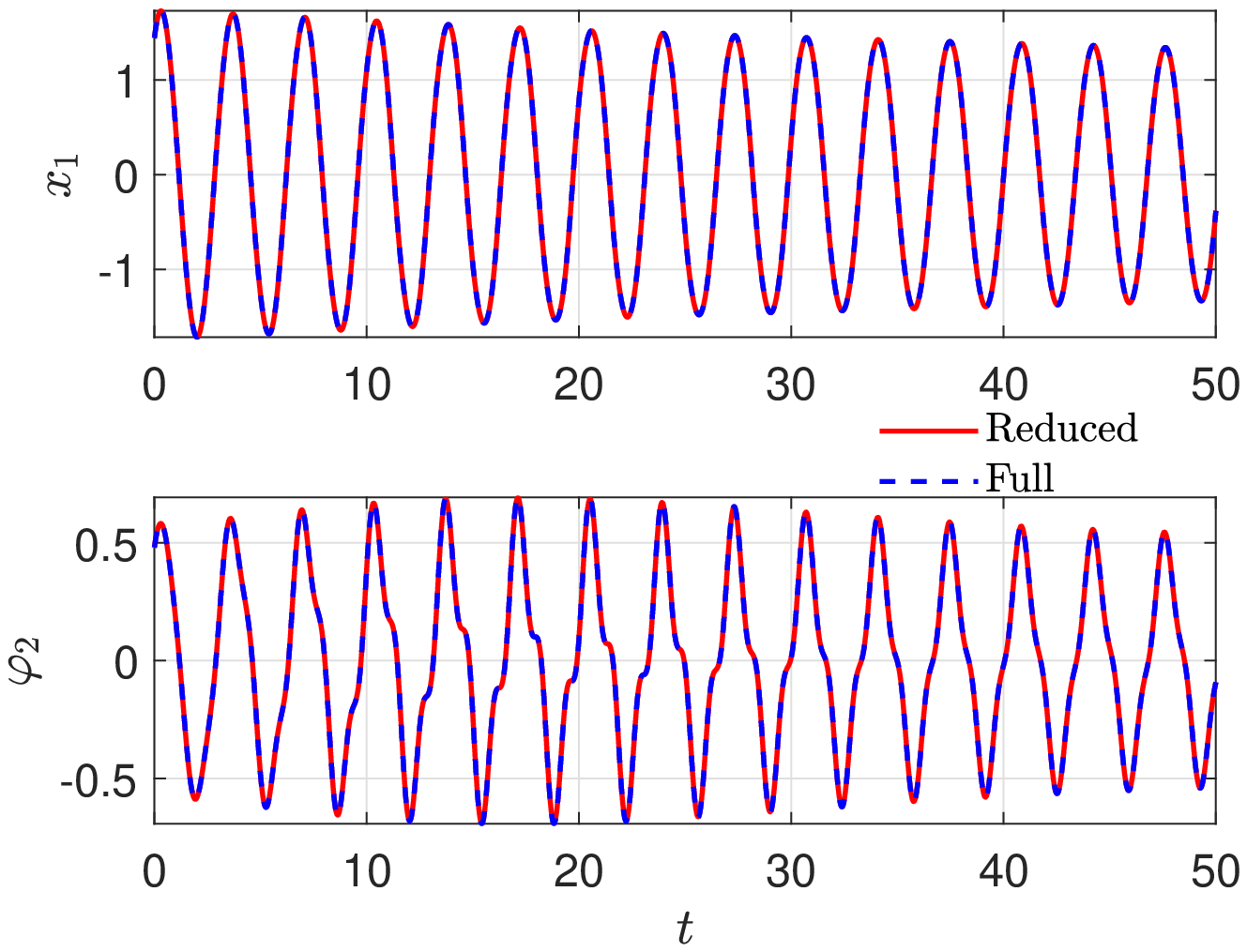}
	\includegraphics[width=.45\textwidth]{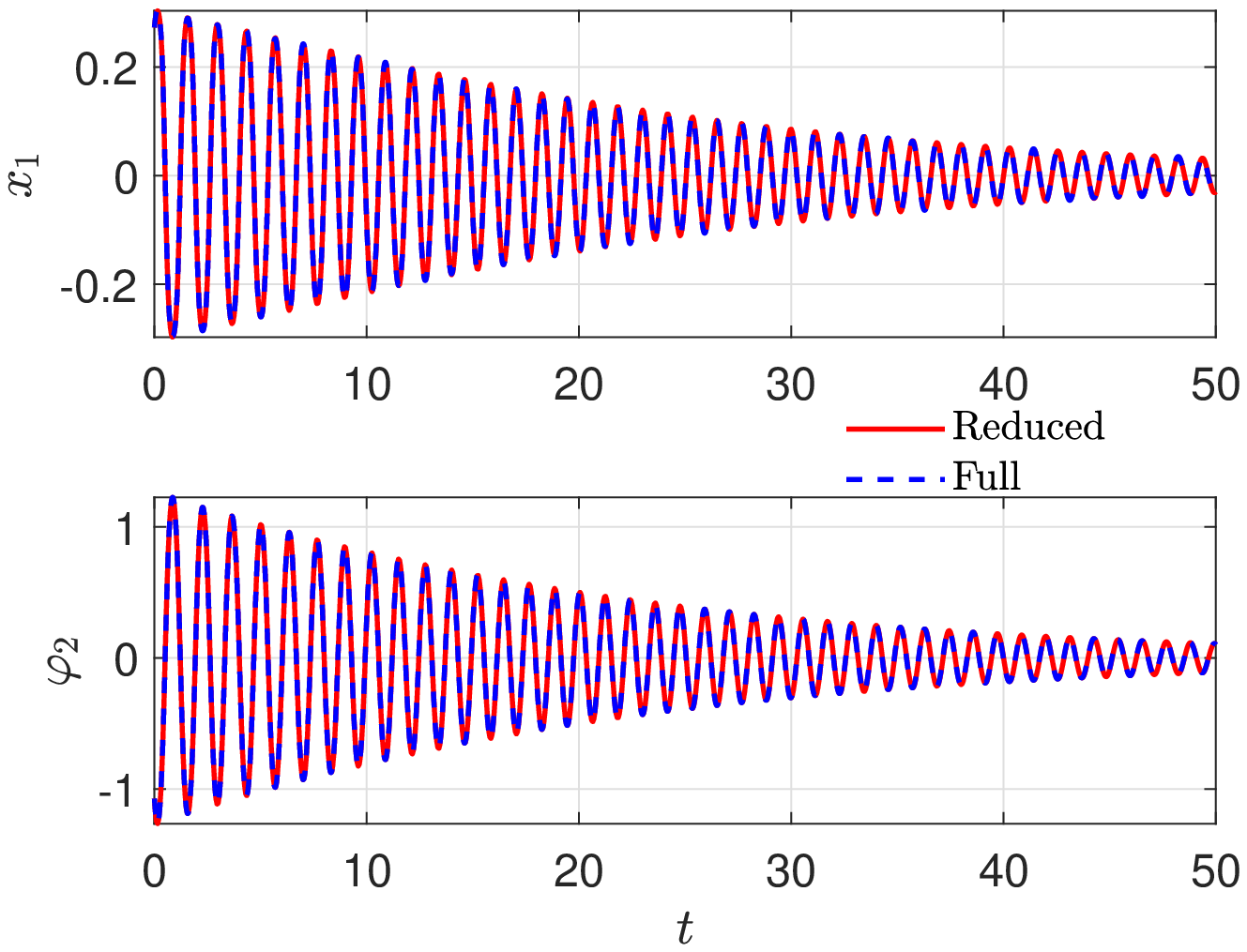}	
	\caption{Time histories for the displacements $(x_1,\varphi_2)$ of the pendulum slider system~\eqref{eq:dae-pend-slider} with initial condition IC1 $(3.5e^{\mathrm{i}},3.5e^{-\mathrm{i}},0,0)$ (left) and IC2 $(0,0,3.5e^{\mathrm{i}},3.5e^{-\mathrm{i}})$ (right) on the SSM approximated up to $\mathcal{O}(13)$.}
	\label{fig:slide-pend-ic1}
\end{figure*}

The response in the left panel of Fig.~\ref{fig:slide-pend-ic1} is dominated by the first mode that involves the vibration of the slider near the first natural frequency  whereas the response in right panel is dominated by the second mode with oscillation of the pendulum at near thrice the frequency of the first mode. Interestingly, the response time history of $\varphi_2$ in the left panel of Fig.~\ref{fig:slide-pend-ic1} indicates modal interaction as the amplitude first increases and then decays, i.e., exchange of energy between the vibration of the slider and the oscillation of the pendulum. 


Adding a periodic forcing $\epsilon\cos\Omega t$ to the slider, we now compute the FRC of the system~\eqref{eq:dae-pend-slider} with $\epsilon=0.08$ and $\Omega\approx\mathrm{Im}(\lambda_1)$ using SSMTool. The FRC is obtained directly by analyzing the reduced-dynamics of the 4-dimensional SSM in a normal-form parametrization style, where the periodic orbit is simply given by the solution to a fixed point problem (see~\cite{part-i} for details). FRCs for the horizontal vibration of the slider ($x_1$) and the oscillation of the pendulum ($\varphi_2$) are shown in Fig.~\ref{fig:slide-pend-frc}. We observe that the FRCs obtained by SSM-based reduction converge towards the reference solution at $\mathcal{O}(7)$. 

\begin{figure}[!ht]
	\centering
	\includegraphics[width=.45\textwidth]{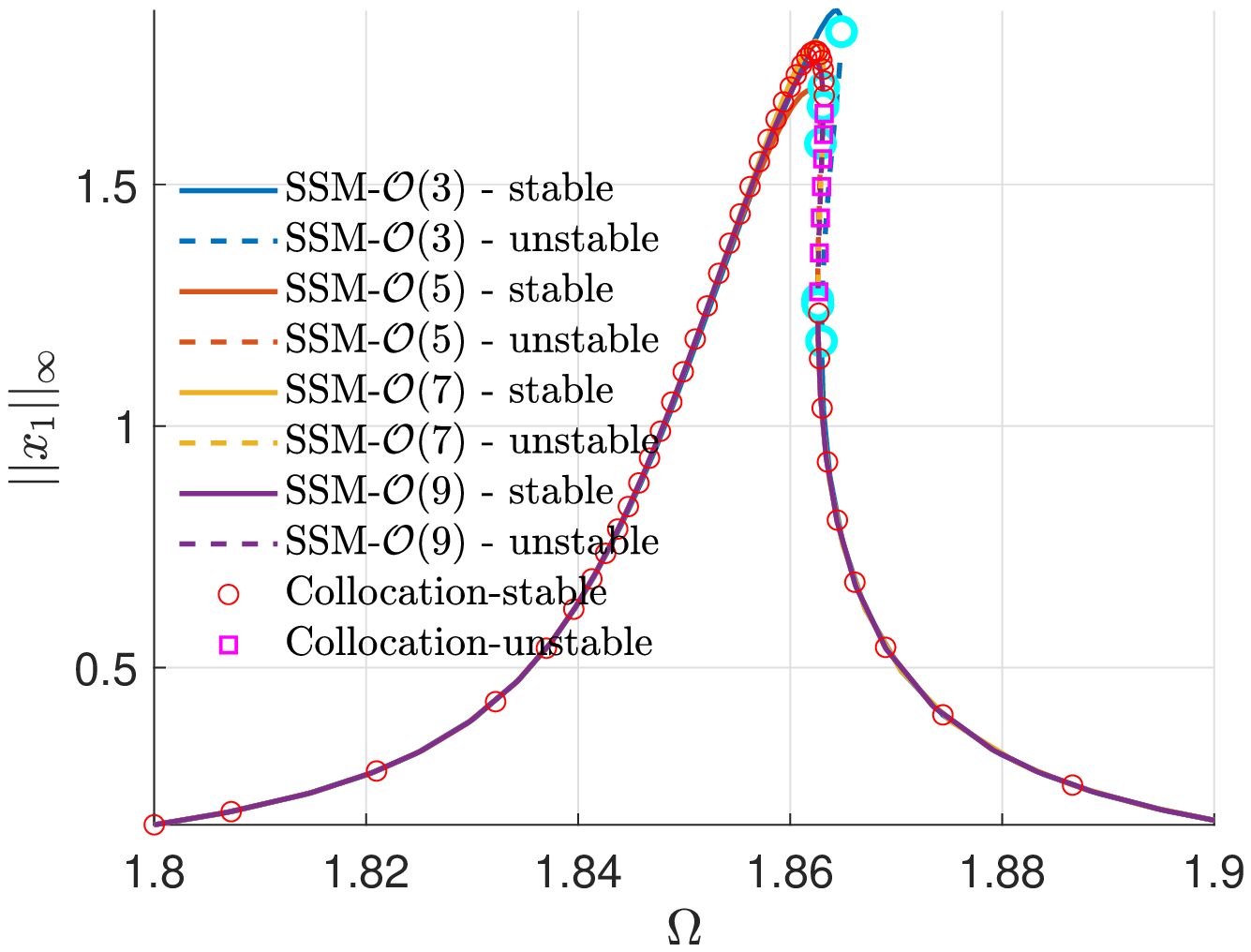}
	\includegraphics[width=.45\textwidth]{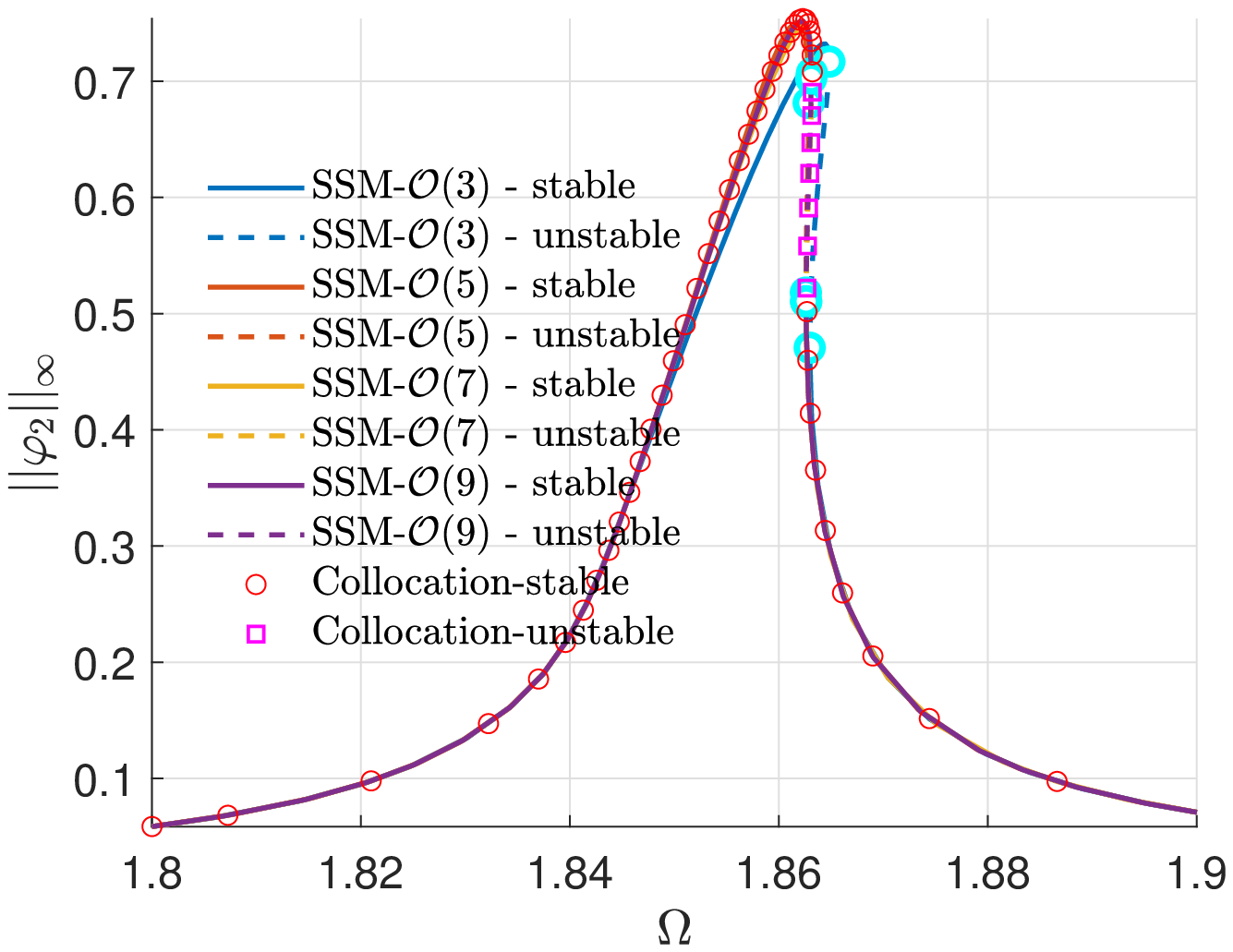}
	\caption{Forced response curves in vibration amplitude of the slider ($x_1$, left panel) and the pendulum ($\varphi_2$, right panel), obtained by SSM reduction at various approximation orders. The reference solutions obtained by parameter continuation of the periodic orbits of~\eqref{eq:pend-slide-ode} (labelled as `Collocation') are presented for validation.}
	\label{fig:slide-pend-frc}
\end{figure}

\begin{sloppypar}
We obtain the reference solutions in Figs.~\ref{fig:slide-pend-ic1} and \ref{fig:slide-pend-frc} by simulating the Euler–Lagrange equations of the full system in the generalized coordinates $x_1$ and $\varphi_2$ given by
\begin{align}
	m_1\ddot{x}_1&+m_2(\ddot{x}_1-0.5l\sin\varphi_2\dot{\varphi}_2^2+0.5l\cos\varphi_2\ddot{\varphi}_2)\nonumber\\&+c_1\dot{x}_1+k_1x_1=\epsilon f_1\cos\omega t,\nonumber\\
	J_2\ddot{\varphi}_2&+m_2\left(0.5l\ddot{x}_1\cos\varphi_2+0.25l^2\ddot{\varphi}_2\right)\nonumber\\
	& +c_2\dot{\varphi}_2+k_2\varphi_2+0.5lm_2g\sin\varphi_2=0\label{eq:pend-slide-ode}.
\end{align}
In particular, the reference FRC in Fig.~\ref{fig:slide-pend-frc} is computed via parameter continuation using the \texttt{po}-toolbox of~\textsc{coco}~\cite{COCO} on system~\eqref{eq:pend-slide-ode}. We remark that the derivation of Euler–Lagrange equation with minimal number of coordinates, while concise in this simple example, becomes cumbersome and unfeasible for higher-dimensional constrained mechanical systems, such as the pendulum chain in our next example. At the same time, the derivation of DAE formulations remains straightforward for high-dimensional systems.
\end{sloppypar}

\subsubsection{A chain of pendulums}
\label{sec:example-chain}
As our final pendulum-based example, we consider a chain of pendulums attached to a slider~\cite{ju2021efficient}, illustrated in Fig.~\ref{fig:pend-chain}.  We derive the equations of motion of this system in the form~\eqref{eq:full-first} in Appendix~\ref{sec:eom-pend}.

\begin{figure}[!ht]
\centering
\includegraphics[width=.35\textwidth]{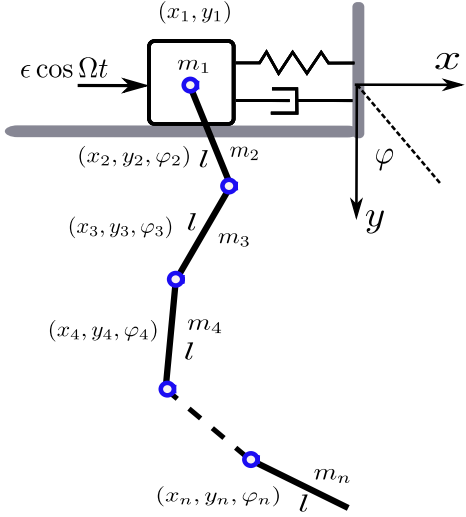}
\caption{A chain of pendulums attached to a slider.}
\label{fig:pend-chain}
\end{figure}

We choose the system parameters in Fig.~\ref{fig:pend-chain} as $n=41$, $m_1=0.61$, $m_2=\cdots=m_{41}=0.02$, $c_1=0.22$, $c=0.02$, $k_1=6.5$, $k=4.1$, and $l=0.03$. The first two pairs of nontrivial eigenvalues of the linear part of the DAE system are given as
\begin{gather}
  \lambda_{1,2}=-0.0569\pm1.9939\mathrm{i},\nonumber\\
  \lambda_{3,4}=-0.0730\pm4.9038\mathrm{i}.
\end{gather}
Due to the absence of any (near) internal resonance, we choose the slowest spectral subspace spanned by the first mode as the master subspace for SSM reduction. We compute the corresponding two-dimensional SSM and its reduced dynamics via SSMTool. In the unforced setting ($\epsilon=0$), the autonomous ROM on the SSM in polar coordinates $ (\rho ,\vartheta) $ upto $\mathcal{O}(13)$ is given as
\begin{align}
     \dot{\rho}=&-0.05691\rho_1-0.004259{\rho_1}^3+0.0001677{\rho}_1^5\nonumber\\
     & +0.000832{\rho}_1^7 -0.0009862{\rho}_1^9+0.001242{\rho}_1^{11}\nonumber\\
     & -0.001757{\rho}_1^{13} ,\nonumber\\
    \dot{\vartheta}= & 1.994+0.04249{\rho}_1^2-0.0107{\rho}_1^4\nonumber\\
    & +0.0003925{\rho}_1^6 +0.002133{\rho}_1^8 -0.002778{\rho}_1^{10}\nonumber\\ & +0.004065{\rho}_1^{12}=\omega(\rho) .\label{eq:chain-red}
\end{align}

\begin{figure}[!ht]
	\centering
	\includegraphics[width=.45\textwidth]{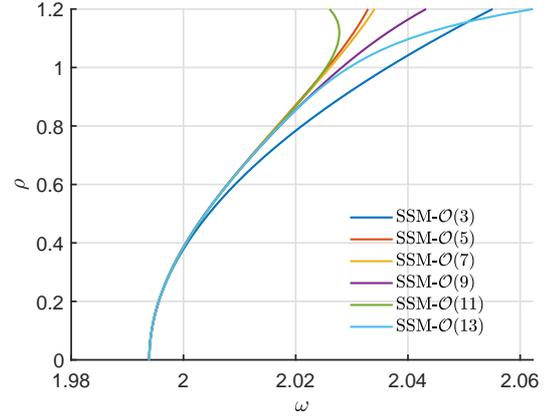}
	\caption{Backbone curves in reduced coordinates for the chain of pendulums attached to a slider.}
	\label{fig:chain-rho-omega}
\end{figure}

The damped backbone curves obtained directly from the reduced dynamics~\eqref{eq:chain-red} for different orders of SSM approximation are shown in Fig.~\ref{fig:chain-rho-omega}. Similarly to the previous examples, we observe that higher-order expansions are useful to obtain convergence in backbones at higher amplitudes. In particular, the backbone curves in Fig.~\ref{fig:chain-rho-omega} show convergence at $\mathcal{O}(3)$ for $\rho\leq0.3$, at $\mathcal{O}(5)$ expansion for $\rho\leq0.6$, and at $\mathcal{O}(13)$ expansion for $\rho\leq0.8$. The backbone curves do not seem to converge for $\rho\geq1$, indicating $\rho=1$ is outside the domain of convergence for the power series $\omega(\rho)$ in~\eqref{eq:chain-red}.

\begin{figure}[!ht]
	\centering
	\includegraphics[width=.45\textwidth]{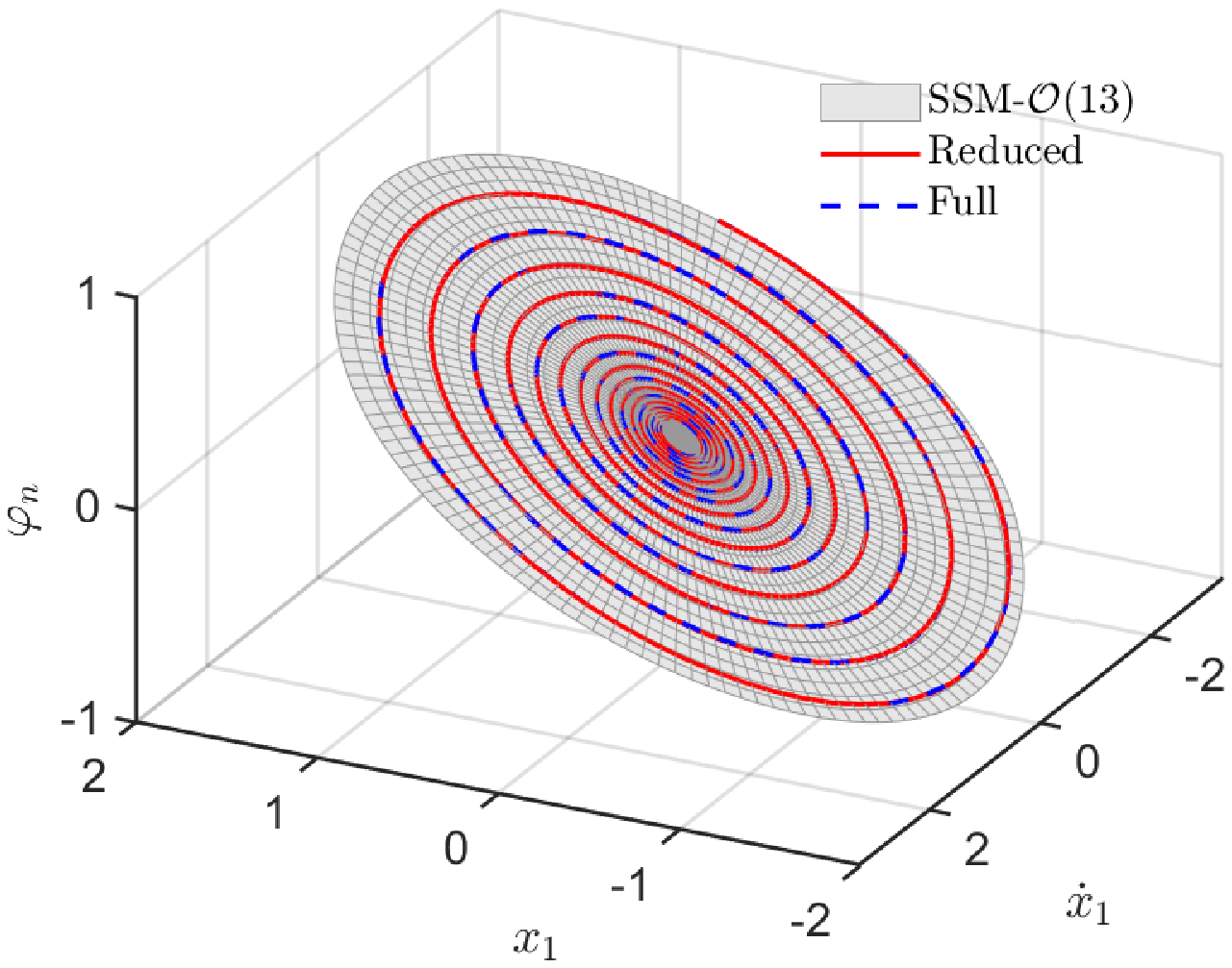}
	\includegraphics[width=.45\textwidth]{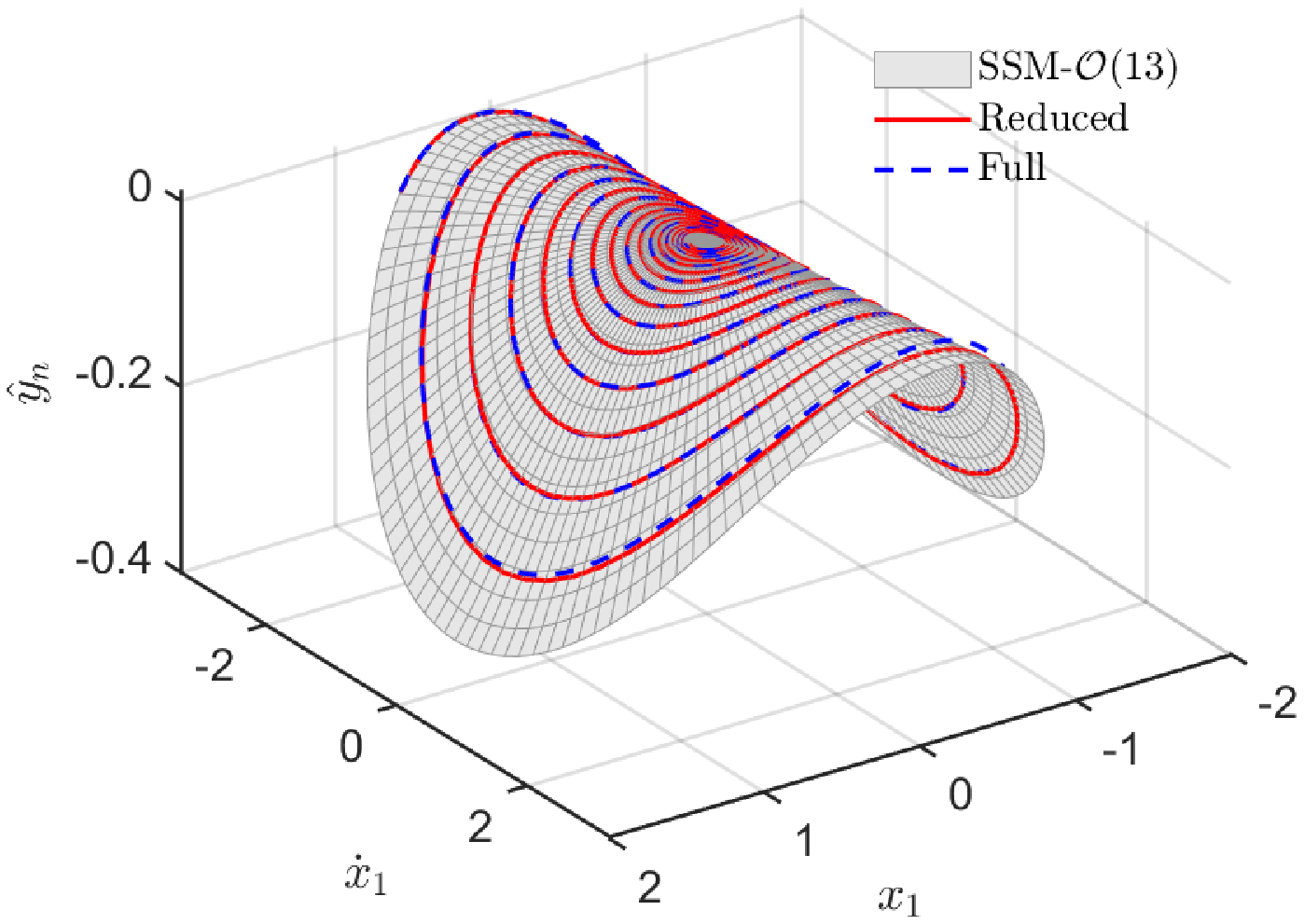}
	\caption{Projections of the SSM at $\mathcal{O}(13)$ approximation for the chain of pendulums onto $(x_1,\dot{x}_1,\varphi_n)$ and $(x_1,\dot{x}_1,\hat{y}_n)$. The red solid lines indicates the projections of the trajectory of the reduced-order model~\eqref{eq:chain-red}, starting from the initial position $(\rho_0,\vartheta_0)=(0.8,3)$. The dashed blue lines represent the projections of the trajectory of the full system for the same initial position. The thin solid gray curves represent contour lines of equal parameterized distance $\rho$ and $\vartheta$.}
	\label{fig:chain-ssm}
\end{figure}

To validate convergence, we illustrate the invariance of the SSM approximated upto $\mathcal{O}(13)$ within the convergence domain $\rho\leq0.8$ in Fig.~\ref{fig:chain-ssm}. To check invariance, we pick an initial condition $\boldsymbol{p}_0=(\rho_0 e^{\mathrm{i}\vartheta_0},\rho_0e^{-\mathrm{i}\vartheta_0})$ on the SSM with $\rho_0=0.8$ and $\vartheta_0=3$, and perform time integration of both the SSM-based ROM and the full system in the index-1 formulation~\eqref{eq:eom-2nd-nolambda}. Indeed, the full-system trajectory stays on the computed SSM and overlaps with the prediction of the SSM-based ROM. 

\begin{figure}[!ht]
	\centering
	\includegraphics[width=.45\textwidth]{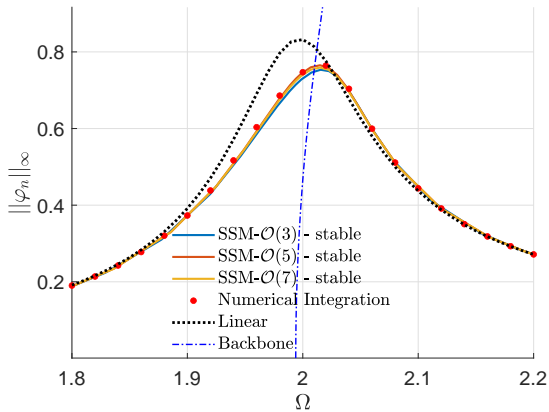}
	\caption{Forced response curve and backbone curve in the rotation angle of the last pendulum in the chain of pendulums.}
	\label{fig:chain-frc}
\end{figure}

Adding a periodic forcing $\epsilon\cos\Omega t$ to the slider in Fig.~\ref{fig:pend-chain}, we compute the FRC of the system~\eqref{eq:chain-eom} with $\epsilon=0.6$ and $\Omega\approx\mathrm{Im}(\lambda_1)$ using SSMTool. The FRC for the rotation angle of the last pendulum ($\varphi_n$) at various orders of SSM approximation is shown in Fig.~\ref{fig:chain-frc}. We observe that the SSM-based FRCs at $\mathcal{O}(5)$ converge towards the reference solution. The reference solution points Fig.~\ref{fig:chain-frc} in here are obtained from numerical time integration of the full system~\eqref{eq:chain-eom} in ODEs~\eqref{eq:eom-2nd-nolambda}. Specifically, we sample the frequency span $[1.8,2.2]$ rad/s uniformly to obtain 21 forcing frequencies. For each sampled frequency, we initialize a time integration at the unforced equilibrium position. The integrated solution is checked for convergence towards a periodic response after each forcing cycle of period $T = 2\pi/\Omega$ according to the criterion
\begin{equation}
	\frac{||\boldsymbol{z}(iT)-\boldsymbol{z}((i-1)T)||}{||\boldsymbol{z}((i-1)T)||}<\delta,
\end{equation}
where $ \delta $ is a user defined relative tolerance for convergence, chosen be 0.001. 

All computations of this example are performed on an Intel(R) Core(TM) i7-6700HQ processor (2.60 GHz) of a laptop. The total computational time for obtaining the 21 periodic orbits via numerical time integration using \texttt{ode15s} of MATLAB is 14.6 minutes, i.e., each periodic orbit takes 42 seconds on an average. At the same time, the entire FRC via an $\mathcal{O}(5)$ SSM reduction is obtained in just two seconds. 
Additionally, we plot the linear periodic response of~\eqref{eq:chain-eom}, which overestimates the peak amplitude, as shown in Fig.~\ref{fig:chain-frc}.

\subsection{A frequency divider}
\label{sec:example-divider}
As a final example, we consider the finite-element model of a frequency divider shown in Fig.~\ref{fig:frequency-divider}. This device is composed of two cantilevered beams that are initially perpendicular and their free-ends are connected via a revolute joint, resulting in a flexible multibody system. We choose the geometry such that the first two modes of the system satisfy a near 1:2 internal resonance, namely, $\omega_2\approx2\omega_1$. When the system is periodically forced near its second natural frequency, i.e.,  $\Omega\approx\omega_2$, we expect a subharmonic periodic response with frequency $0.5\Omega\approx\omega_1$ due to modal interactions. This nonlinear feature has been exploited for the design of frequency dividers~\cite{Strachan2013,Qalandar2014}.

\begin{figure}[!ht]
\centering
\includegraphics[width=.3\textwidth]{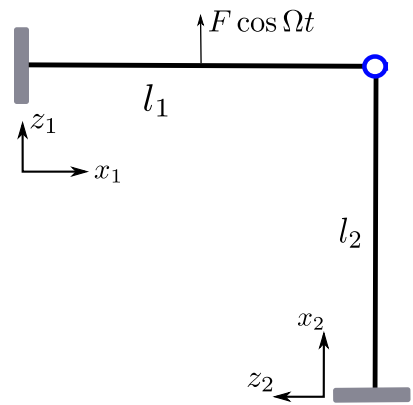}
\caption{A frequency divider made of two beams connected via a revolute joint.}
\label{fig:frequency-divider}
\end{figure}

We model the frequency divider using von K\'{a}rm\'{a}n beam elements in two-dimensional space. The finite element discretization results in three degrees of freedom per node that are associated to the axial displacement ($u$), the transverse displacement ($w$), and the rotation angle ($\phi$) (see~\cite{Jain2018} for details) . The equations of motion of the full system can be derived using the equations of motion of the two cantilevered beam substructures of length $ l_1 $ and $ l_2 $ by introducing two configuration constraints corresponding to the revolute joint as
\begin{gather}
	\begin{pmatrix}\boldsymbol{M}_1 & \boldsymbol{0}\\\boldsymbol{0}&\boldsymbol{M}_2\end{pmatrix}\begin{pmatrix}\boldsymbol{x}_1\\\boldsymbol{x}_2\end{pmatrix}+
	\begin{pmatrix}\boldsymbol{C}_1 & \boldsymbol{0}\\\boldsymbol{0}&\boldsymbol{C}_2\end{pmatrix}\begin{pmatrix}\dot{\boldsymbol{x}}_1\\\dot{\boldsymbol{x}}_2\end{pmatrix}\nonumber\\
	\qquad+\begin{pmatrix}\boldsymbol{K}_1 & \boldsymbol{0}\\\boldsymbol{0}&\boldsymbol{K}_2\end{pmatrix}\begin{pmatrix}\boldsymbol{x}_1\\\boldsymbol{x}_2\end{pmatrix}+
	\begin{pmatrix}\boldsymbol{f}_1(\boldsymbol{x}_1)\\\boldsymbol{f}_2(\boldsymbol{x}_2)\end{pmatrix}\nonumber\\
	\qquad+\boldsymbol{G}^\mathrm{T}\boldsymbol{\mu}=\epsilon\begin{pmatrix}\boldsymbol{f}_1^{\mathrm{ext}}(\Omega t)\\\boldsymbol{f}_2^{\mathrm{ext}}(\Omega t)\end{pmatrix},\label{eq:divider-diff}
\end{gather}
\begin{equation}
	\label{eq:divider-g}
	\boldsymbol{g} = \begin{pmatrix}
		g_1(\boldsymbol{x}_1,\boldsymbol{x}_2)\\ g_2(\boldsymbol{x}_1,\boldsymbol{x}_2)
	\end{pmatrix} 
	=\begin{pmatrix}
	u_1|_{x_1=l_1}+w_2|_{x_2=l_2},\\
	w_1|_{x_1=l_1}-u_2|_{x_2=l_2}
\end{pmatrix} = \mathbf{0},
\end{equation}
where the subsystems 1 and 2 denote the equations of motion of the hortizontal and vertical beams in Fig.~\ref{fig:frequency-divider}; $g_1$ and $g_2$ define the continuity constraints imposed by the revolute joint on the free ends of the two beams; $\boldsymbol{G}=\partial\boldsymbol{g}/{\partial(\boldsymbol{x}_1,\boldsymbol{x}_2)}$ is the Jacobian of the constraint equations; and $\boldsymbol{\mu}\in\mathbb{R}^2$ is the vector of Lagrange multipliers corresponding to these two constraints. 

With $h_i$, $b_i$ and $l_i$ denoting the thickness, width and the length of the beam in the $i$-th subsystem, we choose $h_1=h_2=1\,\mathrm{mm}$, $b_1=10\mathrm{mm}$, $b_2=1\,\text{mm}$, $l_1=707\,\mathrm{mm}$, and $l_2=1000\,\mathrm{mm}$. Both beams have the same material properties, namely density of $2700\times10^{-9}\,\mathrm{kg/mm^3}$, and the Young’s modulus of $70\times10^6\,\mathrm{kPa}$. We discretize the horizontal beam using $N_1=10$ finite elements and  the vertical beam using $N_2=14$ finite elements, yielding a 146-dimensional phase space $\boldsymbol{z} = (\boldsymbol{x}_1,\boldsymbol{x}_2, \dot{\boldsymbol{x}}_1,\dot{\boldsymbol{x}}_2, \boldsymbol{\mu})$ for the governing system in the first-order DAE form. 

\begin{figure}[!ht]
	\centering
	\includegraphics[width=.45\textwidth]{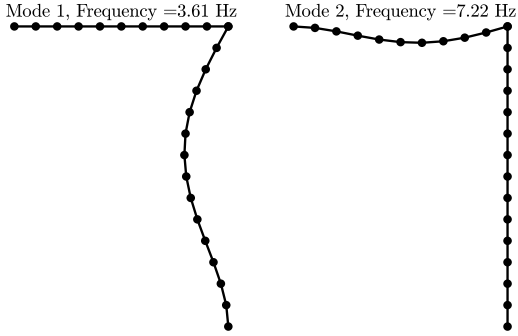}
	\caption{The first two mode shapes of the frequency divider.}
	\label{fig:frequency-divider-modeshape}
\end{figure}

For the chosen physical parameters, the first two undamped natural frequencies of the linear part of system~\eqref{eq:divider-diff} are given by
\begin{equation}
    \omega_1=22.66\,\mathrm{rad/s},\quad \omega_2=45.34\,\mathrm{rad/s}\approx2\omega_1.
\end{equation}
Thus, we obtain an internal resonance between the natural frequencies of the first two modes. The corresponding mode shapes are plotted in Fig.~\ref{fig:frequency-divider-modeshape}. We observe that both modes exhibit localized bending of either beams: mode 1 features bending of the vertical beam leaving the horizational beam undeformed, vice versa for mode 2.    

We choose a Rayleigh damping model $\boldsymbol{C}_i=\beta\boldsymbol{K}_i$ ($i=1,2)$ with damping ratio $\beta=10^{-3}/7$, resulting in the following two two pairs of eigenvalues corresponding to the first two modes of the damped linear system 
\begin{gather}
    \lambda_{1,2}=-0.04\pm22.66\mathrm{i}\approx\mathrm{i}\omega_1,\nonumber\\ \lambda_{3,4}=-0.15\pm45.34\mathrm{i}\approx\mathrm{i}\omega_2.
\end{gather}

To account for the near 1:2 internal resonance, we take the spectral subspace spanned by the first two modes as the master subspace for SSM computation. We compute the corresponding four-dimensional SSM and its reduced dynamics via SSMTool. In the unforced setting ($\epsilon=0$), we obtain the autonomous ROM on the SSM in polar coordinates up to cubic approximation as 
\begin{align}
    \dot{\rho}_1=&-{5.642\cdot10^{-5}}{\rho}_1^3 -{1.332\cdot10^{-8}}\rho_1{\rho}_2^2 \nonumber\\
    & +{\left(2.972\cdot10^{-4}\cos\sigma-8.97\cdot10^{-3}\sin\sigma\right)}\rho_1 \rho_2 \nonumber\\
    & -0.03669\rho_1,\nonumber\\
    \dot{\vartheta}_1=&2.635\cdot10^{-4}{\rho}_1^2 -1.7\cdot10^{-6}\rho_2^2 \nonumber\\
    & +{\left(-8.97\cdot10^{-3}\cos\sigma-2.972\cdot10^{-4}\sin\sigma\right)}\rho_2\nonumber\\
    & +22.66,\nonumber\\
    \dot{\rho}_2=&-5.164\cdot10^{-9}{\rho}_1^2\rho_2 \nonumber\\
    & +{\left(9.135\cdot10^{-4}\sin\sigma-3.027\cdot10^{-5}\cos\sigma\right)}{\rho}_1^2 \nonumber\\ &-7.761\cdot10^{-5}{\rho}_2^3 -0.1468\rho_2,\nonumber\\
    \dot{\vartheta}_2=&3.538\cdot10^{-4}{\rho}_2^2 -3.462\cdot10^{-7}\rho_1^2- \nonumber\\
    & {\rho}_1^2 {\left(9.135\cdot10^{-4}\cos\sigma+3.027\cdot10^{-5}\sin\sigma\right)}/{\rho_2 }\nonumber\\
    & +45.34\label{eq:red-auto-divider},
\end{align}
where $\sigma=2\vartheta_1-\vartheta_2$. Note the $\rho_1^2$ terms in the vector fields for $\rho_2$ and $\vartheta_2$ are direct results of the near 1:2 internal resonance. Importantly, these terms imply that the response of the second mode will not be trivial if the first mode $\rho_1$ is activated, namely, the energy of the first mode can be transferred to the second mode. In contrast, $\rho_1\equiv0$ is a solution family to the first mode for any nontrivial $\rho_2$ (see the first sub-equation of~\eqref{eq:red-auto-divider}). This implies that the dynamics along the first mode stays trivial if it is not activated initially, and the energy of the second mode cannot be transferred to the first mode in this case. We have observed similar phenomenon in the pendulum-slider example with a near 1:3 internal resonance. 

We again use the invariance error measure~\eqref{eq:auto-error-measure} for selecting appropriate expansion orders. As detailed in Appendix~\ref{sec:appendix-divider}, expansions at $\mathcal{O}(3)$, $\mathcal{O}(5)$ and $\mathcal{O}(7)$ will be required to meet error tolerance 0.01 in the domains $\varrho\leq14$, $\varrho\leq30$, and $\varrho\leq50$, respectively. In Appendix~\ref{sec:appendix-divider}, we also validate that the error tolerance is acceptable by comparing transient responses of the system obtained by both SSM-based ROM predictions and direct numerical integration to the full system.

We now apply a harmonic forcing at the midpoint of the horizontal beam, as illustrated in Fig.~\ref{fig:frequency-divider}, and compute the FRC via SSMTool to illustrate the mechanism of the frequency divider. Here, we restrict $\Omega\approx\omega_2$ such that the second mode is in resonance with the forcing frequency. As we will see, the internal resonance between the first two modes causes energy exchange between them, resulting in finite vibration of the vertical beam when the horizontal beam is excited.


Recall that $\rho_1=0$ results in a vanishing vector field for the $\rho_1$ variable in~Eq.~\eqref{eq:red-auto-divider}. Indeed, this also holds in the non-autonomous (forced) setting, where the reduced dynamics~\eqref{eq:red-auto-divider} is simply modified by adding $\Omega t$-dependent terms to the vector field of $(\rho_2,\vartheta_2)$ variables since only the second mode is in resonance with the forcing~\cite{part-i,part-ii}. In addition, we can factor out these $\Omega t$-dependent terms with proper coordinate transformations ($\vartheta_1=\theta_1+0.5\Omega t$ and $\vartheta_2=\theta_2+\Omega t)$, as detailed in ~\cite{part-i,part-ii}.
In the transformed system, we have a family of fixed points with $\rho_1=0$ and $\rho_2(\Omega)\neq0$, some of which turn out to be unstable depending on the forcing frequency $\Omega$. In this example, such a change in stability is accompanied with a secondary solution branch with nontrivial $\rho_1$ bifurcating from the main solution branch. Physically, this nontrivial $\rho_1$ branch, which contains stable solutions, is responsible for the observation of finite amplitude vibrations in the vertical beam at half of the forcing frequency when the horizontal beam is forced near its resonance.

\begin{figure}[!ht]
	\centering
	\includegraphics[width=.5\textwidth]{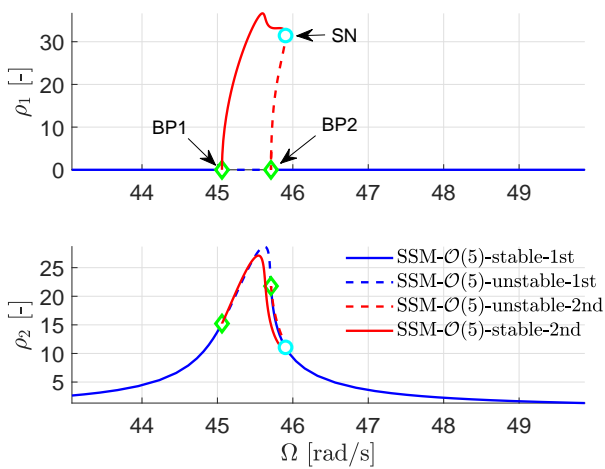}
	\caption{Forced response curves in $(\rho_1,\rho_2)$ of the frequency divider with $F=0.1$. Here and in Fig.~\ref{fig:divider-frc-phy}, the `1st' and `2nd' in the figure legend denote the first (main) and secondary solution branches.}
	\label{fig:divider-frc-rho}
\end{figure}

Similarly to the pendulum slider example, we obtain the FRC directly by analyzing the reduced-dynamics of the 4-dimensional SSM in a normal-form parametrization style, where the periodic orbit is simply given by the solution to a fixed point problem (see~\cite{part-i} for details). Choosing a forcing amplitude of $F=0.1$ in this example, we compute the SSM upto $\mathcal{O}(5)$ via SSMTool and analyze the reduced dynamics on the SSM to obtain the two solution branches, as shown in Fig.~\ref{fig:divider-frc-rho}. We perform the numerical continuation via the~\texttt{ep}-toolbox of~\textsc{coco}~\cite{dankowicz2013recipes} for detecting the branch points along the main solution branch with $\rho_1\equiv0$ and switching to the secondary solution branch. 

\begin{figure}[!ht]
	\centering
	\includegraphics[width=.45\textwidth]{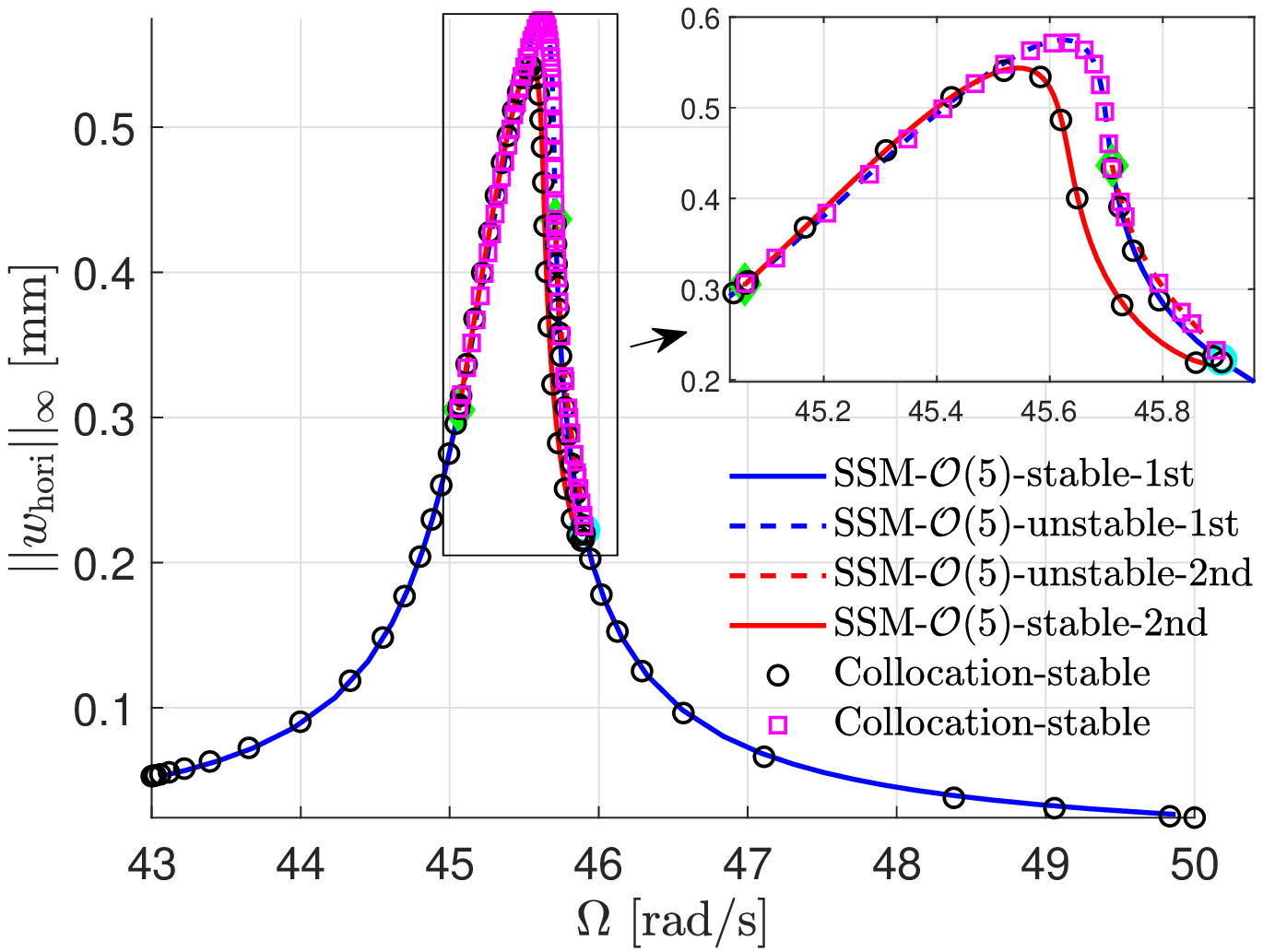}
	\includegraphics[width=.45\textwidth]{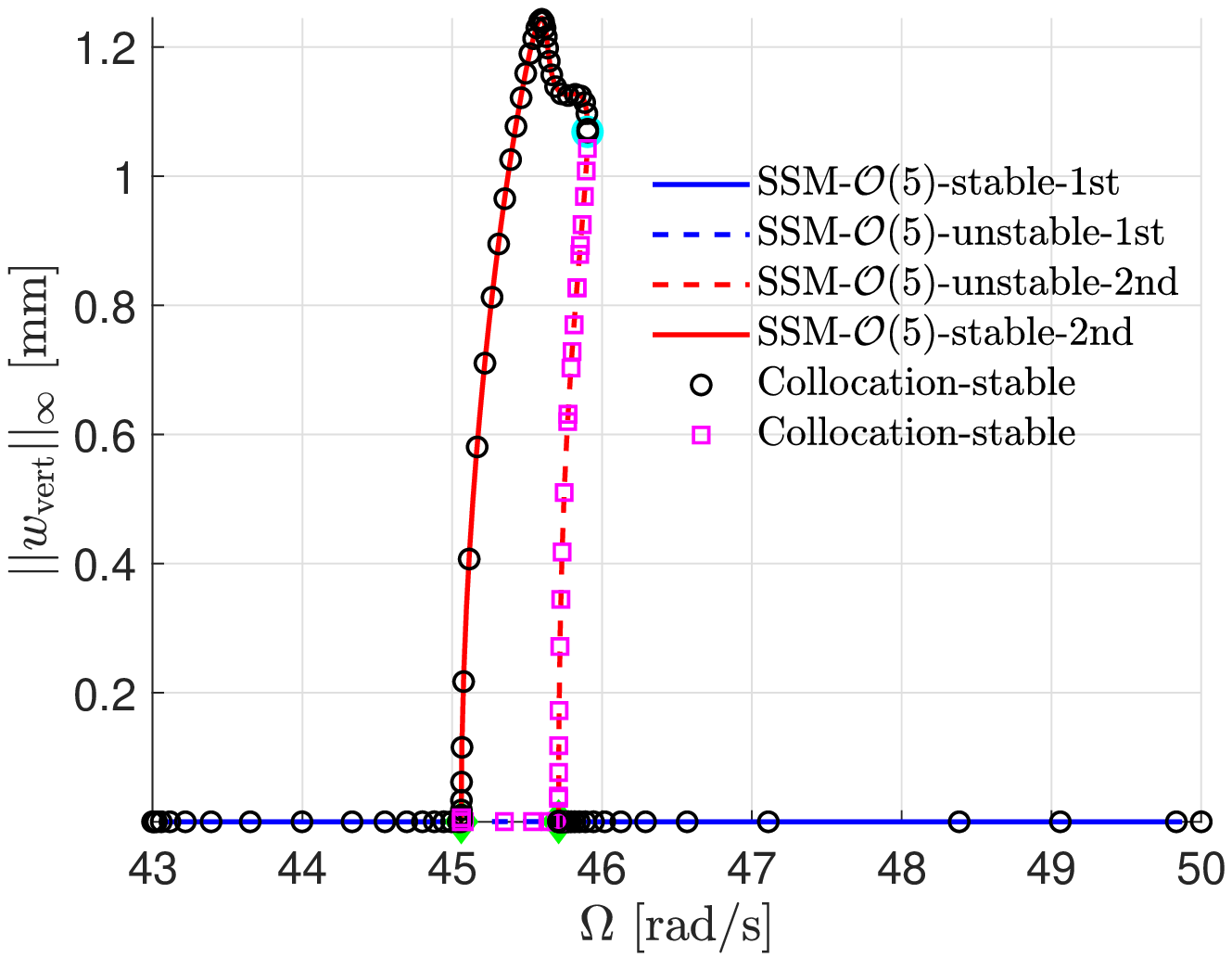}
	\caption{Forced response curves in vibration amplitudes of the transverse deflections of the horizontal (upper panel) and the vertical (lower panel) beams. Here $w_\mathrm{hori}$ and $w_\mathrm{vert}$ are transverse deflection at the midpoint the the horizontal and vertical beams of the divider.}
	\label{fig:divider-frc-phy}
\end{figure}
As shown in Fig.~\ref{fig:divider-frc-rho}, two branch points BP1 and BP2 are detected on the main branch (blue curve) at $\Omega_{\mathrm{BP1}}=45.06$ and $\Omega_{\mathrm{BP2}}=45.71$. The solution along the main branch becomes unstable between the two branch points, i.e., for $\Omega\in[\Omega_{\mathrm{BP1}},\Omega_{\mathrm{BP2}}]$. We switch branches at BP1 to continue the solution along the secondary branch (red curve) with $\rho_1\neq0$. Indeed, the peak for the polar amplitude $\rho_1$ along the secondary branch is higher than that for $\rho_2$. Given that only second mode is excited by external forcing, this indicates transfer of energy from the second mode to the first mode. We further obtain a saddle-node bifurcation (SN) along the secondary branch which results in change of solution stability along the branch. Finally, the secondary branch merges with the main branch at the other branch point (BP2). Note that the fundamental frequency of a periodic response on the secondary branch is $\Omega/2$ as the period of the response is doubled when we switch from the main branch to the secondary branch. Interestingly, we obtain stable solutions along the secondary branch in the frequency range that main branch's response becomes unstable. This guarantees the experimental observability of the secondary branch's response for $\Omega\in[\Omega_{\mathrm{BP1}},\Omega_{\mathrm{BP2}}]$ in contrast to the response on along the main branch.

To demonstrate the accuracy and efficiency of the predictions from our SSM-based ROM, we calculate the FRC of the full system~\eqref{eq:divider-diff}-\eqref{eq:divider-g} using the \texttt{po}-toolbox of \textsc{coco}. We discretize a periodic orbit along the main branch using a mesh with 10 time intervals, 5 base points and 4 collocation nodes in each interval. We set the maximum continuation step size and allowed residual for the predictor in the continuation algorithm to be 100 and 1000. Two periodic-doubling bifurcation points are detected along the main branch, which match well with the two branch points shown in Fig.~\ref{fig:divider-frc-rho}. We then switch to the secondary branch of periodic obits with doubled time period. To maintain the mesh refinement upto switching to the secondary branch, we now discretize a periodic orbit using a mesh with 20 intervals. Accordingly, we also increase the maximum allowed residual for predictor step to 10,000. We carefully choose these solver settings to ensure that the computational time of the collocation method using \texttt{po} is reasonable.

As shown in Fig.~\ref{fig:divider-frc-phy}, the FRC of the full system match with the ones predicted by the SSM-based ROM. Here, the computations of FRCs using \textsc{coco} are performed on a remote Intel Xeon E3-1585Lv5 processor (3.0-3.7 GHz) on the ETH Euler cluster. The FRC computation for full system using the collocation method took about 27 hours. In contrast, the computation of FRCs using $\mathcal{O}(5)$ SSM-reduction, performed on an Intel(R) Core(TM) i7-6700HQ processor (2.60 GHz) of a laptop, took only 10 seconds.

One may note the similarity between the FRC of $||w_\mathrm{hori}||_\infty$ and the one of $\rho_2$, and the similarity between the FRC of $||w_\mathrm{vert}||_\infty$ and the one of $\rho_1$. Such similarities can be explained by the fact that midpoints of the horizontal and vertical beam are (nearly) located at the peak response of the second and first bending modes, respectively (cf.~Fig.~\ref{fig:frequency-divider-modeshape}).

\section{Conclusion}
\label{sec:conclusion}
Using the theory of spectral submanifolds (SSMs), we have developed reduced-order models (ROMs) for nonlinear mechanical systems with configuration constraints and possible internal resonances. We have used these SSM-based ROMs to extract damped backbone curves, predict the transient response of the unforced system, extract the periodic forced response curves (FRCs) under external harmonic forcing, and predict the bifurcations of steady-state periodic response under the variation of excitation frequency and amplitude. We have demonstrated our reduction technique over several examples from rigid as well as flexible multibody dynamics, including a chain of pendulums linked to a slider and a finite-element model of a frequency divider featuring internal resonance.

\begin{sloppypar}
We also made an open-source implementation of all methods and examples discussed here in the software package SSMTool~\cite{ssmtool21}, which supports automated computations of SSMs and their associated reduced dynamics. SSMTool also enables the detection and analysis of quasi-periodic responses of constrained systems as well. Typically, quasi-periodic solutions are born out of Hopf bifurcations along an FRC, as shown in~\cite{part-ii}. Furthermore, we have illustrated a reformulation technique that transforms systems with non-polynomial nonlinearities into systems with only polynomial nonlinearties, enabling the application of SSMTool to systems with non-polynomial nonlinearities. We further introduced invariance error measure that enables the a posteriori estimation of the convergence domain of SSM approximations and allows reliable reduction without the need for  validations based on full system simulations.

In the forced setting, we have restricted to the leading-order approximation of the non-autonomous part of the SSMs. For larger forcing amplitudes, however, higher order non-autonomous terms are important as we observed high invariance errors and differences with respect to the full system simulations (cf.~upper-left panel of Fig.~\ref{fig:3Dos-x1-frc}). Flexible multibody systems undergoing combination of overall motion and large deformations were not considered in this work but are of great interest. All these development are currently underway and will be reported elsewhere. Further areas of application not considered in the current work include piezoelectric structures~\cite{lazarus2012finite}, incompressible flows~\cite{montlaur2012high}, computational electromagnetics~\cite{cortes2018systems}, and power grids~\cite{liu2019solving}. 
\end{sloppypar}

\appendix
\section{Proof of Theorem~\ref{th:dae2ode}}
\label{sec:appendix-dae2ode}
This proof uses Maggi's reformulation of constrained systems, as presented in~\cite{laulusa2008review}. With an initial condition $\boldsymbol{x}_0$ that satisfies $\boldsymbol{g}(\boldsymbol{x}_0)=0$, the algebraic constraints $\boldsymbol{g}=\boldsymbol{0}$ in~\eqref{eq:eom-second-dae} can be extended by $\dot{\boldsymbol{g}}+\alpha\boldsymbol{g}=\boldsymbol{0}$, or equivalently
\begin{equation}
\label{eq:g-dot}
    \boldsymbol{G}(\boldsymbol{x})\dot{\boldsymbol{x}}+\alpha\boldsymbol{g}(\boldsymbol{x})=0,
\end{equation}
where $\alpha\in\mathbb{R}^+$ is a stabilization parameter. We introduce $n-n_\mathrm{c}$ \emph{kinematic characteristics}, which are also called \emph{generalized speeds}~\cite{laulusa2008review}, denoted by $\boldsymbol{e}$, that satisfy
\begin{equation}
\label{eq:define-e}
    \begin{pmatrix}\boldsymbol{G}_{(n_\mathrm{c}\times n)}(\boldsymbol{x})\\\widecheck{\boldsymbol{G}}_{((n-n_\mathrm{c})\times n)}(\boldsymbol{x})\end{pmatrix}\dot{\boldsymbol{x}}+\begin{pmatrix}\alpha\boldsymbol{g}_{(n_\mathrm{c})}(\boldsymbol{x})\\\boldsymbol{0}_{(n-n_\mathrm{c})}\end{pmatrix}=\begin{pmatrix}\boldsymbol{0}_{(n_\mathrm{c})}\\\boldsymbol{e}_{(n-n_\mathrm{c})}
    \end{pmatrix}.
\end{equation}
Here and below, the bracketed subscripts denote the size of a matrix or vector. Recall that $\boldsymbol{G}\in\mathbb{R}^{n_\mathrm{c}\times n}$ and it is of full rank. Thus one can choose $\widecheck{\boldsymbol{G}}$ such that the matrix formed by $\boldsymbol{G}$ and $\widecheck{\boldsymbol{G}}$ defines an invertible linear transformation. With this inverse
\begin{equation}
    \begin{pmatrix}\boldsymbol{G}_{(n_\mathrm{c}\times n)}(\boldsymbol{x})\\\widecheck{\boldsymbol{G}}_{((n-n_\mathrm{c})\times n)}(\boldsymbol{x})\end{pmatrix}^{-1}=\begin{pmatrix}\widecheck{\boldsymbol{\Gamma}}_{(n_\mathrm{c}\times n)}(\boldsymbol{x}) & \boldsymbol{\Gamma}_{(n\times (n-n_\mathrm{c}))}(\boldsymbol{x})\end{pmatrix},
\end{equation}
the generalized velocities $\dot{\boldsymbol{x}}$ can be readily expressed in terms of the kinematic characteristics as
\begin{equation}
\label{eq:xdot}
    \dot{\boldsymbol{x}}=\boldsymbol{\Gamma}(\boldsymbol{x})\boldsymbol{e}-\alpha\widecheck{\boldsymbol{\Gamma}}(\boldsymbol{x})\boldsymbol{g}(\boldsymbol{x}).
\end{equation}
From now on, we do not show the $\boldsymbol{x}$ dependencies and the sizes of matrices and vectors explicitly for compactness. 
We note that
\begin{gather}
    \begin{pmatrix}\boldsymbol{G}\\\widecheck{\boldsymbol{G}}\end{pmatrix}\begin{pmatrix}\widecheck{\boldsymbol{\Gamma}} & \boldsymbol{\Gamma}\end{pmatrix}=\begin{pmatrix}\boldsymbol{G}\widecheck{\boldsymbol{\Gamma}} & \boldsymbol{G}\boldsymbol{\Gamma}\\\widecheck{\boldsymbol{G}}\widecheck{\boldsymbol{\Gamma}} & \widecheck{\boldsymbol{G}}\boldsymbol{\Gamma}\end{pmatrix}=\begin{pmatrix}\boldsymbol{I} & \boldsymbol{0} \\\boldsymbol{0} & \boldsymbol{I}\end{pmatrix},\nonumber\\
    \begin{pmatrix}\widecheck{\boldsymbol{\Gamma}} & \boldsymbol{\Gamma}\end{pmatrix}\begin{pmatrix}\boldsymbol{G}\\\widecheck{\boldsymbol{G}}\end{pmatrix}=\widecheck{\boldsymbol{\Gamma}}\boldsymbol{G}+\boldsymbol{\Gamma}\widecheck{\boldsymbol{G}}=\boldsymbol{I}.\label{eq:inverse-formula}
\end{gather}
Since $\boldsymbol{G}\boldsymbol{\Gamma}=\boldsymbol{0}$, $\boldsymbol{\Gamma}$ spans the null space of the constraint matrix. Differentiating~\eqref{eq:xdot}, the generalized accelerations $\ddot{\boldsymbol{x}}$ can also be expressed in terms of
the kinematic characteristics
\begin{equation}
\label{eq:xddot}
    \ddot{\boldsymbol{x}}=\boldsymbol{\Gamma}\dot{\boldsymbol{e}}+\dot{\boldsymbol{\Gamma}}\boldsymbol{e}-\alpha\widecheck{\boldsymbol{\Gamma}}\boldsymbol{G}\dot{\boldsymbol{x}}-\alpha\dot{\widecheck{\boldsymbol{\Gamma}}}\boldsymbol{g}.
\end{equation}
Substituting~\eqref{eq:xdot} and~\eqref{eq:xddot} into the differential equations in~\eqref{eq:eom-second-dae}, premultiplying the equations by $\boldsymbol{\Gamma}^\mathrm{T}$ and utilizing the fact that $\boldsymbol{G}\boldsymbol{\Gamma}=\boldsymbol{0}$, we obtain governing equations expressed in terms of the kinematic characteristics as follows
\begin{align}
&\boldsymbol{\Gamma}^\mathrm{T}\boldsymbol{M}\boldsymbol{\Gamma}\dot{\boldsymbol{e}}
+(\boldsymbol{\Gamma}^\mathrm{T}\boldsymbol{C}-\alpha\boldsymbol{\Gamma}^\mathrm{T}\boldsymbol{M}\widecheck{\boldsymbol{\Gamma}}\boldsymbol{G})\dot{\boldsymbol{x}}+\boldsymbol{\Gamma}^\mathrm{T}\boldsymbol{K}\boldsymbol{x}\nonumber\\
&\quad +\boldsymbol{\Gamma}^\mathrm{T}\boldsymbol{f}+\boldsymbol{\Gamma}^\mathrm{T}\boldsymbol{M}\dot{\boldsymbol{\Gamma}}\boldsymbol{e}-\alpha\boldsymbol{\Gamma}^\mathrm{T}\boldsymbol{M}\dot{\widecheck{\boldsymbol{\Gamma}}}\boldsymbol{g}=\epsilon \boldsymbol{\Gamma}^\mathrm{T}\boldsymbol{f}^{\mathrm{ext}}.\label{eq:eom-e}
\end{align}
Therefore, the original equations of motion can be reformulated without separate equations for constraints as follows
\begin{align}
\label{eq:odesmaggi}
    &\begin{pmatrix}\boldsymbol{I} &\boldsymbol{0}\\ \boldsymbol{\Gamma}^\mathrm{T}\boldsymbol{C}-\alpha\boldsymbol{\Gamma}^\mathrm{T}\boldsymbol{M}\widecheck{\boldsymbol{\Gamma}}\boldsymbol{G} & \boldsymbol{\Gamma}^\mathrm{T}\boldsymbol{M}\boldsymbol{\Gamma}
    \end{pmatrix}\begin{pmatrix}\dot{\boldsymbol{x}}\\\dot{\boldsymbol{e}}\end{pmatrix}= \begin{pmatrix}
    -\alpha\widecheck{\boldsymbol{\Gamma}}\boldsymbol{G}_0 & \boldsymbol{\Gamma}\\
    -\boldsymbol{\Gamma}^\mathrm{T}\boldsymbol{K} & \boldsymbol{0}
    \end{pmatrix}\begin{pmatrix}
    \boldsymbol{x}\\\boldsymbol{e}
    \end{pmatrix}+\nonumber\\
     & \begin{pmatrix}
    -\alpha\widecheck{\boldsymbol{\Gamma}}\boldsymbol{g}_{\mathrm{nl}}\\ -\boldsymbol{\Gamma}^\mathrm{T}\boldsymbol{f}-\boldsymbol{\Gamma}^\mathrm{T}\boldsymbol{M}\dot{\boldsymbol{\Gamma}}\boldsymbol{e}+\alpha\boldsymbol{\Gamma}^\mathrm{T}\boldsymbol{M}\dot{\widecheck{\boldsymbol{\Gamma}}}\boldsymbol{g}
    \end{pmatrix}+\epsilon\begin{pmatrix}
    \boldsymbol{0}\\\boldsymbol{\Gamma}^\mathrm{T}\boldsymbol{f}^{\mathrm{ext}}
    \end{pmatrix}.
\end{align}
In the case of linear configuration constraints, we have $\boldsymbol{G}=\boldsymbol{G}_0$ and we can choose a constant matrix $\widecheck{\boldsymbol{G}}$. Then both $\boldsymbol{\Gamma}$ and $\widecheck{\boldsymbol{\Gamma}}$ are constant matrices as well. Consequently, the two coefficient matrices in the above system of equations are also constant.

\section{Spectrum of linear ODEs}
\label{sec:spectrum-ode}
Next, we study the spectrum of the system~\eqref{eq:odesmaggi} linearized around the origin. Let the constant parts of $\boldsymbol{\Gamma}$ and $\widecheck{\boldsymbol{\Gamma}}$ be $\boldsymbol{\Gamma}_0$ and $\widecheck{\boldsymbol{\Gamma}}_0$. The equations of motion of the linearized system are
\begin{align}
    &\begin{pmatrix}\boldsymbol{I} &\boldsymbol{0}\\ \boldsymbol{\Gamma}_0^\mathrm{T}\boldsymbol{C}-\alpha\boldsymbol{\Gamma}_0^\mathrm{T}\boldsymbol{M}\widecheck{\boldsymbol{\Gamma}}_0\boldsymbol{G}_0 & \boldsymbol{\Gamma}_0^\mathrm{T}\boldsymbol{M}\boldsymbol{\Gamma}_0
    \end{pmatrix}\begin{pmatrix}\dot{\boldsymbol{x}}\\\dot{\boldsymbol{e}}\end{pmatrix}\nonumber\\
    &\qquad=\begin{pmatrix}
    -\alpha\widecheck{\boldsymbol{\Gamma}}_0\boldsymbol{G}_0 & \boldsymbol{\Gamma}_0\\
    -\boldsymbol{\Gamma}_0^\mathrm{T}\boldsymbol{K} & \boldsymbol{0}
    \end{pmatrix}\begin{pmatrix}
    \boldsymbol{x}\\\boldsymbol{e}
    \end{pmatrix}.\label{eq:lin-x-e}
\end{align}
We assume the linearized system has an eigensolution as follows
\begin{equation}
    \begin{pmatrix}
    \boldsymbol{x}\\\boldsymbol{e}
    \end{pmatrix}=\begin{pmatrix}
    \widehat{\boldsymbol{x}}\\\widehat{\boldsymbol{e}}
    \end{pmatrix}e^{\lambda t}.
\end{equation}
Substituting this solution into~\eqref{eq:lin-x-e} yields
\begin{gather}
    \lambda\widehat{\boldsymbol{x}}=-\alpha\widecheck{\boldsymbol{\Gamma}}_0\boldsymbol{G}_0\widehat{\boldsymbol{x}}+\boldsymbol{\Gamma}_0\widehat{\boldsymbol{e}},\label{eq:hatx}\\
    \lambda(\boldsymbol{\Gamma}_0^\mathrm{T}\boldsymbol{C}-\alpha\boldsymbol{\Gamma}_0^\mathrm{T}\boldsymbol{M}\widecheck{\boldsymbol{\Gamma}}_0\boldsymbol{G}_0)\widehat{\boldsymbol{x}}+\lambda\boldsymbol{\Gamma}_0^\mathrm{T}\boldsymbol{M}\boldsymbol{\Gamma}_0\widehat{\boldsymbol{e}}=-\boldsymbol{\Gamma}_0^\mathrm{T}\boldsymbol{K}\hat{\boldsymbol{x}}\label{eq:hate}.
\end{gather}
We obtain from~\eqref{eq:hatx} that $ \boldsymbol{\Gamma}_0\widehat{\boldsymbol{e}}=\lambda\widehat{\boldsymbol{x}}+\alpha\widecheck{\boldsymbol{\Gamma}}_0\boldsymbol{G}_0\widehat{\boldsymbol{x}}$. We substitute the expression of $\boldsymbol{\Gamma}_0\widehat{\boldsymbol{e}}$ into~\eqref{eq:hate} to obtain
\begin{equation}
\label{eq:lamd-1}
    \boldsymbol{\Gamma}_0^\mathrm{T}(\lambda^2\boldsymbol{M}\widehat{\boldsymbol{x}}+\lambda\boldsymbol{C}\widehat{\boldsymbol{x}}+\boldsymbol{K}\widehat{\boldsymbol{x}})=\boldsymbol{0}.
\end{equation}
Given that $\widecheck{\boldsymbol{\Gamma}}_0\boldsymbol{G}_0+\boldsymbol{\Gamma}_0\widecheck{\boldsymbol{G}}_0=\boldsymbol{I}$ (see~\eqref{eq:inverse-formula}), we obtain from~\eqref{eq:hatx} that
\begin{equation}
\label{eq:lamd-2}
    (\lambda+\alpha)\widehat{\boldsymbol{x}}=\alpha\boldsymbol{\Gamma}_0\widecheck{\boldsymbol{G}}_0\widehat{\boldsymbol{x}}+\boldsymbol{\Gamma}_0\widehat{\boldsymbol{e}}.
\end{equation}

Now we analyze the solutions to~\eqref{eq:lamd-1} and~\eqref{eq:lamd-2} depending on whether $\lambda+\alpha$ is zero or not. In the case that $\lambda+\alpha\neq0$, we infer from~\eqref{eq:lamd-2} that $\widehat{\boldsymbol{x}}\in\mathrm{Range}(\boldsymbol{\Gamma}_0)$ and then we can introduce $\tilde{\boldsymbol{x}}\in\mathbb{C}^{n-n_\mathrm{c}}$ such that
\begin{equation}
\label{eq:widehatx}
    \widehat{\boldsymbol{x}}=\boldsymbol{\Gamma}_0\tilde{\boldsymbol{x}}.
\end{equation}
Substituting~\eqref{eq:widehatx} into~\eqref{eq:lamd-1}, we obtain
\begin{equation}
    (\lambda^2\boldsymbol{\Gamma}_0^\mathrm{T}\boldsymbol{M}\boldsymbol{\Gamma}_0+\lambda\boldsymbol{\Gamma}_0^\mathrm{T}\boldsymbol{C}\boldsymbol{\Gamma}_0+\boldsymbol{\Gamma}_0^\mathrm{T}\boldsymbol{K}\boldsymbol{\Gamma}_0)\tilde{\boldsymbol{x}}=\boldsymbol{0}.
\end{equation}
Therefore, we obtain $2(n-n_\mathrm{c})$ eigenvalues from a system with mass matrix $\boldsymbol{\Gamma}_0^\mathrm{T}\boldsymbol{M}\boldsymbol{\Gamma}_0$, damping matrix $\boldsymbol{\Gamma}_0^\mathrm{T}\boldsymbol{C}\boldsymbol{\Gamma}_0$ and stiffness matrix $\boldsymbol{\Gamma}_0^\mathrm{T}\boldsymbol{K}\boldsymbol{\Gamma}_0$.

In the special of $\lambda+\alpha=0$, we obtain from~\eqref{eq:lamd-2} that $\widehat{\boldsymbol{e}}=-\alpha\widecheck{\boldsymbol{G}}_0\widehat{\boldsymbol{x}}$. Given that $\boldsymbol{G}_0\boldsymbol{\Gamma}_0=\boldsymbol{0}$ (cf.~\eqref{eq:inverse-formula}), we infer from~\eqref{eq:lamd-1} that
\begin{equation}
    (\alpha^2\boldsymbol{M}-\alpha\boldsymbol{C}+\boldsymbol{K})\widehat{\boldsymbol{x}}\in\mathrm{range}(\boldsymbol{G}_0^\mathrm{T}).
\end{equation}
Since $\alpha\in\mathbb{R}^+$ is arbitrary and can be chosen different from any eigenvalues of the system ($\boldsymbol{M},\boldsymbol{C},\boldsymbol{K}$), the coefficient matrix $(\alpha^2\boldsymbol{M}-\alpha\boldsymbol{C}+\boldsymbol{K})$ is invertible, and then there are $m$ linearly independent solutions of $\widehat{\boldsymbol{x}}$.

In summary, the eigenvalues and eigenvectors of the linearized system~\eqref{eq:lin-x-e} are given as below
\begin{enumerate}[label=(\alph*)]
\item $i=1,\cdots,2(n-n_\mathrm{c})$:
\begin{gather}
    (\lambda_i^2\boldsymbol{\Gamma}_0^\mathrm{T}\boldsymbol{M}\boldsymbol{\Gamma}_0+\lambda\boldsymbol{\Gamma}_0^\mathrm{T}\boldsymbol{C}\boldsymbol{\Gamma}_0+\boldsymbol{\Gamma}_0^\mathrm{T}\boldsymbol{K}\boldsymbol{\Gamma}_0)\tilde{\boldsymbol{x}}_i=\boldsymbol{0},\nonumber\\
    \widehat{\boldsymbol{x}}_i=\boldsymbol{\Gamma}_0\tilde{\boldsymbol{x}}_i,\,\widehat{\boldsymbol{e}}_i=(\lambda_i+\alpha)\tilde{\boldsymbol{x}}_i-\alpha\widecheck{\boldsymbol{G}}\widehat{\boldsymbol{x}}_i,\label{eq:group-a}
\end{gather}
\item $i=2(n-n_\mathrm{c})+1,\cdots,2(n-n_\mathrm{c})+n_\mathrm{c}$: $\lambda_{i}=-\alpha$ and
\begin{equation}
\label{eq:group-b}
    \begin{pmatrix}
    \widehat{\boldsymbol{x}}_{2(n-n_\mathrm{c})+1} & \cdots & \widehat{\boldsymbol{x}}_{2(n-n_\mathrm{c})+n_\mathrm{c}}\\
    \widehat{\boldsymbol{e}}_{2(n-n_\mathrm{c})+1} & \cdots & \widehat{\boldsymbol{e}}_{2(n-n_\mathrm{c})+n_\mathrm{c}}
    \end{pmatrix}=\begin{pmatrix}
    \boldsymbol{G}_0^\mathrm{T}\\-\alpha\widecheck{\boldsymbol{G}}_0\boldsymbol{G}_0^\mathrm{T}
    \end{pmatrix}.
\end{equation}
\end{enumerate}

\section{Spectrum of linear DAEs}
\label{sec:spectrum-dae}
The linear part of~\eqref{eq:eom-second-dae} in the limit $\epsilon=0$ is given by
\begin{equation}
\label{eq:eom-second-dae-lin}
\boldsymbol{M}\ddot{\boldsymbol{x}}+\boldsymbol{C}\dot{\boldsymbol{x}}+\boldsymbol{K}\boldsymbol{x}+\boldsymbol{G}_0^\mathrm{T}\boldsymbol{\mu}=\boldsymbol{0}, \quad\boldsymbol{G}_0\boldsymbol{x}=\boldsymbol{0}.
\end{equation}
Consider an eigensolution
\begin{equation}
    \begin{pmatrix}
    \boldsymbol{x}\\\boldsymbol{\mu}
    \end{pmatrix}=\begin{pmatrix}
    \widehat{\boldsymbol{x}}\\\widehat{\boldsymbol{\mu}}
    \end{pmatrix}e^{\lambda t}.
\end{equation}
Substituting the solution above into~\eqref{eq:eom-second-dae-lin} yields
\begin{equation}
\label{eq:eom-second-dae-lin-eig}
\lambda^2\boldsymbol{M}\widehat{\boldsymbol{x}}+\lambda\boldsymbol{C}\widehat{\boldsymbol{x}}+\boldsymbol{K}\widehat{\boldsymbol{x}}+\boldsymbol{G}_0^\mathrm{T}\widehat{\boldsymbol{\mu}}=\boldsymbol{0}, \quad\boldsymbol{G}_0\widehat{\boldsymbol{x}}=\boldsymbol{0}.
\end{equation}
Since $\widehat{\boldsymbol{x}}\in\mathrm{kernel}(\boldsymbol{G}_0)$ and $\boldsymbol{G}_0$ is of full rank, we can find $\boldsymbol{\Gamma}_0\in\mathbb{R}^{n\times(n-n_\mathrm{c})}$ (cf.~\eqref{eq:inverse-formula}) such that $\boldsymbol{G}_0\boldsymbol{\Gamma}_0=\boldsymbol{0}$. Then we can introduce $\tilde{\boldsymbol{x}}\in\mathbb{C}^{n-n_\mathrm{c}}$ such that $\widehat{\boldsymbol{x}}=\boldsymbol{\Gamma}_0\tilde{\boldsymbol{x}}$ and then the second sub-equation in~\eqref{eq:eom-second-dae-lin-eig} is satisfied. We substitute $\widehat{\boldsymbol{x}}=\boldsymbol{\Gamma}_0\tilde{\boldsymbol{x}}$ into the first sub-equation and premultiply the equation by $\boldsymbol{\Gamma}_0^\mathrm{T}$ to obtain
\begin{equation}
(\lambda^2\boldsymbol{\Gamma}_0^\mathrm{T}\boldsymbol{M}\boldsymbol{\Gamma}_0+\lambda\boldsymbol{\Gamma}_0^\mathrm{T}\boldsymbol{C}\boldsymbol{\Gamma}_0+\boldsymbol{\Gamma}_0^\mathrm{T}\boldsymbol{K}\boldsymbol{\Gamma}_0)\tilde{\boldsymbol{x}}=\boldsymbol{0}.
\end{equation}
In the case of $\tilde{\boldsymbol{x}}\neq\boldsymbol{0}$, the spectrum of the equation above is composed of $2(n-n_\mathrm{c})$ eigenvalues from a system with mass matrix $\boldsymbol{\Gamma}_0^\mathrm{T}\boldsymbol{M}\boldsymbol{\Gamma}_0$, damping matrix $\boldsymbol{\Gamma}_0^\mathrm{T}\boldsymbol{C}\boldsymbol{\Gamma}_0$ and stiffness matrix $\boldsymbol{\Gamma}_0^\mathrm{T}\boldsymbol{K}\boldsymbol{\Gamma}_0$ (cf.~\eqref{eq:group-a}).

In the case of $\tilde{\boldsymbol{x}}=\boldsymbol{0}$, we have $\widehat{\boldsymbol{x}}=\boldsymbol{0}$ and~\eqref{eq:eom-second-dae-lin-eig} becomes
\begin{equation}
\label{eq:lamd-inf-1}
\lambda^2\boldsymbol{M}\widehat{\boldsymbol{x}}+\lambda\boldsymbol{C}\widehat{\boldsymbol{x}}+\boldsymbol{G}_0^\mathrm{T}\widehat{\boldsymbol{\mu}}=\boldsymbol{0}.
\end{equation}
Provided that $\lambda\in\mathbb{C}$, the equation above is reduced to $\boldsymbol{G}_0^\mathrm{T}\widehat{\boldsymbol{\mu}}=\boldsymbol{0}$, from which we obtain $\boldsymbol{\mu}=\boldsymbol{0}$ because $\boldsymbol{G}_0$ is of full rank. However, $(\widehat{\boldsymbol{x}},\widehat{\boldsymbol{\mu}})=(\boldsymbol{0},\boldsymbol{0})$ is not an eigenvector. We can rewrite~\eqref{eq:lamd-inf-1} as $\boldsymbol{G}_0^\mathrm{T}\widehat{\boldsymbol{\mu}}/\lambda^2=\boldsymbol{0}$, from which we infer that $\widehat{\boldsymbol{\mu}}$ can be arbitrary non-zeros when $|\lambda|=\infty$. So the system has eigenvalues with infinite magnitude. 

Next we explore how many such eigenvalues of infinite magnitude the system has. We note that~\eqref{eq:eom-second-dae-lin} is equivalent to the linear part of~\eqref{eq:full-first} when $\epsilon=0$. In~\eqref{eq:full-first}, the equations of motion are in the first order form with two matrices $\boldsymbol{A}$ and $\boldsymbol{B}$. Since we have assumed that the the matrix pencil $\lambda\boldsymbol{B}-\boldsymbol{A}$ is regular, namely, $\mathrm{det}(\lambda\boldsymbol{B}-\boldsymbol{A})\neq0$ for some $\lambda\in\mathbb{C}$, the system has $3n_\mathrm{c}$ such infinite eigenvalues~\cite{benner2015numerical}.

In summary, the eigenvalues and eigenvectors of the system~\eqref{eq:eom-second-dae-lin} are given by
\begin{enumerate}[resume,label=(\alph*)]
\item $i=1,\cdots,2(n-n_\mathrm{c})$:
\begin{gather}
(\lambda_i^2\boldsymbol{\Gamma}_0^\mathrm{T}\boldsymbol{M}\boldsymbol{\Gamma}_0+\lambda\boldsymbol{\Gamma}_0^\mathrm{T}\boldsymbol{C}\boldsymbol{\Gamma}_0+\boldsymbol{\Gamma}_0^\mathrm{T}\boldsymbol{K}\boldsymbol{\Gamma}_0)\tilde{\boldsymbol{x}}_i=\boldsymbol{0},\nonumber\\
\widehat{\boldsymbol{x}}_i=\boldsymbol{\Gamma}_0\tilde{\boldsymbol{x}}_i,\quad\lambda_i^2\boldsymbol{M}\widehat{\boldsymbol{x}}_i+\lambda_i\boldsymbol{C}\widehat{\boldsymbol{x}}_i+\boldsymbol{K}\widehat{\boldsymbol{x}}_i+\boldsymbol{G}_0^\mathrm{T}\widehat{\boldsymbol{\mu}}_i=\boldsymbol{0}\label{eq:group-c}
\end{gather}
\item $i=2(n-n_\mathrm{c}),\cdots,2n+n_\mathrm{c}$:
\begin{equation}
\label{eq:group-d}
    |\lambda_i|=\infty,\quad\widehat{\boldsymbol{x}}_i=\boldsymbol{0},\,\widehat{\boldsymbol{\mu}}_i\neq\boldsymbol{0}.
\end{equation}
\end{enumerate}

\section{Index-1 formulation with stabilization}
\label{sec:index-1}
Provided that we have a consistent initial condition $(\boldsymbol{x}_0,\dot{\boldsymbol{x}}_0)$ that satisfies
\begin{equation}
    \boldsymbol{g}(\boldsymbol{x}_0)=\boldsymbol{0},\quad \boldsymbol{G}(\boldsymbol{x}_0)\dot{\boldsymbol{x}}_0=\boldsymbol{0},
\end{equation}
the algebraic constraints $\boldsymbol{g}=\boldsymbol{0}$ in~\eqref{eq:eom-second-dae} can be replaced by
\begin{equation}
\label{eq:g-ddot}
    \ddot{\boldsymbol{g}}+\alpha\dot{\boldsymbol{g}}+\beta\boldsymbol{g}=\boldsymbol{0},
\end{equation}
where $\alpha,\beta\in\mathbb{R}^+$ are stabilization parameters. The differential equations in~\eqref{eq:eom-second-dae}, along with the replaced configuration constraints~\eqref{eq:g-ddot}, can be written in the more compact form
\begin{equation}
\label{eq:second-index1}
    \begin{pmatrix}\boldsymbol{M}&\boldsymbol{G}^\mathrm{T}\\\boldsymbol{G} &\boldsymbol{0}\end{pmatrix}\begin{pmatrix}\ddot{\boldsymbol{x}}\\\boldsymbol{\mu}\end{pmatrix}=\begin{pmatrix}\hat{\boldsymbol{f}}\\\boldsymbol{c}\end{pmatrix},
\end{equation}
where
\begin{gather}
    \hat{\boldsymbol{f}}=\epsilon \boldsymbol{f}^{\mathrm{ext}}(\Omega t)-\boldsymbol{C}\dot{\boldsymbol{x}}-\boldsymbol{K}\boldsymbol{x}-\boldsymbol{f}(\boldsymbol{x},\dot{\boldsymbol{x}}),\\
    \boldsymbol{c}=-\alpha\boldsymbol{G}\dot{\boldsymbol{x}}-\beta\boldsymbol{g}-\dot{\boldsymbol{G}}\dot{\boldsymbol{x}}.
\end{gather}

One can express the vector of Lagrange multipliers $\boldsymbol{\mu}$ from~\eqref{eq:second-index1} as
\begin{equation}
    \boldsymbol{\mu}=-\left(\boldsymbol{G}\boldsymbol{M}^{-1}\boldsymbol{G}^\mathrm{T}\right)^{-1}\left(\boldsymbol{c}-\boldsymbol{G}\boldsymbol{M}^{-1}\hat{\boldsymbol{f}}\right).
\end{equation}
Then the equations of motion without algebraic equations are obtained as
\begin{equation}
    \boldsymbol{M}\ddot{\boldsymbol{x}}=\hat{\boldsymbol{f}}+\boldsymbol{G}^\mathrm{T} \left(\boldsymbol{G}\boldsymbol{M}^{-1}\boldsymbol{G}^\mathrm{T}\right)^{-1}\left(\boldsymbol{c}-\boldsymbol{G}\boldsymbol{M}^{-1}\hat{\boldsymbol{f}}\right),
\end{equation}
which can be rewritten in the more familiar form
\begin{align}
\label{eq:eom-2nd-nolambda}
    &\boldsymbol{M}\ddot{\boldsymbol{x}}+(\boldsymbol{I}-\boldsymbol{P})\left(\boldsymbol{C}\dot{\boldsymbol{x}}+\boldsymbol{K}\boldsymbol{x}+\boldsymbol{f}(\boldsymbol{x},\dot{\boldsymbol{x}})\right)
    \nonumber\\ &\qquad-\boldsymbol{G}^\mathrm{T}\left(\boldsymbol{G}\boldsymbol{M}^{-1}\boldsymbol{G}^\mathrm{T}\right)^{-1}\boldsymbol{c}=\epsilon(\boldsymbol{I}-\boldsymbol{P}) \boldsymbol{f}^{\mathrm{ext}}(\Omega t),
\end{align}
where
\begin{equation}
    \boldsymbol{P}=\boldsymbol{G}^\mathrm{T} \left(\boldsymbol{G}\boldsymbol{M}^{-1}\boldsymbol{G}^\mathrm{T}\right)^{-1}\boldsymbol{G}\boldsymbol{M}^{-1}
\end{equation}
is a projector and $\boldsymbol{I}-\boldsymbol{P}$ is the complement projector of $\boldsymbol{P}$. We note that $\boldsymbol{P}$ is $\boldsymbol{x}$-dependent if the configuration constraints are nonlinear.

\section{Equations of motion for a chain of pendulums}
\label{sec:eom-pend}
The generalized coordinates for the slider and the rods are given by
\begin{equation}
    \mathbf{q}_1=(x_1,y_1),\quad \mathbf{q}_i=(x_i,y_i,\varphi_i),\quad 2\leq i\leq n.
\end{equation}
The kinetic energy of the system is given by $T=\sum_{i=1}^n T_i$ with
\begin{gather}
    T_i=\frac{1}{2}\mathbf{q}_i^\mathrm{T}\mathbf{M}_i\mathbf{q}_i,\quad \mathbf{M}_1=\mathrm{diag}(m_1,m_1),\nonumber\\
    \mathbf{M}_i=\mathrm{diag}(m_i,m_i,J_i),\, 2\leq i\leq n.
\end{gather}
The potential energy of the system is given by $V=\sum_{i=1}^n V_i$ with
\begin{gather}
    V_1=\frac{k_1}{2}x_1^2-m_1gy_1,\quad V_2=\frac{k}{2}\varphi_2^2-m_2gy_2,\nonumber\\ V_i=\frac{k}{2}(\varphi_{i}-\varphi_{i-1})^2-m_igy_i,\,\,3\leq i\leq n-1.
\end{gather}
The configuration constraints are given by
\begin{gather}
    g_1=y_1,\quad g_2=x_2-0.5l\sin\varphi_2-x_1,\nonumber\\ g_3=y_2-0.5l\cos\varphi_2-y_1\label{eq:chain-config-group1}
\end{gather}
and
\begin{gather}
g_{2i}=x_{i+1}-0.5l\sin\varphi_{i+1}-(x_i+0.5l\sin\varphi_i),\nonumber\\ 
g_{2i+1}=y_{i+1}-0.5l\cos\varphi_{i+1}-(y_i+0.5l\cos\varphi_i),\label{eq:chain-config-group2}
\end{gather}
for $2\leq i\leq n-1$.

\begin{sloppypar}
Let $\mathbf{q}=(\mathbf{q}_1,\mathbf{q}_2,\cdots,\mathbf{q}_n)$ and introduce Lagrangian $L(\mathbf{q},\dot{\mathbf{q}})=T(\dot{\mathbf{q}})-V(\mathbf{q})$, The equations of motion is given by
\begin{equation}
    \frac{d}{dt}\left(\frac{\partial L}{\partial\dot{\mathbf{q}}}\right)-\frac{\partial L}{\partial{\mathbf{q}}}+\boldsymbol{G}^{\mathrm{T}}\boldsymbol{\mu}=\boldsymbol{Q},\quad \boldsymbol{g}(\mathbf{q})=\boldsymbol{0}\label{eq:chain-eom},
\end{equation}
where $\boldsymbol{Q}$ denote a vector of generalized forces which are related to damping and external forces. In particular, we have differential equations as follows
\begin{align}
& m_1\ddot{x}_1+k_1x_1-\mu_2=-c_1\dot{x}_1+\epsilon f_1\cos\Omega t,\nonumber\\ & m_1\ddot{y}_1-m_1g+\mu_1-\mu_3=0,\label{eq:chain-1st}
\end{align}
\begin{align}
& m_2\ddot{x}_2+\mu_2-\mu_4=0,\quad m_2\ddot{y}_2-m_2g+\mu_3-\mu_5=0,\nonumber\\
& J_2\ddot{\varphi}_2+k\varphi_2-k(\varphi_3-\varphi_2)-0.5l\cos\varphi_2(\mu_2+\mu_4)\nonumber\\
&\quad+0.5l\sin\varphi_2(\mu_3+\mu_5)=-c\dot{\varphi}_2-c(\dot{\varphi}_2-\dot{\varphi}_3),
\end{align}
\begin{align}
& m_i\ddot{x}_i+\mu_{2(i-1)}-\mu_{2i}=0,\nonumber\\ & m_i\ddot{y}_i-m_ig+\mu_{2i-1}-\mu_{2i+1}=0,\nonumber  \\
& J_i\ddot{\varphi}_i+k(\varphi_i-\varphi_{i-1})-k(\varphi_{i+1}-\varphi_i)\nonumber\\
&\quad-0.5l\cos\varphi_i\mu_{2i}+0.5l\sin\varphi_i\mu_{2i+1}\nonumber\\
& \quad-0.5l\cos\varphi_i\mu_{2(i-1)}+0.5l\sin\varphi_i\mu_{2i-1}\nonumber\\
&\quad=-c(\dot{\varphi}_i-\dot{\varphi}_{i-1})-c(\dot{\varphi}_i-\dot{\varphi}_{i+1}),
\end{align}
for $3\leq i\leq n-1$, and
\begin{align}
& m_n\ddot{x}_n+\mu_{2(n-1)}=0,\quad m_n\ddot{y}_n-m_ng+\mu_{2n-1}=0,\nonumber\\
& J_n\ddot{\varphi}_n+k(\varphi_n-\varphi_{n-1})-0.5l\cos\varphi_n\mu_{2n-2}\nonumber\\
&\quad +0.5l\sin\varphi_n\mu_{2n-1}=-c(\dot{\varphi}_n-\dot{\varphi}_{n-1})\label{eq:chain-last}.
\end{align}
Similarly to the previous example, we shift the rest state to the origin of the phase space such that the origin becomes a fixed point. Specifically, we take
\begin{gather}
    y_i=\hat{y}_i+0.5l+(i-2)l, \,\,2\leq i\leq n, \nonumber\\ \mu_{2i-1}=\hat{\mu}_{2i-1}+\sum_{j=i}^n m_i g,\,\, 1\leq i\leq n.\label{eq:chain-shift}
\end{gather}
We further introduce the auxiliary variables 
\begin{equation}
\label{eq:chain-auxi}
    u_{2i-3}=\sin\varphi_{i},\,\, u_{2i-2}=1-\cos\varphi_i,\quad 2\leq i\leq n,
\end{equation}
to recast the trigonometric terms in~\eqref{eq:chain-1st}-\eqref{eq:chain-last} into polynomials. Substituting~\eqref{eq:chain-shift} and~\eqref{eq:chain-auxi} into~\eqref{eq:chain-1st}-\eqref{eq:chain-last}, the differential equations become
\begin{align}
& m_1\ddot{x}_1+k_1x_1-\mu_2=-c_1\dot{x}_1+f\cos\Omega t,\nonumber\\
& m_1\ddot{{y}}_1+\mu_1-\mu_3=0,\label{eq:chain-1st-v1}
\end{align}
\begin{align}
& m_2\ddot{x}_2+\mu_2-\mu_4=0,\quad m_2\ddot{\hat{y}}_2+\hat{\mu}_3-\hat{\mu}_5=0,\nonumber\\
& J_2\ddot{\varphi}_2+k(2\varphi_2-\varphi_3)-0.5l(1-u_2)(\mu_2+\mu_4)\nonumber\\
&\quad+0.5lu_1(\hat{\mu}_3+\hat{\mu}_5+m_2g+2\sum_{j=3}^nm_jg)\nonumber\\
& \quad =-c(2\dot{\varphi}_2-\dot{\varphi}_3),
\end{align}
\begin{align}
& m_i\ddot{x}_i+\mu_{2(i-1)}-\mu_{2i}=0,\nonumber\\ & m_i\ddot{\hat{y}}_i+\hat{\mu}_{2i-1}-\hat{\mu}_{2i+1}=0,\nonumber  \\
& J_i\ddot{\varphi}_i+k(2\varphi_i-\varphi_{i-1}-\varphi_{i+1})\nonumber\\
& \quad-0.5l(1-u_{2i-2})(\mu_{2i}+\mu_{2(i-1)})\nonumber\\
& \quad+0.5lu_{2i-3}(\mu_{2i+1}+\mu_{2i-1}+m_ig+2\sum_{j=i+1}^nm_jg)\nonumber\\
&\quad=-c(\dot{2\varphi}_i-\dot{\varphi}_{i-1}-\dot{\varphi}_{i+1}), 
\end{align}
for $3\leq i\leq n-1$,
\begin{align}
& m_n\ddot{x}_n+\mu_{2(n-1)}=0,\quad m_n\ddot{\hat{y}}_n+\hat{\mu}_{2n-1}=0,\nonumber\\
& J_n\ddot{\varphi}_n+k(\varphi_n-\varphi_{n-1}) -0.5l(1-u_{2n-2})\mu_{2n-2}\nonumber\\
&\quad +0.5lu_{2n-3}(\hat{\mu}_{2n-1}+m_ng)=-c(\dot{\varphi}_n-\dot{\varphi}_{n-1})\label{eq:chain-last-v1},
\end{align}
and the configuration constraints~\eqref{eq:chain-config-group1}-\eqref{eq:chain-config-group2} become
\begin{gather}
    g_1=y_1,\quad g_2=x_2-0.5lu_1-x_1,\nonumber\\ g_3=\hat{y}_2+0.5lu_2-y_1,\label{eq:chain-config-group1-v1}
\end{gather}
and
\begin{gather}
g_{2i}=x_{i+1}-0.5lu_{2i-1}-(x_i+0.5lu_{2i-3}),\nonumber\\
g_{2i+1}=\hat{y}_{i+1}-\hat{y}_i+0.5l(u_{2i-2}+u_{2i}),\label{eq:chain-config-group2-v2}
\end{gather}
for $2\leq i\leq n-1$.
We have additional equations for the auxiliary variables in the form
\begin{equation}
\label{eq:chain-add}
    \dot{u}_{2i-3}=(1-u_{2i-2})\dot{\varphi}_i,\,\,u_{2i-3}^2-2u_{2i-2}+u_{2i-2}^2=0,\quad 2\leq i\leq n.
\end{equation}
Let $\mathbf{u}=(u_1,\cdots,u_{2n-2})$ and define $\boldsymbol{z}=(\mathbf{q},\dot{\mathbf{q}},\boldsymbol{\mu},\mathbf{u})$, the equations of motion~\eqref{eq:chain-1st-v1}-\eqref{eq:chain-add} can be written in the form of~\eqref{eq:full-first}.
\end{sloppypar}

\section{Supplementary analysis}
\subsection{Example~\ref{sec:pend-slider}}
\label{sec:appendix-slider}
\begin{sloppypar}
Along with the plots of time histories of the translational displacement of the slider ($x_1$) and the angular displacement of the pendulum ($\varphi_2$) shown in Fig.~\ref{fig:slide-pend-ic1}, we present the corresponding plots of time histories of Lagrange multipliers $\mu_2$ and $\mu_3$ defined in~\eqref{eq:eom-slider}. These two multipliers represent reaction forces induced by the last two configuration constraints in~\eqref{eq:constraints-slider}. The time histories for reaction forces $\mu_2$ and $\mu_3$ with initial conditions IC1 and IC2 are shown in Fig.~\ref{fig:slide-pend-lamd}, from which we observe a close match between the results from SSM predictions and the reference results of the full system. Here the reference results are obtained from the following steps
\begin{enumerate}
    \item Perform forward simulation of Euler-Lagrange equations~\eqref{eq:pend-slide-ode};
    \item Calculate $\ddot{x}_2$ and $\ddot{y}_2$ with $x_1$ and $\varphi_2$ obtained from the above simulation;
    \item Calculate $\mu_2$ and $\mu_3$ from the third and fourth sub-equations in~\eqref{eq:eom-slider}.
\end{enumerate}
\end{sloppypar}

\begin{figure}[!ht]
	\centering
	\includegraphics[width=.45\textwidth]{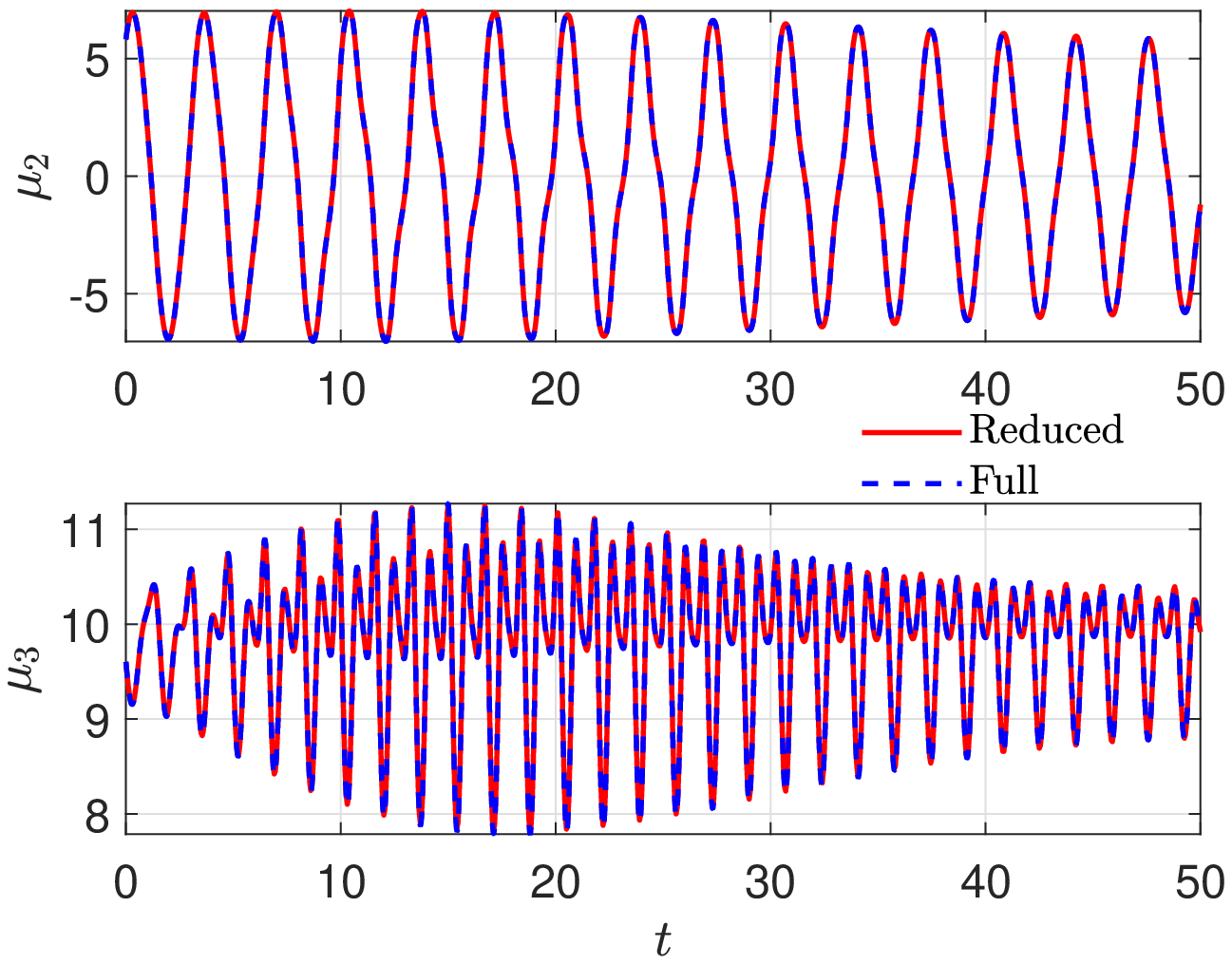}
	\includegraphics[width=.45\textwidth]{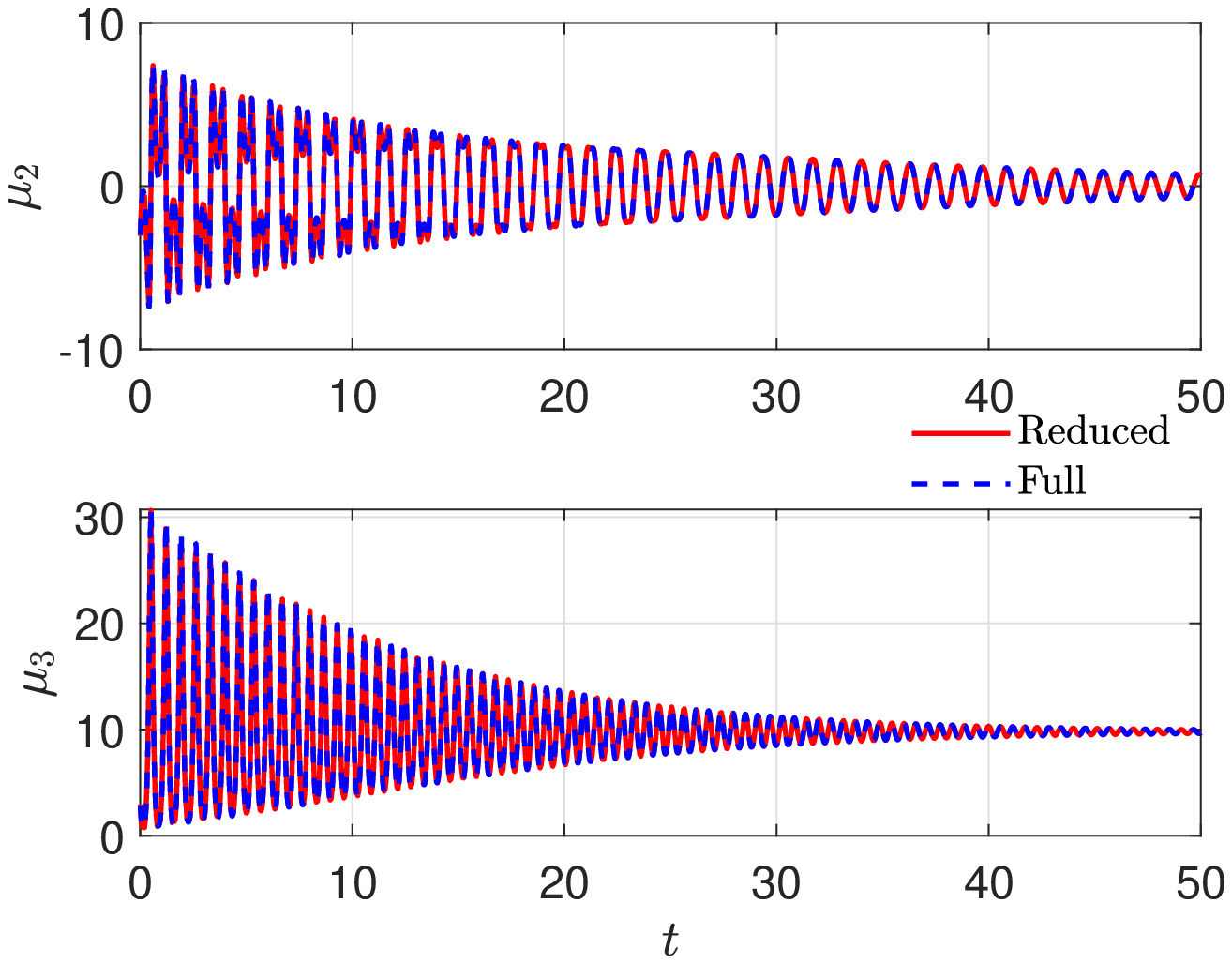}	
	\caption{Time histories for the reaction forces $(\mu_2,\mu_3)$ of the pendulum slider system~\eqref{eq:dae-pend-slider} with initial condition IC1 $(3.5e^{\mathrm{i}},3.5e^{-\mathrm{i}},0,0)$ (upper panel) and IC2 $(0,0,3.5e^{\mathrm{i}},3.5e^{-\mathrm{i}})$ (lower panel) on the SSM approximated up to $\mathcal{O}(13)$.}
	\label{fig:slide-pend-lamd}
\end{figure}

\subsection{Example~\ref{sec:example-divider}}
\label{sec:appendix-divider}
\begin{sloppypar}
We again use the invariance error measure introduced in Eq.~\eqref{eq:auto-error-measure} for validating our results and choosing an appropriate order of expansion. Similarly to the pendulum-slider example, we choose the refinement parameters $n_\alpha=10$, $n_{\vartheta}=30$ (see Eq.~\eqref{eq:samples-on-ssm}) and estimate the average invariance error over a 4-sphere of radius $\varrho$ in the parametrization space at various orders of approximation of the SSM. As shown in Fig.~\ref{fig:error-auto-divider}, the normalized error measure decreases with decreasing $\varrho$ and increasing orders, as seen previously.
We choose an error tolerance of 0.01, which is still inside the convergence domain boundary deduced from Fig.~\ref{fig:error-auto-divider}. We infer from Fig.~\ref{fig:error-auto-divider} that expansions at $\mathcal{O}(3)$, $\mathcal{O}(5)$ and $\mathcal{O}(7)$ will be required to meet the chosen error tolerance in the domains $\varrho\leq14$, $\varrho\leq30$, and $\varrho\leq50$, respectively.
\end{sloppypar}

\begin{figure}[!ht]
	\centering
	\includegraphics[width=.5\textwidth]{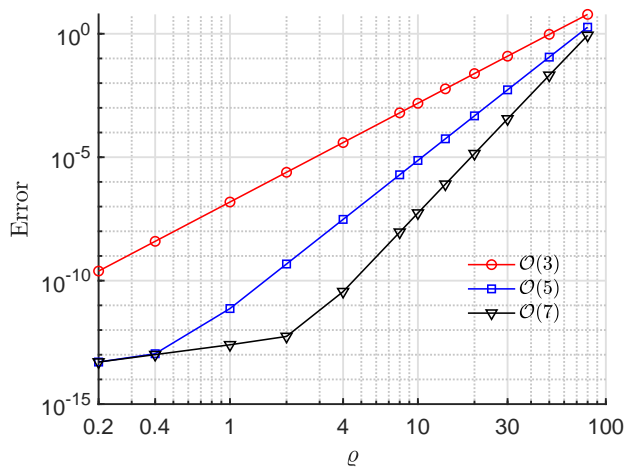}
	\caption{\small Invariance error measure~\eqref{eq:auto-error-measure} of the 4-dimensional SSM approximations of the frequency divider system~\eqref{eq:divider-diff}.}
	\label{fig:error-auto-divider}
\end{figure}

We compare the trajectories from the simulations of the  SSM-based ROM up to $\mathcal{O}(7)$ with those of the full system. Similarly to~\eqref{eq:ic-slide-pend-ssm}, we consider two sets of initial conditions as
\begin{gather}
	\mathrm{IC1}: \boldsymbol{p}_0 = (50e^{2\mathrm{i}},50e^{-2\mathrm{i}},0,0),\nonumber\\ \mathrm{IC2}: \boldsymbol{p}_0 = (0,0,50e^{2\mathrm{i}},50e^{-2\mathrm{i}}),\label{eq:ic-divider-ssm}
\end{gather}
where IC1 and IC2 are chosen along the first and the second mode on a hypersphere of radius $\varrho=50$ in the parametrization space of the SSM. 
In Fig.~\ref{fig:divider-ic12}, we plot the transverse displacement trajectories at the midpoints of the horizontal and the vertical beams initialized at IC1 and IC2.  The results obtained by SSM prediction match well with the reference results from the full system.

\begin{figure}[!ht]
\centering
\includegraphics[width=.45\textwidth]{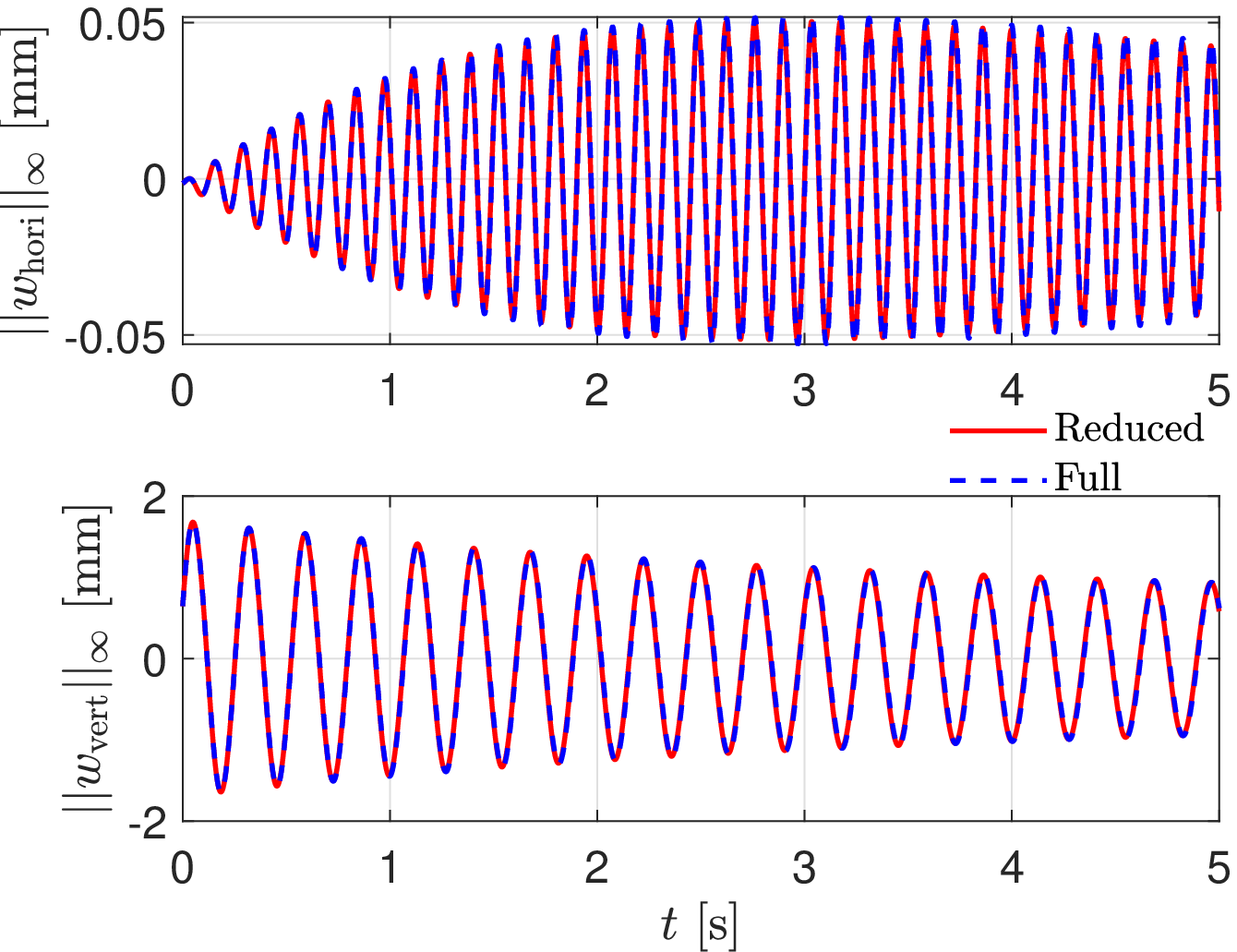}
\includegraphics[width=.45\textwidth]{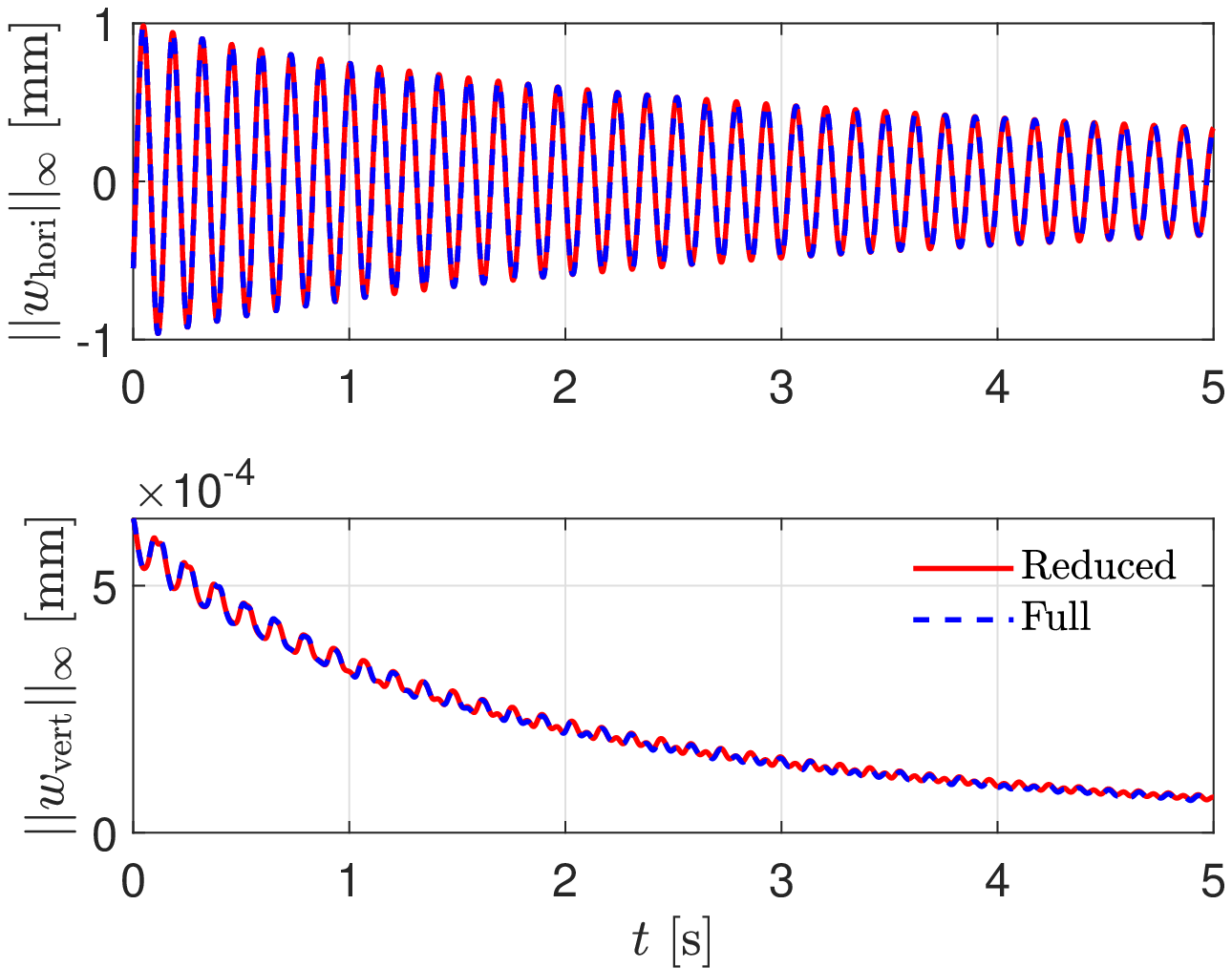}
\caption{Time histories for transverse deflections at the midpoints of the two beams with initial condition IC1 $(50e^{2\mathrm{i}},50e^{-2\mathrm{i}},0,0)$ (upper panel) and IC2 $(0,0,50e^{2\mathrm{i}},50e^{-2\mathrm{i}})$ (lower panel). Here $w_\mathrm{hori}$ and $w_\mathrm{vert}$ are transverse deflection at the midpoint the the horizontal and vertical beams of the divider.}
\label{fig:divider-ic12}
\end{figure}

In the upper panel of Fig.~\ref{fig:divider-ic12}, we observe that the vibration amplitude of the transverse displacement at the midpoint of the horizontal beam increases gradually from zero and then decays in time. In contrast, we observe a monotonic decay for the vibration amplitude of the transverse displacement at the midpoint of the vertical beam. Since the transverse displacements at the midpoints of the horizontal and vertical beams feature the vibrations of the second and the first modes (see Fig.~\ref{fig:frequency-divider-modeshape}), the time histories in the upper panel indicate a energy transfer from the first mode to the second mode. In contrast, the transverse displacement at the midpoint of the vertical beam in the right panel stays close to zero because $\rho_1\equiv0$ along the trajectory of the reduced dynamics~\eqref{eq:red-auto-divider} and then the first mode is inactive.



\bibliography{manuscript.bbl} 
\bibliographystyle{ieeetr}

\end{document}